\documentclass[bibliography=totoc]{article}
\usepackage{array}
\usepackage{tabularx}
\usepackage{hyperref}
\usepackage[english,french]{babel}
\usepackage{todonotes}
\usepackage[percent]{overpic}

\usepackage[utf8]{inputenc} 
\usepackage[T1]{fontenc}
\usepackage{amsmath}
\usepackage{amssymb}
\usepackage{mathrsfs}
\usepackage{lmodern}
\usepackage{amsfonts}
\usepackage{dsfont}
\usepackage{amssymb}
\usepackage{amsthm}
\usepackage{stmaryrd}
\usepackage{bbm}
\usepackage{xparse}
\usepackage{tikz}
\usepackage{subfiles}
\usepackage{float}
\usepackage{geometry}
\usepackage{enumerate}
\usepackage{xcolor}
\usepackage{thmtools}
\usepackage{thm-restate}
\usepackage{pgf,tikz}
\usetikzlibrary{arrows}
\usepackage{subcaption}

\hypersetup{
	colorlinks   = true, 
	urlcolor     = blue, 
	linkcolor    = blue, 
	citecolor    = red 
}

\theoremstyle{plain}
\newtheorem{thm}{Theorem}[section]
\newtheorem{lem}[thm]{Lemma}
\newtheorem{prop}[thm]{Proposition}

\theoremstyle{definition}

\newtheorem{rem}[thm]{Remark}

\numberwithin{equation}{section}

\newcommand{\Var}{\text{\normalfont Var}}
\newcommand{\conv}{\text{\normalfont conv}}

\newcommand{\tconv}{\text{\normalfont conv}}

\newcommand{\la}{\lambda}
\renewcommand{\P}{\mathbb{P}}
\newcommand{\E}{\mathbb{E}}
\newcommand{\B}{\mathbb{B}}
\newcommand{\R}{\mathbb{R}}

\newcommand{\Z}{\mathbb{Z}}
\renewcommand{\S}{\mathbb{S}}

\newcommand{\cP}{\mathcal{P}}
\newcommand{\cF}{\mathcal{F}}

\newcommand{\cT}{\mathcal{T}}
\newcommand{\cU}{\mathcal{U}}

\DeclareOldFontCommand{\sc}{\normalfont\scshape}{\mathsc}

\makeindex		    
\definecolor{zzttqq}{rgb}{0.6,0.2,0.}
\usepackage{pgf,tikz,pgfplots}
\pgfplotsset{compat=1.15}
\usepackage{mathrsfs}
\usetikzlibrary{arrows}

\DeclareMathOperator*{\lsup}{\overline{\lim}}

\newcommand\blfootnote[1]{%
  \begingroup
  \renewcommand\thefootnote{}\footnote{#1}%
  \addtocounter{footnote}{-1}%
  \endgroup
}

\title{
Limit theory for the first layers of the random convex hull peeling in the unit ball
}
\date{\today}
\author{Pierre Calka, Gauthier Quilan}

\begin{document}
\selectlanguage{english}
\maketitle
\begin{abstract}
    The convex hull peeling of a point set is obtained by taking the convex hull of the set and repeating iteratively the operation on the interior points until no point remains. The boundary of each hull is called a layer. We study the number of $k$-dimensional faces and the outer defect intrinsic volumes of the first layers of the convex hull peeling of a homogeneous Poisson point process in the unit ball whose intensity goes to infinity. More precisely we provide asymptotic limits for their expectation and variance as well as a central limit theorem. In particular, the growth rates do not depend on the layer.
\end{abstract}
\blfootnote{\textit{American Mathematical Society 2010 subject classifications.} Primary 60D05; Secondary 52A22. 60G55e}
\section{Introduction}
\subsection{Context}

Random polytopes as convex hulls of random points have been extensively studied in stochastic geometry. An overview of the subject can be found in \cite{R10} and \cite[Chapter 8.2]{SW08} for instance. Let $\cP_\la$ be a Poisson point process with intensity measure $\la \mathrm{d}x$ in a convex body $K$ of $\R^d$. The study of the asymptotic behaviour as $\la \rightarrow \infty$ of $\conv(\cP_\la)$ started with the work of R\'enyi and Sulanke in \cite{RS63,RS64}, in a binomial setting. They obtain in particular a different growth rate for the mean number of extreme points when $K$ is a smooth convex body with a ${\mathcal C}^2$-regular boundary and when $K$ is a polytope, namely polynomial for the former and logarithmic for the latter.
Since then diverse results on the number of $k$-dimensional faces and on  the defect intrinsic volumes of $\conv(\cP_\la)$ have been proved for both choices of $K$. We only consider the smooth case in this paper. Asymptotic expectations 
are shown notably in \cite{S88,R05a, B04}. In particular, it is known that the mean number of extreme points grows like $\la^{\frac{d-1}{d+1}}$ up to a multiplicative constant. First-order results have then been complemented by central limit theorems in \cite{R05b, TTW18, CSY}. Results on the variance of these quantities go from general bounds in \cite{BR10, BFV10} to explicit limits in \cite{CSY, CY1}.
More recently, concentration inequalities have been derived in \cite{GT18}. The explicit formulas obtained by \cite{B05,K21a,K21b} are worth noting among the very few non-asymptotic results available. 

The subject of this paper is a generalization of the study of the convex hull of random points to the so-called \textit{convex hull peeling}. We start by taking the convex hull of the whole process and then repeatedly take the convex hull of the points that were not extreme at the previous step until no point remains. Let us write $\conv_1(\cP_\la) := \conv(\cP_\la)$ and by induction for any $n \geq 1$, $\conv_{n+1}(\cP_\la) := \conv(\text{int}(\conv_n(\cP_\la)) \cap \cP_\la)$. The boundary of the $n$-th convex hull $\partial\conv_n(\cP_\la)$ will be called the $n$-th \textit{layer} of the convex hull peeling of $\cP_\la$. The words \textit{peeling} and \textit{layers} were chosen by analogy with the peeling of an onion, see
\begin{figure}
    \centering
    \includegraphics[scale=0.6]{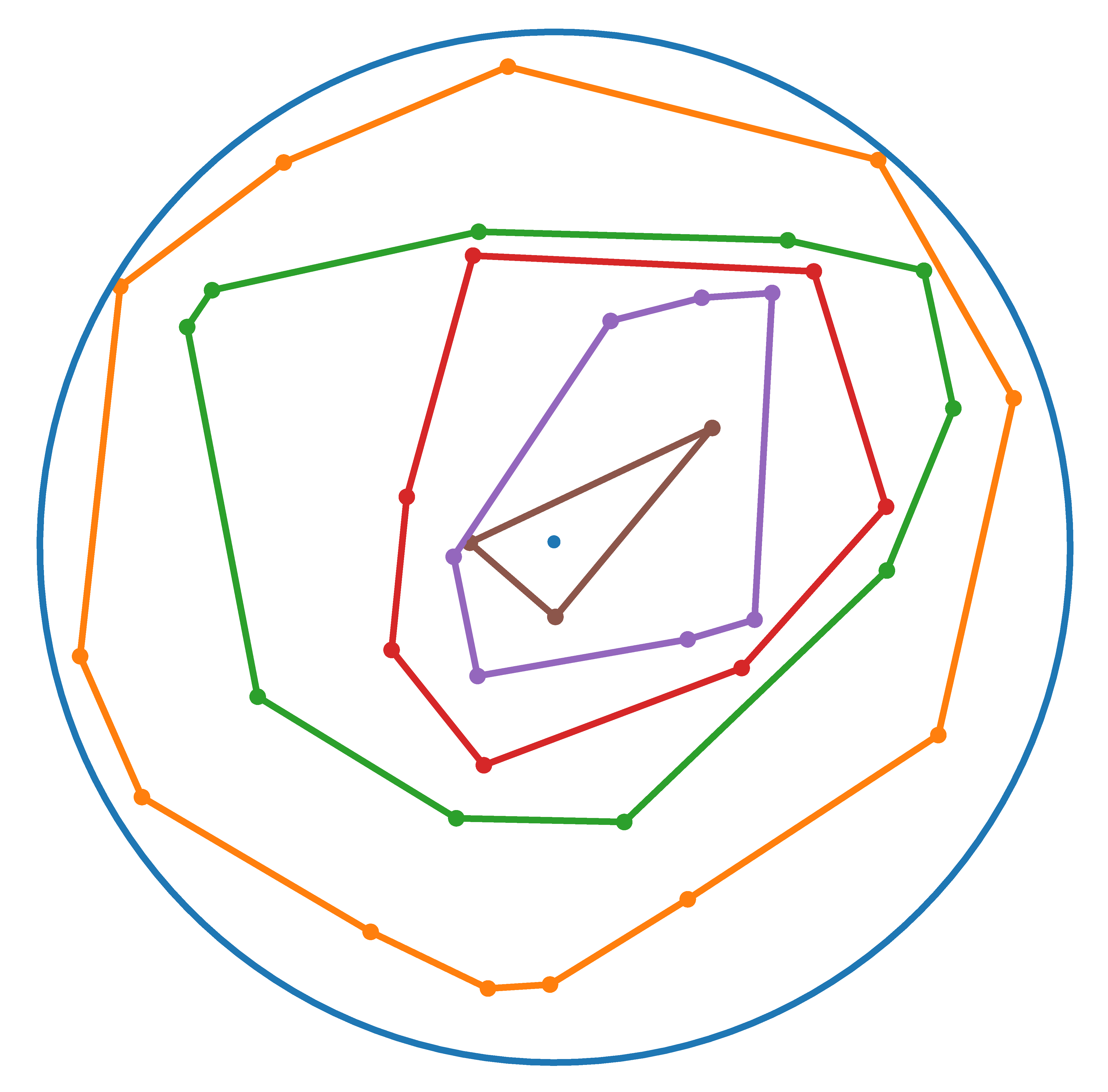}
    \caption{Example of a convex hull peeling in $\B^2$ with six layers (here the last layer has only one point).}
    \label{fig:peeling}
\end{figure}
Figure \ref{fig:peeling}.

The convex hull peeling was first introduced by Barnett in \cite{B76} as a way to order multivariate data and give a meaning to how central a point is with respect to a dataset. Indeed, the layer number of a point can be interpreted as the \textit{depth} of that point with respect to the input and we expect it to be all the larger the more central the point is. The convex hull peeling has then been used in robust statistics and outlier detection, see \cite{DG92, HA04, RS04}. It fits into a list of classical techniques for ordering multivariate data, including half-space   depth, simplicial depth or zonoid depth, see e.g. \cite{C10} for an overview on these techniques. 

  However it seems that very few theoretical results exist on the convex hull peeling of a random sample. For instance, the survey \cite{R10} devotes a small section on convex hull peeling but does not provide any reference and states that in continuation with the available asymptotic results for the convex hull, \textit{investigations concerning expectations and deviation inequalities for [the subsequent layers of the convex hull peeling] are unknown}. To the best of our knowledge, there are mostly two papers which deal with the asymptotic properties of the convex hull peeling of random points. The first one due to Dalal \cite{D04} states that the mean total number of layers of the convex hull peeling of $n$ i.i.d. uniform points in any bounded region of $\R^d$ is lower and upper bounded by multiples of $n^{\frac{2}{d+1}}$. Very recently, a breakthrough work by Calder and Smart \cite{CS20} greatly improves Dalal's estimate. Let $\cP_{\la f}$ be a Poisson point process of intensity measure $\la f(x)\mathrm{d}x$ in $K$ with $f$ a continuous and positive function on $K$. They consider the convex height function $h_\la(x)$ of any point $x$ in  $\R^d$ as the largest $n$ such that $x\in \mbox{int}(\mbox{conv}_n(\cP_{\la f}))$. Note that Dalal's work covers the particular problem of estimating the expectation of $\max h_\la$ when $f=1$. Calder and Smart show that $\la^{-\frac{2}{d+1}}h_\la$ converges uniformly in probability with an explicit exponential bound and almost surely to a multiple of a function $h$ which is the unique viscosity solution of an explicit PDE. In other words, they obtain in particular that when $\la\to\infty$, almost surely 
\begin{equation}\label{eq:CSmainresult}
\la^{-\frac{2}{d+1}}h_\la \overset{\mbox{\tiny{unif}}}{\longrightarrow} \alpha h    
\end{equation}
where $\alpha$ is a positive constant which depends only on dimension $d$ and $h$ is the unique viscosity solution of
\begin{equation*}
\left\{\begin{array}{l}\langle Dh, \mbox{cof}(-D^2h)Dh\rangle=f^2 \mbox{ in int}(K)\\
h=0 \mbox{ on } \partial K
\end{array}\right.,
\end{equation*}
$\mbox{cof}(\cdot)$ being the cofactor matrix. Denoting by $\B^d$ the unit ball of $\R^d$, we observe that \eqref{eq:CSmainresult} implies in the particular case when $K=\B^d$ and $f=1$ that the rescaled total number of layers of the peeling satisfies almost surely
\begin{equation}\label{eq:CSmax}
\la^{-\frac{2}{d+1}}\max h_\la \longrightarrow \beta_d:=\frac{(d+1)\alpha}{2d^{\frac{d-1}{d+1}}\text{\normalfont Vol}_{d-1}^{\frac{2}{d+1}}(\mathbb{S}^{d-1})}    
\end{equation}
where $\text{\normalfont Vol}_{d-1}(\mathbb{S}^{d-1})$ is the surface area of the $(d-1)$-dimensional unit sphere $\S^{d-1}$.

Because of the normalization of $h_\la$ in  \eqref{eq:CSmainresult}, this uniform convergence result can only provide information on the regime of the peeling limited to layers numbered $\la^{\frac{2}{d+1}}$ up to a multiplicative constant. The authors do not investigate any combinatorial or geometric functional of these layers. Nonetheless, they conjecture with a short heuristic argument that the number $N_{n(\la,t), 0,\la}$ of Poisson points on a layer numbered $n(\la, t) := \lfloor t \la^{\frac{2}{d+1}}\rfloor$ should satisfy a law of large numbers when $\la\to \infty$, i.e. almost surely
\begin{equation*}\label{eq:CSconjecture}
\la^{-\frac{d-1}{d+1}}N_{n(\la,t), 0,\la}\longrightarrow \int_{\{\alpha h=t\}}f^{\frac{d-1}{d+1}}    \kappa^{\frac1{d+1}}\mathrm{d}S
\end{equation*}
where $\kappa$ is the Gauss curvature of the level set $\{\alpha h=t\}$ and $\mathrm{d}S$ is the Hausdorff measure of that set.
In particular, when $K$ is the unit ball and $f=1$, the conjectured result should read, see \cite[display (1.18)]{CS20},
\begin{equation}\label{eq:CSconjectureball}
\la^{-\frac{d-1}{d+1}}N_{n(\la,t), 0,\la}\longrightarrow \frac{d+1}{2\beta_d}\left(1-\beta_d^{-1}t\right)_+^{\frac{d-1}{2}}   \end{equation}
where the constant $\beta_d$ is introduced at \eqref{eq:CSmax}.

In comparison to \cite{CS20}, our approach is different, i.e. we consider the case $K=\B^d$ and $f=1$, we choose to fix a layer numbered $n$ that does not depend on $\la$ and study the geometric properties of $\partial \conv_n(\cP_\la)$ as $\la\to\infty$. In other words, we investigate a different regime, namely the regime of the first layers in the context of uniform points in the ball. There are several reasons to do so: when applying the convex hull peeling to outlier detection, we expect the outliers to be located on the first layers of the peeling, which provides some motivation for understanding the cardinality of these particular layers. Moreover, we intend to use a global scaling transformation on the ball which has been introduced in \cite{CSY} for the study of the convex hull and which is expected to bring exhaustive information on the visible layers after rescaling, namely the first layers.

\subsection{Model}
Let $\cP_\la$ be a Poisson point process of intensity measure $\la\mathrm{d}x$ in the unit ball $\B^d$ of $\R^d$. We construct the consecutive hulls $\conv_n(\cP_\la)$, $n\ge 1$, of the peeling of $\cP_\la$. For $n\ge 1$ and $k\in\{0,\ldots,d - 1 \}$, we denote by $N_{n,k,\la}$ the number of $k$-dimensional faces of the $n$-th layer $\partial\conv_n(\cP_\la)$ and for $k\in \{1, \ldots, d\}$, by $V_{n,k,\la}$ the defect $k$-dimensional intrinsic volume of $\conv_n(\cP_\la)$, i.e.
\begin{equation}
V_{n,k,\la}=V_k(\B^d)-V_k(\conv_n(\cP_\la))    
\end{equation}
where $V_k$ stands for the $k$-th intrinsic volume, see for example \cite[p. 600]{SW08} for a definition and some properties of the intrinsic volumes.

We focus on these two families of random variables and  aim at studying their first and second-order properties.

\subsection{Main results}

For two non-negative functions $f$ and $g$, we write $f = O(g)$ if there exist a constant $C > 0$ and $\lambda_0 > 0$ such that for any $\lambda \geq \lambda_0$ we have $f(\la) \leq C g(\la)$.
Theorem \ref{thm:principalintro} below provides expectation and variance asymptotics as well as a central limit theorem for the variables $N_{n,k,\la}$. 
\begin{thm}\label{thm:principalintro}
	For any $n \geq 1$ and $k \in \lbrace 0, \ldots , d-1 \rbrace$ there exist $C_{n,k,d}, C'_{n,k,d} \in (0,\infty)$ such that
	$$\lim\limits_{\lambda \rightarrow +\infty} \lambda^{-\frac{d-1}{d+1}}{\mathbb{E}[N_{n,k,\lambda}]}
	= C_{n,k,d} \mbox{ and }
	\lim\limits_{\lambda \rightarrow +\infty} {\lambda^{-\frac{d-1}{d+1}}} {\text{\normalfont Var}[N_{n,k,\lambda}]}
	= C'_{n,k,d}.
	$$
Moreover, when $\la\to\infty$, we have
    %
    \begin{equation*}
        \sup_t \Bigg|
        \P\left(\frac{N_{n,k,\la} - \E[N_{n,k,\la}]}{\sqrt{\Var[N_{n,k,\la}]}} \le t \right)
        - \P(\mathcal{N}(0,1) \le t )\Bigg| = O\left(\la^{-\frac{d-1}{2(d+1)}} (\log \la)^{3d + 1} \right).
    \end{equation*}
\end{thm}
 In Theorem \ref{thm:principalintrovol}, we derive similar results for the variables $V_{n,k,\la}$.  
\begin{thm}\label{thm:principalintrovol}
	For any $n \geq 1$ and $k \in \lbrace 1, \ldots , d \rbrace$ there exist $C_{V,n,k,d}, C'_{V,n,k,d} \in (0,\infty)$ such that
	$$\lim\limits_{\lambda \rightarrow +\infty} \lambda^{\frac{2}{d+1}}\mathbb{E}[V_{n,k,\lambda}]
	= C_{V,n,k,d} 
	\mbox{ and }
	\lim\limits_{\lambda \rightarrow +\infty} \lambda^{\frac{d+3}{d+1}}\text{\normalfont Var}[V_{n,k,\lambda}]
	= C'_{V,n,k,d} .$$
Moreover, when $\la\to\infty$, we have
    \begin{equation*}
        \sup_t \Bigg|
        \P\left(\frac{V_{n,k,\la} - \E[V_{n,k,\la}]}{\sqrt{\Var[V_{n,k,\la}]}} \le t \right)
        - \P(\mathcal{N}(0,1) \le t )\Bigg| = O\left(\la^{-\frac{d-1}{2(d+1)}} (\log \la)^{3d + 1} \right).
    \end{equation*}
\end{thm}

The rates in Theorems \ref{thm:principalintro} and \ref{thm:principalintrovol} are identical to those obtained for the first layer, i.e. the convex hull of $\cP_\la$, as described in \cite{CSY}. In particular, the underlying limiting expectations and variances are proved to be different from zero. They have an explicit formulation in terms of a random process  derived from a homogeneous Poisson point process in the product space $\R^{d-1}\times\R_+$, see Theorems \ref{lim esperance precis} -- \ref{lim variance precis vol}. This solves the conjecture discussed in \cite{CS20} and along the lines above \eqref{eq:CSconjectureball} in the particular regime when the layer number does not depend on the size of the input. 

Our tools are those of stabilization theory that were used in \cite{CSY} to prove precise variance asymptotics for the first layer. The key idea consists in writing $N_{n,k,\la}$ and $V_{n,k,\la}$ as a sum $\sum_{x \in \cP_\la} \xi(x, \cP_\la)$ for some functional $\xi$ and proving that for a given point $x$ this functional only depends on the process in a neighbourhood of $x$. That is what we call stabilization. The importance of the stabilization can already be seen in the study of the variance of $N_{n,k,\la}$ as it implies that $\xi(x,\cP_\la)$ and $\xi(y,\cP_\la)$ are independent when $x$ and $y$ are far enough from each other, which simplifies the calculation of the variance. The stabilization of $\xi$ is in fact used much more extensively for all six results stated in Theorems \ref{thm:principalintro} and \ref{thm:principalintrovol} and constitutes the main difficulty of this paper. Indeed, the formation of each layer of the peeling requires a global knowledge of the point set and also of the history of the previously constructed layers. In particular, for a given point $x$, there is no easy local criterion for checking that $x$ is on the $n$-th layer of the peeling. In this regard, the problem is significantly different from the study of the convex hull as done in \cite{CSY}. The only characterization that we can use is incremental, see Lemma \ref{lem:be on layer n}, and this explains why the proof of stabilization is done by induction on the layer number. Incidentally, this also requires to estimate the position of each layer, see e.g. Lemma \ref{lemme hauteur max}.

The strategy of proof of the expectation and variance asymptotics in Theorem \ref{thm:principalintro}  is the following. 
\begin{itemize}
    \item Using the decomposition of each variable as a sum over $x\in\cP_\la$ of a functional $\xi(x,\cP_\la)$, we rewrite the expectation and variance of $N_{n,k,\lambda}$ as an integral thanks to Mecke's formula for Poisson point processes.
    \item Dealing with this multiple integral, we intend to use Lebesgue's dominated convergence theorem after applying a suitable change of variables inside the integral. To do so, we need to rescale the model. This leads us to introducing the notions of parabolic hull peeling in the upper half-space, see Section \ref{sec:rescaling}.      
    \item It then remains to show the convergence and domination of the integrands, rewritten as either an expectation or a covariance of a local functional of the parabolic hull peeling. This requires to show the so-called stabilization of the functionals, see Section \ref{sec:stab}. The stabilization implies in turn general moment bounds and the convergence of the integrands, see Section \ref{sec:lpandconv}.
\end{itemize}
The proof of the central limit theorem also relies on the stabilization results from Section \ref{sec:stab} as well as a Gauss approximation result in the particular setting of dependency graphs.

Finally, showing the positivity of the limiting expectations and variances represents another challenge. It requires to introduce a particular configuration where the determination of the layers and the calculation of the considered variables are natural and then to randomize this idealized configuration. This general principle has been used previously for proving the positivity of the limiting variances of $N_{0,k,\la}$ and $V_{0,k,\la}$, see e.g. \cite{R05b} and \cite{BFV10}. The construction that we do in the context of the $n$-th layer is partly inspired by \cite{D04}.

We have chosen to concentrate mainly on the variables $N_{n,k,\la}$ throughout the paper and to discuss briefly the adaptations that are needed in the case of the variables $V_{n,k,\la}$ at the end of the paper, see Section \ref{sec:intrinsic}. 

\subsection{Outline}
The paper is structured as follows.
\begin{itemize} 
    \item In Section \ref{sec:rescaling} we introduce the scaling transformation and study its effect on the point process and on the subsequent convex hulls. Incidentally we state a few basic properties on the peeling. We then define the scores as functionals of a point $x$ and of the point process such that 
    the variables $N_{n,k,\la}$ and $V_{n,k,\la}$ can be decomposed as sums of such scores. 
    We conclude with statements of more refined versions of the expectation and variance asymptotics of Theorem \ref{thm:principalintro}, with precise limiting constants.
    \item Section \ref{sec:stab} is devoted to
    proving the stabilization of the rescaled scores, i.e. that with probability exponentially close to $1$ they only depend on the process in the neighbourhood of the point considered. 
    \item In Section \ref{sec:lpandconv} we use stabilization properties shown to prove $L^p$ bounds and a convergence of in expectation of the rescaled scores.
    \item Section \ref{sec:proofs} contains the proofs of our main results.
    \item Finally, Section \ref{sec:conclusion} collects several concluding remarks about possible extensions of our work and open problems.
\end{itemize}
\section{Rescaling and scores}\label{sec:rescaling}
In this section, we introduce an \textit{ad hoc} scaling procedure originated in \cite{SY08b} and \cite{CSY}. We then study the image by that scaling transformation of the Poisson point process and of the layers of the underlying convex hull peeling. Next we prove general properties on the construction of the rescaled layers which are analogues of similar properties of the initial convex hull peeling.
Finally we introduce functionals that we call scores and we decompose $N_{n,k,\la}$ as the sum over every point of the process of these scores. This leads us to write explicit formulas for the constants in Theorem \ref{thm:principalintro} and \ref{thm:principalintrovol}, where scores are involved.

\subsection{Scaling transformation}
To describe the scaling transformation that we will use on the point process, we first recall a few definitions. We write $T_{e_d}$ for the tangent space of $\S^{d-1}$ at point $e_d = (0, 0, \ldots, 0,1)$. The exponential map $\exp_{e_d} : T_{e_d}\cong \R^{d-1} \longrightarrow \S^{d-1}$ maps a vector $v$ of $T_{e_d}$ to the vector $u$ that lies at the end of the geodesic of $\S^{d-1}$ of length
$\Vert v \Vert$ that starts at $e_d$ with direction $v$.
The function $\exp_{e_d}$ induces a one-to-one map between $B_{d-1}(\pi)$ and $\S^{d-1}\setminus \{-e_d\}$ of inverse $\exp^{-1}$ where $B_{l}(r)$ denotes the open ball centered at $0$ of radius $r$ in $\R^l$. 
This lets us define a one-to-one map $T^{(\la)}$  between $\B^d\setminus [0,-e_d]$ and $W_\la := \la^{\frac{1}{d+1}}B_{d-1}(\pi) \times [0, \la^{\frac{2}{d+1}})$ by
$$ T^{(\la)}(x) := \left(\la^{\frac{1}{d+1}}\exp^{-1}\left( \frac{x}{\Vert x\Vert}\right), \la^{\frac{2}{d+1}}(1 - \Vert x \Vert) \right) $$
for all $x \in \B^d \setminus [0,-e_d]$.
In general we will denote by $w = (v,h)$ with $v\in \R^{d-1}$ and $h\in \R_+$ a generic point in $W_\la$. 
The transformation $T^{(\la)}$ was already used in \cite{CSY} to obtain variance asymptotics of functionals of the convex hull of $\cP_\la$ that include the number of $k$-faces (short for \textit{$k$-dimensional faces}) and the $k$-th intrinsic volume. This transformation enjoys two important properties. First, unit volume subsets of $W_\la$ near
the hyperplane $\R^ {d-1}\times \{0\}$ contain $\Theta(1)$ rescaled points and actually, we can show that the limit point process is Poisson and has intensity $1$, see Lemma \ref{convergence rescaled}. Secondly, the transformation preserves the parabolic shape of both the defect radius-vector function and the defect support function of the random polytope $\mbox{conv}(\cP_\la)$, as described in \cite[p. 53--54]{CSY}. 
These properties have been crucial in the proofs of the results contained in \cite{CSY} on the convex hull $\mbox{conv}(\cP_\la)$. It turns out that $T^{(\la)}$ plays a similar role for the first $n$ layers as long as we take a fixed $n$  that does not vary with $\la$. Indeed, at the limit, the convex hull is mapped by $T^{(\la)}$ to what we call the parabolic hull of the limit rescaled process. In the same way it maps the convex hull peeling to the analogue of the peeling procedure in the parabolic picture that we name \textit{parabolic hull peeling} of the limit rescaled process. We give more details below, after describing the effect of $T^{(\la)}$ on $\cP_\la$.

Our scaling transformation maps the Poisson point process $\cP_\la$ to a Poisson point process on $W_\la$ that we denote by $\cP^{(\la)}$. Its intensity has a density with respect to the Lebesgue measure given by 
\begin{equation}\label{eq:intensity process}
(v,h) \mapsto \frac{ \sin^{d-2}\left(\la^{-\frac{1}{d+1}}\Vert v \Vert \right)}{\Vert\la^{-\frac{1}{d+1}}  v \Vert^{d-2}} \left( 1 - \la^{-\frac{2}{d+1} }h \right),
\end{equation}
see \cite[p. 57]{CSY} for the computation. As proved in \cite[p. 71]{CSY}, this point process converges in distribution to a homogeneous Poisson point process of intensity one on $\R^{d-1}\times \R_+$, that we denote by $\cP$ or $\cP^{(\infty)}$.
\begin{lem}\label{convergence rescaled}
	We have $\lim\limits_{\lambda \rightarrow +\infty} \mathcal{P}^{(\lambda)}\ = \mathcal{P} $ in distribution.
\end{lem}
Next we recall from \cite{CSY} the effect of the rescaling on the spherical caps in the ball. This will then allow us to deduce the images of the consecutive layers by the rescaling. 
Any spherical cap in the unit ball of $\R^d$ can be written
\begin{equation}\label{eq:defcap}
\text{cap}(x_0) := \{ x \in \B^d : \langle x, \frac{x_0}{\Vert x_0 \Vert} \rangle  > \Vert x_0 \Vert \},\quad x_0 \in \B^d.     
\end{equation}
One can see that $\text{cap}(x_0)$ is the cap orthogonal to $x_0$ at distance $\Vert x_0 \Vert$ of the origin. Let us write $(v_0, h_0) := T^{(\la)}(x_0)$. The cap $\text{cap}(x_0)$ is sent by $T^{(\la)}$ to a so-called \textit{downward quasi-paraboloid} $[\Pi^{\downarrow}]^{(\la)}(v_0, h_0)$. Furthermore, the quasi-paraboloids $[\Pi^{\downarrow}]^{(\la)}(v_0, h_0)$ converge to a paraboloid $$[\Pi^\downarrow(v_0, h_0)]^{(\infty)}(v_0,h_0) = \Pi^\downarrow(v_0, h_0) := \{ (v,h) \in \R^{d-1}\times \R_+ : h < h_0 - \frac{\Vert v - v_0 \Vert^2}{2} \}.$$ These results are made precise in Lemma \ref{lem:image cap} below, whose proof can be found in \cite[p. 72-73]{CSY} or in \cite[Lemma 3.1]{CY2} up to a a small adaptation.
Note that in this convergence result, with a slight abuse, we use the notation $\partial[\Pi^{\downarrow}]^{(\la)}$ (resp. $\partial\Pi^{\downarrow}$) for the function  from $\R^{d-1}$ to $\R_+$ whose graph is the boundary of the set $[\Pi^{\downarrow}]^{(\la)}$ (resp. $\Pi^{\downarrow}$).

Beforehand, 
we need to introduce useful notation for several types of cylinders that are used in the rest of the paper. For any $v \in \R^{d-1}$ and $r > 0$, $C_v(r)$ denotes the vertical cylinder $B_{d-1}(v,r) \times [0, \infty)$ with the convention $C(r)=C_0(r)$. 
We also define the truncated cylinders
    $C_v^{\geq t}(r) := C_v(r) \cap \{ (v',h') \in \R^d : h' \geq t \}$,
    $C_v^{\leq t}(r) := C_v(r) \cap \{ (v',h') \in \R^d : h' \leq t \}$ and  
    $C_v^I(t) := C_v(r) \cap \{ (v',h') \in \R^d : h' \in I \}$ for any $t>0$ and any interval $I\subset \R_+$.
\begin{lem}\label{lem:image cap}
    Let $\la > 0$ and $x_0 \in \B^d$. We write $(v_0, h_0) = T^{(\la)}(x_0).$ Then we have
        \item \begin{align}\label{eq:paraboleverslebas}
            T^{(\la)}(\emph{cap}(x_0)) &= [\Pi^{\downarrow}]^{(\la)} (v_0, h_0)\\&:= \{ (v,h) \in W_\la : h < \la^{\frac{2}{d+1}} \left( 1 - \frac{1 - \la^{-\frac{2}{d+1}}h_0}{\cos(e_\la(v,v_0))} \right)\}\notag
        \end{align}
        where $e_\la(v, v_0) := d_{\S^{d-1}}\left(\exp_{d-1}(\la^{-\frac{1}{d+1}}v), \exp_{d-1}(\la^{-\frac{1}{d+1}}v_0)\right)$ for $v,v_0 \in \R^{d-1}$.

        Additionally, for any $L\ge 1$, we have the following convergence result.
       \begin{equation}\label{eq:cvusurtoutcompact}
       \lim\limits_{\la\rightarrow \infty}\partial[\Pi^{\downarrow}]^{(\la)}(v_0,h_0) \cap C_{(v_0,h_0)}(L) = \partial\Pi^{\downarrow}(v_0, h_0) \cap C_{(v_0,h_0)}(L)
       \end{equation}
       for the uniform convergence.
        
\end{lem}

\begin{rem}
As in \cite{CSY}, when $\la\in (0,\infty)$, we can introduce a dual set
\begin{align}\label{eq:paraboleverslehaut}
[\Pi^{\uparrow}]^{(\la)}(v_0,h_0)&:=T^{(\la)}(\partial B(\frac{x_0}{2},\frac{\|x_0\|}{2}))\\&=\{ (v,h) \in W_\la : h > \la^{\frac{2}{d+1}} ( 1 -(1 - \la^{-\frac{2}{d+1}}h_0) \cos(e_\la(v,v_0)) )\}\notag
\end{align}
that we call the upward quasi-paraboloid with apex $(v_0,h_0)$. Similarly, we define
\begin{align}\label{eq:paraboleverslehaut2}
[\Pi^{\uparrow}]^{(\infty)}(v_0, h_0) =\Pi^{\uparrow}(v_0,h_0)&:=\{(v,h)\in \R^{d-1}\times\R_+: h>h_0+\frac{\|v-v_0\|^2}{2}\}.
\end{align}
In particular, for any $\la\in (0,\infty]$ and any $(v_0,h_0),(v_1,h_1)\in W_\la$,
$$((v_1,h_1)\in [\Pi^{\uparrow}]^{(\la)}(v_0,h_0))\Longleftrightarrow ((v_0,h_0)\in [\Pi^{\downarrow}]^{(\la)}(v_1,h_1)).$$
This fact will be used on many occasions in the forthcoming proofs.
\end{rem}
Recalling that for a locally finite point set $X$, $$\conv(X) = \bigcup_{\substack{H^+ \text{ half-space } \\ X \cap H^+ = \varnothing}} (H^+)^c,$$
we are led by Lemma \ref{lem:image cap} to the following definition, that corresponds to the analogue of the convex hull in the (quasi-)parabolic setting, where the role of the half-spaces is played by downward quasi-paraboloids or paraboloids. For $\la \in (0,\infty]$ and a locally finite point set $Y$ in $W_\la$, we write
 $$\Phi^{(\la)}(Y) := \bigcup_{\substack{w \in W_\la \\ Y \cap [\Pi^\downarrow ]^{(\la)}(w) = \varnothing}} [\Pi^\downarrow ]^{(\la)}(w)^c$$
 that we call the quasi-parabolic hull of $Y$, or parabolic hull when $\la = \infty$. We will generally write $\Phi$ instead of $\Phi^{(\infty)}$ for sake of simplicity.
Thanks to Lemma \ref{lem:image cap}, we obtain that $T^{(\la)}$ maps the convex hull of a point set to the quasi-parabolic hull of the image of this point set, i.e.
 $$T^{(\la)}(\conv(X)) = \Phi^{(\la)}(T^{(\la)}(X)),$$
provided that $\conv(X)$ contains the origin in its interior. In particular, when $X=\cP_\la$, Wendel's formula \cite{W62}
shows that the event $\{0\in \text{int}(\conv(\cP_\la))\}$  has a probability going to $1$ exponentially fast when $\la\to\infty$ and we implicitly condition on that event when working with $T^{(\la)}$.

We call \textit{extreme points} of $Y$ the points of $Y\cap \partial\Phi_1^{(\la)}(Y)$. They are naturally images by $T^{(\la)}$ of the extreme points of the convex hull of $[T^{(\la)}]^{-1}(Y)$.
 
 When $\la\to\infty$, the intersection $\partial \Phi^{(\la)}(\cP^{(\la)})\cap C(R)$, $R>0$, converges in distribution to $\partial \Phi(\cP) \cap C(R)$ where each quasi-parabolic hull is seen as a continuous function over $C(R)$ and the set of continuous functions on $C(R)$ is endowed with the topology of the uniform convergence, see \cite[Theorem 4.1]{CSY}. 
 
Let us now investigate the action of the transform $T^{(\la)}$ on the convex hull peeling procedure. 
Quite naturally, it maps the convex hull peeling to a quasi-parabolic hull peeling that will converge in some sense to a parabolic hull peeling. We define the hulls of the quasi-parabolic and parabolic hull peeling recursively with the following formula: for all $n \geq 2$, $\la \in (0,\infty]$ and a locally finite point set $Y\subset \R^{d-1}\times\R_+$, we set
  $$\Phi^{(\la)}_n(Y) = \Phi^{(\la)}(Y \cap \text{int}(\Phi^{(\la)}_{n-1}(Y))). $$
 When $\la=\infty$, we speak of parabolic hull peeling and we write $\Phi_n$ instead of $\Phi_n^{(\infty)}$.
 In the same way as for the the first layer of the convex hull peeling, the subsequent ones are mapped by $T^{(\la)}$ to the corresponding parabolic hulls, i.e.
  \begin{equation}\label{eq:changement echelle peeling}
  T^{(\la)}(\conv_n(X)) = \Phi^{(\la)}_n(T^{(\la)}(X)),
  \end{equation}
  provided that the origin lies in the interior of $\conv_n(X)$. It is a direct consequence of \cite[Theorem 1.2]{CS20}, see also \eqref{eq:CSmainresult}, and \cite[Equation (1.10) for $x=0$]{CS20} that the event $\{0\in \text{int}(\conv_n(\cP_\la))\}$ has a probability going to $1$ exponentially fast when $\la\to\infty$. Henceforth, when dealing with $\conv_n$, we implicitly condition on that particular event.
  
   
    We call the sets $\partial\Phi_n^{(\la)}(Y)$ the \textit{layers} of the quasi-parabolic or parabolic hull peeling of $Y$.
   
    For any $n\ge 1$, the set $\Phi_n^{(\la)}(Y)$ (resp. $\Phi_n(Y)$) is the complement of a union of down quasi-paraboloids (resp. paraboloids). As the quasi-paraboloids converge to paraboloids, see Lemma \ref{lem:image cap}, and $W_\la$ goes to $\R^{d-1}\times\R_+$ as $\la$ goes to infinity, we can extend \cite[Theorem 4.1]{CSY} to the convergence in distribution of the subsequent rescaled layers of the original convex peeling to the corresponding layers of the parabolic peeling associated with the limit Poisson point process. As a side result, we also obtain the convergence of the point process of points of $\cP^{(\la)}$ on the $n$-th layer. This is summarized in Proposition \ref{prop:convrescaledlayers} below. 
    \begin{prop}\label{prop:convrescaledlayers}
    Let $R>0$ and $n\ge 1$. When $\la\to\infty$, we have that
    $$\partial\Phi_n^{(\la)}(\cP^{(\la)}) \overset{{\mathcal L}}{\to}
    \partial\Phi_n(\cP) $$ 
    where the set of continuous functions over $\R^{d-1}$ is endowed with the topology of the uniform convergence on every compact set. Moreover, 
    $$\cP^{(\la)}\cap \partial\Phi_n^{(\la)}(\cP^{(\la)})
\overset{{\mathcal L}}{\to}\cP\cap \partial\Phi_n(\cP).$$
    \end{prop}
    In other words, Proposition \ref{prop:convrescaledlayers} explains to what extent the parabolic hull peeling is the rescaled limiting model of the convex hull peeling in the ball. Since Proposition \ref{prop:convrescaledlayers} is a natural analogue of the results stated and proved in \cite[Theorem 4.1]{CSY} and \cite[Theorem 1.1]{CY2} for the (quasi)-parabolic hull process, its proof is omitted. 
    Let us note that when $\la=\infty$, the parabolic hull peeling of $\cP$, as seen in
    Figure \ref{fig:parabolic_peeling}, is also a crucial tool of \cite{CS20} under the name of \textit{semiconvex peeling}.
    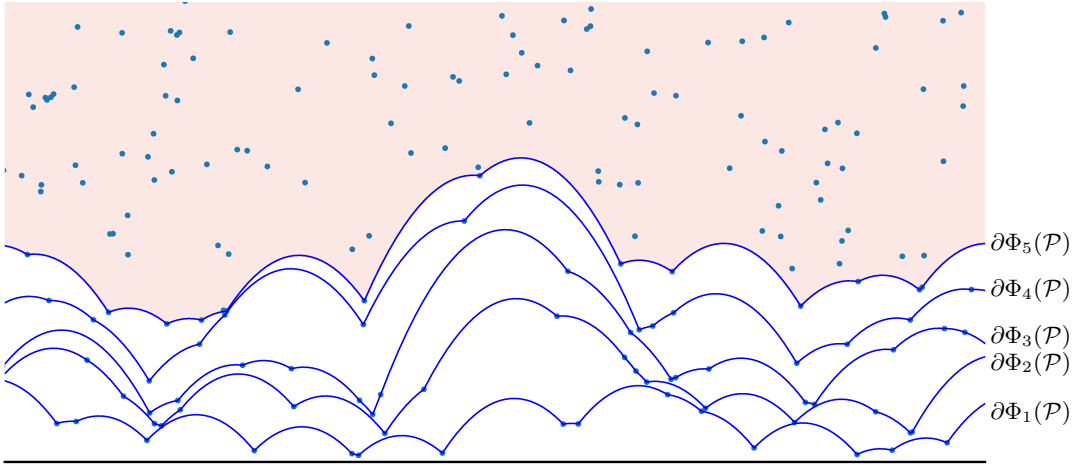
\begin{figure}
        \flushleft
        \begin{overpic}[width=0.9\textwidth]{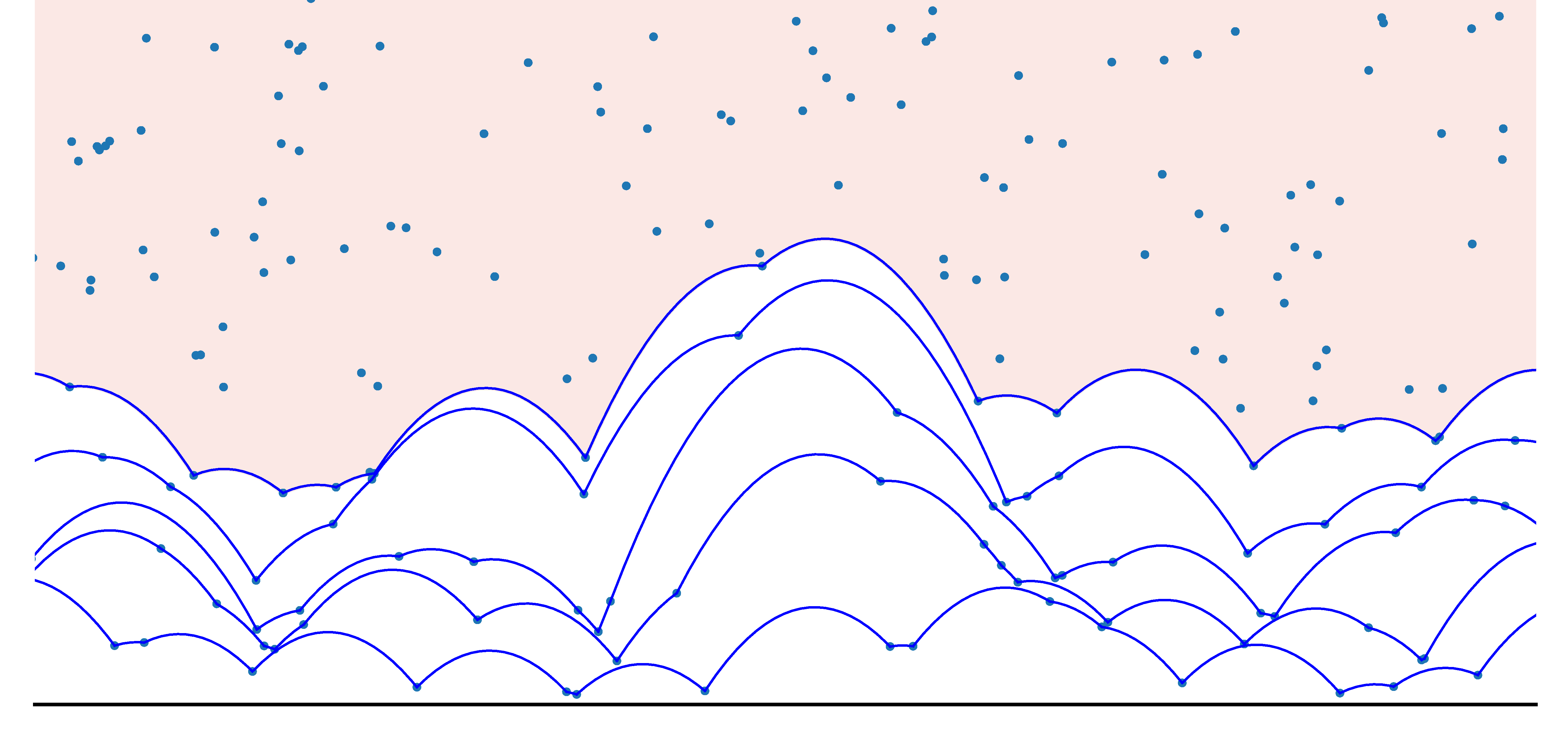}
        \put (98.5,7) {\small$\displaystyle \partial \Phi_1(\cP)$}
        \put (98.5,11.8) {\small$\displaystyle \partial \Phi_2(\cP)$}
        \put (98.5,14.4) {\small$\displaystyle \partial \Phi_3(\cP)$}
        \put (98.5,19) {\small$\displaystyle \partial \Phi_4(\cP)$}
        \put (98.5,23.5) {\small$\displaystyle \partial \Phi_5(\cP)$}
        \end{overpic}
     \caption{First layers of a parabolic hull peeling, with the interior of $\Phi_5(\cP)$ in pink.}
     \label{fig:parabolic_peeling}
    \end{figure}
    \subsection{Properties of the rescaled layers}
     For any $w\in \R^{d-1}\times \R_+$, we introduce the number
    \begin{equation}\label{eq:defnumerocouche}
    \ell^{(\la)}(w,Y)=n \mbox{ such that $w\in \partial \Phi^{(\la)}_n(Y\cup\{w\})$}.  
    \end{equation}
    In particular, for any $x\in \B^d$,
    $\ell^{(\la)}(T^{(\la)}(x),T^{(\la)}(X))$ is the number of the layer of $x$ in the initial convex hull peeling of $X \cup \{x \}$.
    We now aim at proving an explicit criterion for determining $\ell^{(\la)}(w,Y)$, see Lemma \ref{lem:be on layer n}, and the monotonicity of $\ell^{(\la)}(w,Y)$ with respect to $Y$, see Lemma \ref{lem:dalalreecrit}. Both of these properties could be stated for the initial convex hull peeling but we will only use the rescaled versions below.
    
    Let us recall, see e.g. \cite[pages 66-67]{CSY}, that a point $w$ in $\R^{d-1}\times\R_+$ is extreme if and only if there exists $(v_1, h_1) \in \partial [\Pi^\uparrow]^{(\la)}(w)$ such that
    $[\Pi^\downarrow]^{(\la)}(v_1,h_1) \cap Y = \varnothing$.
    
    Lemma \ref{lem:be on layer n} provides a geometric interpretation of the function $\ell^{(\la)}$ that extends the result above and that we will use frequently -- and sometimes implicitly -- in the rest of the paper, see Figure \ref{fig:criterion} for an illustration of this lemma.
    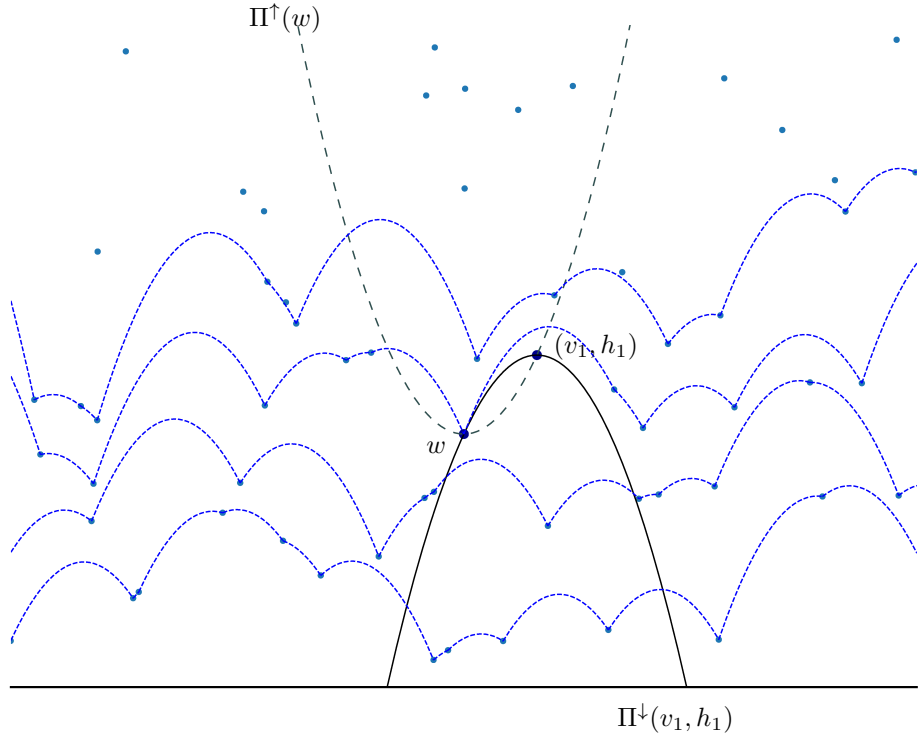
\begin{figure}
        \centering
        \begin{overpic}[width=0.8\textwidth]{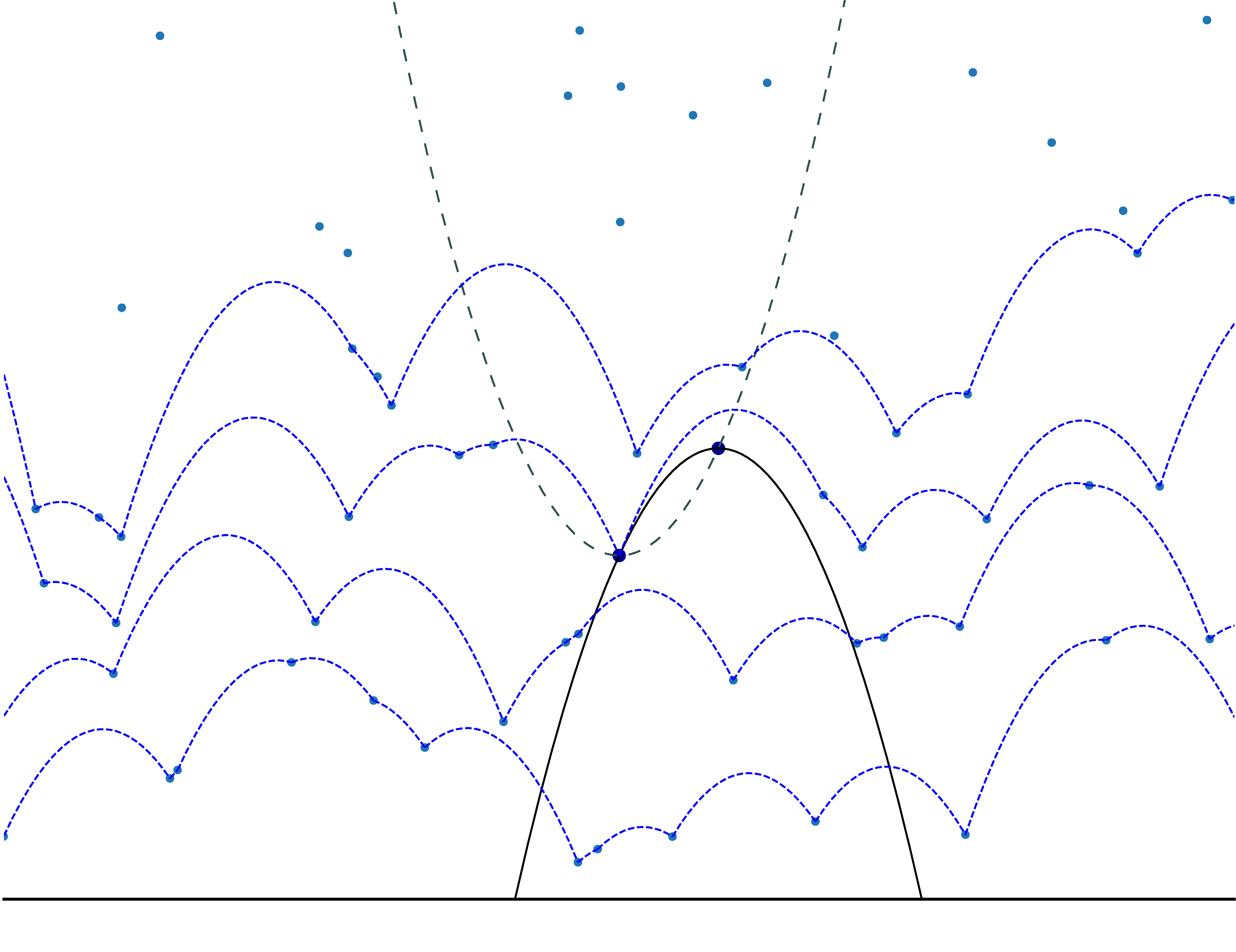}
 \put (46,30) {$\displaystyle w$}
 \put (60, 41) {$ \displaystyle(v_1, h_1)$}
 \put (67,0) {$ \Pi^\downarrow(v_1,h_1)$}
 \put (26.5,77) {$ \Pi^\uparrow(w)$}

\end{overpic}
        \caption{Illustration of the criteria (i) and (ii) of Lemma \ref{lem:be on layer n} in the case $\la=\infty$ and $l^{(\infty)}(w,Y) = 3$. As in (ii), the  paraboloid
        $\Pi^\downarrow(v_1,h_1)$ contains $w$ on its boundary and only contains points on layers at most $2$. As in (i), any translate of $\Pi^\downarrow(v_1,h_1)$  containing $w$ in its boundary (i.e. with $(v_1, h_1) \in \partial \Pi^\uparrow(w)$) contains at least one point of layer at least $2$.}
        \label{fig:criterion}
    \end{figure}
    \begin{lem}\label{lem:be on layer n}
        Let $Y$ be a locally finite subset of $\R^{d-1}\times\R_+$, $w \in Y$, $\lambda\in (0,\infty]$ and $n\ge 1$. Then we have the two following equivalences.
        
        \noindent (i) 
        \emph{(}$\ell^{(\la)}(w,Y)\ge n$\emph{)} $\Longleftrightarrow$ \emph{(}$\forall\,(v_1, h_1) \in \partial[\Pi^\uparrow]^{(\la)}(w):
    Y\cap [\Pi^\downarrow]^{(\la)}(v_1,h_1) \not\subseteq  \cup_{i=1}^{n-2} \partial [\Phi_i]^{(\la)}(Y)$\emph{)}.\\~\\
        (ii) \emph{(}$\ell^{(\la)}(w,Y)\le n$\emph{)} $\Longleftrightarrow$ \emph{(}$\exists\;(v_1, h_1) \in \partial[\Pi^\uparrow]^{(\la)}(w): 
        Y\cap [\Pi^\downarrow]^{(\la)}(v_1,h_1) \subseteq \cup_{i=1}^{n-1} \partial
        [\Phi_i]^{(\la)}(Y)$\emph{)}.
    \end{lem}
    \begin{proof}
    \noindent (i) Let us assume that $\ell^{(\la)}(w,Y) \geq n$ and take $(v_1,h_1) \in \partial[\Pi^\uparrow]^{(\la)}(w).$ If we had $[\Pi^\downarrow]^{(\la)}(v_1,h_1) \cap Y \subseteq  \cup_{i=1}^{n-2} \partial [\Phi_i]^{(\la)}(Y)$, this quasi-paraboloid would not intersect $Y$ after removing the first $(n-2)$ layers, meaning that $w$ would be at most of layer $(n-1)$ so the first implication holds.
    
    Conversely, let us assume that $[\Pi^\downarrow]^{(\la)}(v_1,h_1) \cap Y \not\subseteq  \cup_{i=1}^{n-2} \partial [\Phi_i]^{(\la)}(Y)$. Then after removing the first $(n-2)$ layers, any down quasi-paraboloid whose boundary contains $w$ has to meet $Y$. This implies that $w$ is not extreme after removing the first $(n-2) $ layers and thus $\ell^{(\la)}(w,Y) \geq n$.
    \\~\\
    \noindent (ii) If $\ell^{(\la)}(w,Y) = m \le n$, $w$ is extreme when we remove the first $(m-1)$ layers. Thus let $\;(v_1, h_1) \in \partial[\Pi^\uparrow]^{(\la)}(w)$ be such that $[\Pi^\downarrow]^{(\la)}(v_1, h_1)$ does not contain any point of $Y$ after removing the first $(m-1)$ layers. This implies that
    $$ [\Pi^\downarrow]^{(\la)}(v_1,h_1) \cap Y
        \subseteq \cup_{i=1}^{m-1} \partial
        [\Phi_i]^{(\la)}(Y)\subseteq \cup_{i=1}^{n-1} \partial
        [\Phi_i]^{(\la)}(Y).$$
        
        Conversely, if we assume that $$\exists\;(v_1, h_1) \in \partial[\Pi^\uparrow]^{(\la)}(w): 
        [\Pi^\downarrow]^{(\la)}(v_1,h_1) \cap Y \subseteq \cup_{i=1}^{n-1} \partial
        [\Phi_i]^{(\la)}(Y),$$
        either $w$ belongs to the first $(n-1)$ layers or it is extreme once the first $(n-1)$ layers are removed. Thus $\ell^{(\la)}(w, Y)\leq n$.
    \end{proof}


Lemma \ref{lem:dalalreecrit} below, which is of frequent use in our proofs, shows that the variables $\ell^{(\la)}(w,X)$ are increasing with respect to the set $X$. 
It slightly rephrases \cite[Lemma 3.1]{D04} in the context of the parabolic hull peeling and \cite[Lemma 2.1]{CS20}. For sake of completeness, we include a short proof below.
\begin{lem}\label{lem:dalalreecrit}
For $\la \in (0,\infty]$, if $X\subset Y\subset W_\la$, we have for every $w\in W_\la$, $\ell^{(\la)}(w,X)\le \ell^{(\la)}(w,Y).$    
\end{lem}
\begin{proof}
We prove the result by induction on $n= \ell^{(\la)}(w,Y).$ When $n=1$, $w$ is extreme for the point set $Y\cup\{w\}$ so is also extreme for the smaller point set $X\cup\{w\}$. When $n> 1$, by Lemma \ref{lem:be on layer n} (ii), $w$ lies on the boundary of a down quasi-paraboloid such that each point $w'$ of $Y$  in its interior satisfies $\ell^{(\la)}(w',Y\cup\{w\})\le (n-1)$. When $w'\in X$, the induction hypothesis applied to $w'$ and the point sets $X\cup\{w\}$ and $Y\cup\{w\}$ shows that $\ell^{(\la)}(w',X\cup\{w\})\le \ell^{(\la)}(w',Y\cup\{w\})\le (n-1)$. Consequently, using again Lemma \ref{lem:be on layer n} (ii), we obtain that $\ell^{(\la)}(w,X)\le n.$ This completes the proof.

\end{proof}
\begin{rem}
In fact, when $Y\setminus X$ is finite, we can show by arguments similar to the proof of \cite[Lemma 3.1]{D04} that $\ell^{(\la)}(w,Y)\le \ell^{(\la)}(w,X)+\#(Y\setminus X)$ where $\#(\cdot)$ denotes the cardinality, see Figure \ref{fig:increasing_layer}.
\end{rem}
\begin{figure}
    \centering
    \begin{overpic}[width=0.95\textwidth,trim={1cm 0cm 1cm 0},clip]{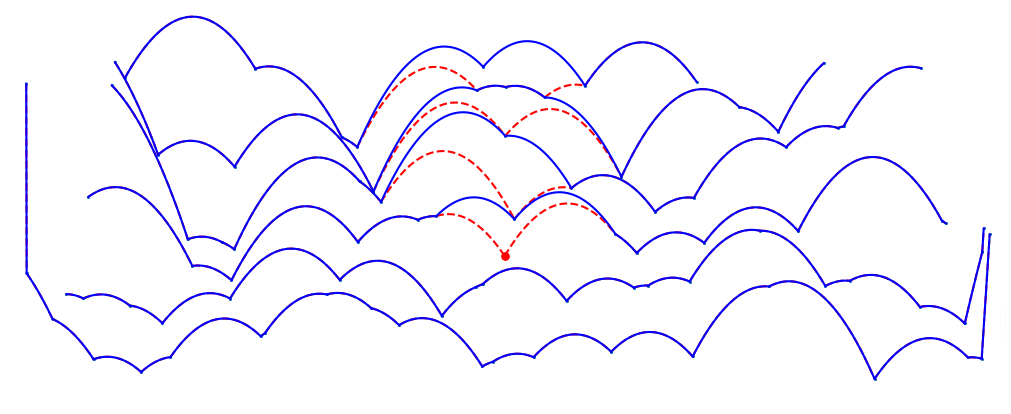}
    \put(0,3.3){\color{black}\line(100,0){100}}
\end{overpic}
    \caption{The effect of adding a (red) point to the parabolic hull peeling. The red dashed lines are the changes applied to the layers.}
    \label{fig:increasing_layer}
\end{figure}

\subsection{Scores and correlation functions}

In this subsection we associate to each point of $\B^d$ (resp. $W_\la$) a random variable depending on that point and on the Poisson point process which we call \textit{score}.
We start by defining the score of a point in the initial convex hull peeling before rescaling, i.e. for every $x\in \B^d$, $n\ge 1$ and $k\in \{0,\ldots,d-1\}$, we introduce the r.v.
$$\xi_{n,k}(x,X):=  \left\{\begin{array}{ll}\frac{1}{k+1} \# \cF_{n,k}(x,X) &\mbox{ if $x\in \partial \tconv_{n}(X\cup\{x\})$}\\0 &\mbox{ otherwise}\end{array}\right.$$
where $\cF_{n,k}(x, X)$ is the set of all $k$-faces containing $x$ of 
$\partial\tconv_n(X \cup \{x\})$. 
The factor $\frac{1}{k+1}$ is needed to take into account the fact that the faces are counted multiple times since a $k$-face contains a.s. $(k+1)$ points of $\cP_{\la}$. In particular, we get the identity
\begin{equation}\label{eq:decompavantrescaling}
N_{n,k,\la}=\sum_{x\in \cP_\la}\xi_{n,k}(x,\cP_\la).    
\end{equation}

We now extend this notion of score to the rescaled model. Let $\la \in (0,\infty]$, $Y$ be a locally finite subset of $W_\la$, $w\in W_\la$, $n\ge 1$ and $k\in \{0,\ldots,d-1\}$. \label{page:defFnk} We denote by $\cF_{n,k}^{(\la)}(w, Y)$ the set of $k$-faces of $\partial \Phi_n^{(\la)}(Y\cup\{w\})$ containing $w$, i.e. the image by $T^{(\la)}$ of the set of $k$-faces of $\mbox{conv}_n([T^{(\la)}]^{-1}(Y\cup\{w\}))$ containing $[T^{(\la)}]^{-1}(w)$ when $\la < \infty$. When $\la = \infty$, $\cF^{(\la)}_{n,k}(w, Y)$ is the set of $k$-dimensional parabolic faces of $\Phi_{n,k}^{(  \la)}(Y \cup \{w\})$, as defined in \cite[p. 65--66]{CSY}, containing $w$.
For any fixed $\la \in (0,\infty],$ we define the score
\begin{equation}\label{eq:defxink}
\xi^{(\la)}_{n,k}(w,Y) 
:= \left\{\begin{array}{ll}\frac{1}{k+1} \#\cF^{(\la)}_{n,k}(w,Y) &\mbox{ if } w\in \partial \Phi^{(\la)}_{n}(Y\cup\{w\})\\0 &\mbox{ otherwise}\end{array}\right..
\end{equation}
	We deduce from \eqref{eq:defxink} and \eqref{eq:changement echelle peeling} that for every $w \in W_\la$, 
	\begin{equation}\label{eq:egalie scores}
	\xi_{n,k}^{(\la)}(w, \cP^{(\la)}) = \xi_{n,k}\left([T^{(\la)}]^{-1}(w), \cP_\la\right).
	\end{equation}
Note that the r.v. $\xi_{n,k}^{(\la)}$ are calibrated such that $\sum_{w\in \cP^{(\la)}}\xi_{n,k}^{(\la)}(w,\cP^{(\la)})$ is a.s. the total number of $k$-faces of $\partial \Phi_n^{(\la)}(\cP^{(\la)})$. 

We then introduce the two-point correlation function which is crucial for deriving the limiting variance. 
For any $\la\in (0,\infty]$ let us write
\begin{align}\label{eq:defcnkla}
c_{n,k}^{(\lambda)}((0, h_0), (v_1, h_1))&:=\E[\xi_{n,k}^{(\la)}((0,h_0),\cP^{(\la)}\cup\{(v_1,h_1)\})\xi_{n,k}^{(\la)}((v_1,h_1),\cP^{(\la)}\cup\{(0,h_0)\})]\nonumber\\
&\hspace*{3.8cm}-\E[\xi_{n,k}^{(\la)}((0,h_0),\cP^{(\la)})]\E[\xi_{n,k}^{(\la)}((v_1,h_1),\cP^{(\la)})].    
\end{align}

We conclude by giving a more precise statement of Theorem \ref{thm:principalintro}, as we have now introduced every notation involved in the limiting constants.

\begin{thm}\label{lim esperance precis}
	For any $n \geq 1$ and $k \in \lbrace 0, \ldots , d-1 \rbrace$ we have
	$$\lim\limits_{\lambda \rightarrow \infty} {\lambda^{-\frac{d-1}{d+1}}} {\mathbb{E}[N_{n,k,\lambda}]}
	= \text{\normalfont Vol}_{d-1}(\mathbb{S}^{d-1}) \int_{0}^{+\infty} \mathbb{E}[\xi_{n, k}^{(\infty)}((0,h), \mathcal{P})] \mathrm{d}h \in ( 0, \infty). $$
\end{thm}

\begin{thm}\label{lim variance precis}
	For any $n \geq 1$ and $k \in \lbrace 0, \ldots , d-1 \rbrace$ we have
	\begin{align*}
	&\lim\limits_{\lambda \rightarrow \infty} {\lambda^{-\frac{d-1}{d+1}}}{\text{\normalfont Var}[N_{n,k,\lambda}]} 
	= \text{\normalfont Vol}_{d-1}(\mathbb{S}^{d-1}) \left( I_1 + I_2  \right) \in ( 0, \infty)
	\end{align*}
	where 
	\begin{equation}\label{eq:defI1}
	I_1 := \int_{0}^{\infty} \mathbb{E}[\xi_{n, k}^{(\infty)}((0,h), \mathcal{P})^2] \mathrm{d}h
	\end{equation}
	and 
	\begin{equation}\label{eq:defI2}
	I_2 := \int_{0}^{+\infty} \int_{0}^{+\infty} \int_{\mathbb{R}^{d-1}} c_{n, k}^{(\infty)}((0,h_0), (v_1, h_1)) \mathrm{d}v_1 \mathrm{d}h_0 \mathrm{d}h_1 .    
	\end{equation}
	
\end{thm}

We restate in a similar way Theorem \ref{thm:principalintrovol} for the intrinsic volumes, using the definitions of $\xi^{(\infty)}_{V,n,k}$ and $c^{(\infty)}_{V,n,k}$ introduced at \eqref{eq:defxiinftyVnk} and \eqref{eq:defclaVnk} respectively.

\begin{thm}\label{lim esperance precis vol}
	For any $n \geq 1$ and $k \in \lbrace 1, \ldots , d \rbrace$ we have
	$$\lim\limits_{\lambda \rightarrow +\infty} \lambda^{\frac{2}{d+1}}\mathbb{E}[V_{n,k,\lambda}] 
	= \text{\normalfont Vol}_{d-1}(\mathbb{S}^{d-1}) \int_{0}^{+\infty} \mathbb{E}[\xi_{V,n, k}^{(\infty)}((0,h), \mathcal{P})] dh \in ( 0, \infty ). $$
\end{thm}

\begin{thm}\label{lim variance precis vol}
	For any $n \geq 1$ and $k \in \lbrace 1, \ldots , d \rbrace$ we have
	\begin{align*}
	&\lim\limits_{\lambda \rightarrow +\infty} \lambda^{\frac{d+3}{d+1}} \text{\normalfont Var}[V_{n,k,\lambda}]
	= \text{\normalfont Vol}_{d-1}(\mathbb{S}^{d-1}) \left( I_1 + I_2  \right) \in (0, \infty )
	\end{align*}
	where 
	\begin{equation*}\label{eq:defI1 V}
	I_1 := \int_{0}^{\infty} \mathbb{E}[\xi_{V,n, k}^{(\infty)}((0,h), \mathcal{P})^2] \mathrm{d}h
	\end{equation*}
	and 
	\begin{equation*}\label{eq:defI2 V}
	I_2 := \int_{0}^{+\infty} \int_{0}^{+\infty} \int_{\mathbb{R}^{d-1}} c_{V,n, k}^{(\infty)}((0,h_0), (v_1, h_1)) \mathrm{d}v_1 \mathrm{d}h_0 \mathrm{d}h_1 .    
	\end{equation*}
	
\end{thm}


\section{Stabilization}\label{sec:stab}
The aim of this section is to show stabilization results for the considered scores. This means roughly that the score calculated at one particular fixed point requires the knowledge of the Poisson points outside of a lateral neighborhood of that fixed point with an exponentially decreasing probability. This tool is essential to get moment bounds in Lemma \ref{borne Lp}, then the convergence of the mean of one score and of the covariance of the scores and ultimately our main results, i.e Theorems \ref{thm:principalintro} and \ref{thm:principalintrovol}. 

\subsection{Local scores and stabilization radius}

First we extend the notion of score to a local score in an angular sector around a point in the following way. For $x\in \B^d$ and $r>0$, we introduce $S(x,r)=\{y\in\R^d:d_{\S^{d-1}}(x/\|x\|,y/\|y\|)\le r\}$ where $d_{\S^{d-1}}$ is the geodesic distance along $\S^{d-1}$ and
$$\xi_{n,k,[r]}(x,\cP_\la):=\xi_{n,k}(x,\cP_\la\cap S(x,r)).$$
We define the stabilization radius in the initial model as
$$R_{n,k}(x,\cP_\la):=\inf\{R>0:\xi_{n,k}(x, \cP_{\la}) =  \xi_{n,k, [r]}(x, \cP_{\la})\  \forall r \geq R \}.$$
In particular, thanks to the rotation invariance of $\cP_\la$, we get the identity in law
\begin{equation}\label{eq:rotinvRnk}
R_{n,k}(x,\cP_\la)\overset{\mbox{\tiny{(d)}}}{=}R_{n,k}(\|x\|e_d,\cP_\la).    
\end{equation}
We formally introduce the stabilization radius in the rescaled model as
\begin{equation}\label{eq:defRnkla}
 R_{n,k}^{(\la)}(w, \cP^{(\la)}) := R_{n,k}([T^{(\la)}]^{-1}(w),\cP_\la).
 \end{equation}
Combining \eqref{eq:defRnkla} with \eqref{eq:rotinvRnk}, we obtain the invariance under horizontal translation of $R_{n,k}^{(\la)}$, namely for any $(v,h)\in W_\la$,
\begin{equation}\label{eq:translinvRnkla}
R_{n,k}^{(\la)}((v,h),\cP^{(\la)})\overset{\mbox{\tiny{(d)}}}{=}R_{n,k}^{(\la)}((0,h),\cP^{(\la)}).
\end{equation}

%

Since the image by $T^{(\la)}$ of $S(e_d,r)$ for $r\in (0,\pi)$ is a cylinder, we also extend the notion of score in the rescaled picture to a local score in a cylinder of radius $r$ around a point in the following way. For any $r > 0$, $\la \in (0, \infty]$ and $w = (v,h) \in W_\la$ we write \begin{equation}\label{eq:xinkr}\xi^{(\la)}_{n,k, [r]}(w, \cP^{(\la)}) := \xi^{(\la)}_{n,k}(w, \cP^{(\la)}\cap C_v(r)) .
\end{equation}
When $w = (0,h)$, the following equality provides a convenient expression of the stabilization radius, which is the one we  use most of the time:
$$ R_{n,k}^{(\la)}((0,h), \cP^{(\la)}) = \inf \{R > 0\ :\  \xi_{n,k}^{(\la)}((0,h), \cP^{(\la)}) =  \xi_{n,k, [r]}^{(\la)}((0,h), \cP^{(\la)})\  \forall r \geq R \}.$$
We provide estimates for the distribution tail of $R_{n,0}^{(\la)}(w,\cP^{(\la)})$ in Section \ref{sec:stabpoints} and of $R_{n,k}^{(\la)}(w,\cP^{(\la)})$ in Section \ref{sec:stabfaces}.
\subsection{Stabilization for points}\label{sec:stabpoints}
A prerequisite for Sections \ref{sec:stabpoints} and \ref{sec:stabfaces} is the following geometric lemma that we will use extensively when showing the stabilization property. Though standard, we include its proof below for the reader's convenience.
\begin{lem}\label{lem:demiparabole}
	Let $\la\in (0,\infty]$ and $w_0, w_1 \in H$ such that $\partial[\Pi^\downarrow]^{(\la)}(w_1)$ goes through $w_0$. Then there exists a half-space $P^+$ delimited by a hyperplane $P$ going 
	through $w_0$ with direction containing $(0, 0, \ldots, 0, 1)$ such that $[\Pi^\downarrow]^{(\la)}(w_0) \cap P^+  \subseteq [\Pi^\downarrow]^{(\la)}(w_1) $.
\end{lem}
\begin{proof}
 For finite $\la$, this is a direct consequence of the following fact in the non-rescaled model: for any point $x_0=r e_d$, $r\in (0,1)$ and any point $x_1\in \partial B(\frac{x_0}{2},\frac{\|x_0\|}{2})$, the set $\mbox{cap}(x_1)$ as defined in \eqref{eq:defcap} contains at least half of $\mbox{cap}(x_0)$. Using \eqref{eq:paraboleverslebas} and \eqref{eq:paraboleverslehaut}, we deduce the required result in the rescaled model.
 
	For $\la=\infty$, we proceed along the following lines. An orthogonal transformation allows us to assume that $w_0 = (0, 0, \ldots, 0, h_0)$ for some $h_0 > 0$ and 
	$w_1 = (a, 0, \ldots, 0, h_1)$ with $a, h_1 > 0$. Since $\partial\Pi^\downarrow(w_1)$ goes through $w_0$, we must have
	$a = \sqrt{2(h_1 - h_0)}$.
	Let us show that $\Pi^\downarrow(w_0) \cap \{ (v,h) : (v)_1 > 0 \}  \subseteq \Pi^\downarrow(w_1) $ where $(v)_1,\ldots,(v)_{d-1}$ denote the consecutive coordinates of $v$. The equations of both paraboloids are
	$$ \Pi^{\downarrow}(w_0): \bigg(h < h_0 - \frac{1}{2}\sum_{i=1}^{d-1} (v)_i^2\bigg)\;\;\mbox{ and }\;\; \Pi^{\downarrow}(w_1): \bigg(h < h_1 - \frac{1}{2} \big(((v)_1 - a)^2 + \sum_{i = 2}^{d-1} (v)_i^2 \big)\bigg). $$
	For $(v,h) \in \Pi^\downarrow(w_0) \cap \{ (v,h) : (v)_1 > 0 \}$, using $a = \sqrt{2(h_1 - h)}$, we get
	\begin{align*}
		h 
												&< h_0 - \frac{1}{2} \left(((v)_1 - a)^2 + \sum_{i = 2}^{d-1} (v)_i^2 \right)
		+ \frac{1}{2}\sqrt{2(h_1 - h_0)}^2 - a (v)_1 \\
		&< h_1 - \frac{1}{2} \left(((v)_1 - a)^2 + \sum_{i = 2}^{d-1} (v)_i^2 \right).
	\end{align*}
This completes the proof.
\end{proof}

In the next lemma we show that the maximal height of the Poisson points on the $n$-th layer $\partial \Phi_n^{(\la)}$ of the quasi-parabolic hull peeling inside a cylinder is bounded with a probability going to 1 exponentially fast with respect to the bound. This will be essential for proving the stabilization result in Proposition \ref{Stabilisation points} and will also be useful when proving Lemma \ref{stab hauteur} which provides a stabilization in height. 

Here and in the sequel we denote by $c, c_1, c_2 ...$ generic positive constants that only depend on $n$, $k$ and $d$ and which may change from line to line.
\begin{lem}\label{lemme hauteur max}
	For all $n \geq 1$, there exist $\lambda_0, c_1,c_2 > 0$ such that for all $t \geq 0$,  $\lambda \in [\lambda_0, \infty]$ and $1 \leq  r < \pi \la^{\frac{1}{d+1}}$, we have 
		 \begin{align*}
		 &
		 \mathbb{P}\left( \exists (v, h) \in  {\mathcal P}^{(\lambda)} \cap \partial \Phi_n^{(\lambda)}\left(\mathcal{P}^{(\lambda)} \cap C(r)\right) \cap C(r/2^n) \text{ with } h \geq t \right) 
		 \leq c_1 r^{d-1} e^{-c_2t(r\wedge\sqrt{t})^{d-1}}
		 \end{align*}
		 and
		 \begin{align*} 
		 &\mathbb{P}\left( \exists (v, h) \in  {\mathcal P}^{(\lambda)} \cap \partial \Phi_n^{(\lambda)}\left(\mathcal{P}^{(\lambda)}\right) \cap C(r/2^n) \text{ with } h \geq t \right)
		 \leq c_1 r^{d-1} e^{-c_2t(r\wedge\sqrt{t})^{d-1}}.
\end{align*}

\end{lem}

\begin{proof}
We only prove the first inequality as the method for getting the second one is very similar.
			We begin with the proof for $\lambda = \infty$ as the case $\lambda < \infty$ is a bit more
	technical. We are going to show it by induction on $n$. 
	We first prove the induction step as it contains the main ideas. 
	Then we describe what needs to be changed to prove 
	the induction step
	for $\lambda < \infty$ and finally we explain the slight modifications that are needed to prove the base case.\\
	
	\noindent\textit{Proof of the induction step for $\lambda = \infty$.} 
	We assume that the result holds for all $l < n$ with a fixed $n > 1$
	and we show that it holds for $n$. 
	Let $w = (v,h) \in C(r/2^n)$ with $h \geq t$. Our first step is to show that the event
	\begin{equation}\label{lemme hauteur max eq w}
		\lbrace w \in \partial\Phi_n\left((\mathcal{P} \cup \lbrace w \rbrace) \cap C(r))\right \rbrace
	\end{equation}
	 occurs with probability smaller than 
	$c_1 \exp(- c_2 h (r\wedge \sqrt{h})^{d-1}). $ Here we add $w$ to the point process because we plan to use Mecke's formula later to deal with a union over all $w$.
	
	Let $(v_1,h_1) \in \R^{d-1}\times \R_+$ such that $w \in \partial\Pi^\downarrow(v_1, h_1)$ and  $\Pi^\downarrow(v_1, h_1)$ only contains points of layer at most $(n-1)$ for $C(r)$. Lemma \ref{lem:be on layer n} (ii) guarantees that such a $(v_1,h_1)$ exists.
	By Lemma \ref{lem:demiparabole}, this downward paraboloid contains at least half of $\Pi^\downarrow(w)$. Consequently, denoting by $A_1,\ldots,A_{2^{d-1}}$ the intersections of $\Pi^{\downarrow}(w)$ with the product of an orthant of $\R^{d-1}$ translated by $v$ with $\R_+$, 
	we have $\text{Vol}_d(A_i  \cap C^{\ge h/2}(r/2^{n-1})) = \frac{1}{2^{d-1}} \text{Vol}_d(\Pi^\downarrow(w)  \cap C^{\ge h/2}(r/2^{n-1}))$ and $\Pi^\downarrow(v_1, h_1)$ contains at least one of the $A_i$. 
	
	Let us write 
	\begin{equation}\label{eq:Bi}
	 B_{i,n}(h,r) := A_i \cap \{ (v',h') : h' \geq h/2 \} \cap C(r/2^{n-1})  .
	 \end{equation}
	From the preceding reasoning we deduce that 
    \begin{align*}
			\lbrace w \in \partial\Phi_n\left((\mathcal{P} \cup \lbrace w \rbrace) \cap C(r))\right \rbrace
	\subseteq \{ \exists i : \cP \cap B_{i,n}(h,r) \subseteq [\Phi_{n}((\cP\cup\{ w\})\cap C(r))]^c \}.
	\end{align*}
	For fixed $i$, $B_{i,n}(h,r)$ is either empty, which happens with probability smaller than $c_1 \exp(-c_2 h (r\wedge \sqrt{h})^{d-1})$ or it contains a point at height larger than $h/2$ on a layer at most $(n-1)$, which happens with probability smaller than $c_1 \exp(-c_2 h (r\wedge \sqrt{h})^{d-1})$ by the induction hypothesis. Consequently we have 
	\begin{equation}\label{lemme hauteur max eq w fixe 2}
	\P(\lbrace w \in \partial\Phi_n\left((\mathcal{P} \cup \lbrace w \rbrace) \cap C(r)\right) ) \le c_1 \exp(-c_2 h (r\wedge \sqrt{h})^{d-1}) . 
	\end{equation}
	
	We can write
	\begin{align*}
	\mathbb{P}\left( \exists (v, h)  \in \mathcal{P}\cap \partial\Phi_n(\mathcal{P} \cap C(r)) \cap C(r/2^n)
	 \text{ with } h \geq t \right)& \\
	\leq \mathbb{E}\left[ \sum_{w \in \mathcal{P} \cap C^{\geq t}(r/2^n) }
	\mathds{1}_{w \in \partial\Phi_n(\mathcal{P} \cap C(r))}  \right]&.
	\end{align*}
	We combine this with Mecke's formula and \eqref{lemme hauteur max eq w fixe 2} to get 
		\begin{align*}
	\mathbb{P}&\left( \exists (v, h)  \in \mathcal{P}\cap\partial\Phi_n(\mathcal{P} \cap C(r)) \cap  C(r/2^n)
	 \text{ with } h \geq t \right) \\
	&\leq \int_{\Vert v \Vert \leq r/2^n} \int_{h \in \left] t, +\infty\right[} 
	\mathbb{P}\left((v,h) \in \partial \Phi_n
	((\mathcal{P} \cup \{ (v, h) \} ) \cap C(r)) \right) \mathrm{d}h \mathrm{d}v\\
	&\leq \int_{\Vert v \Vert \leq r/2^n} \int_{h \in \left] t, +\infty\right[} c_1  \exp(- c_2 h (r \wedge \sqrt{h})^{d-1}) \mathrm{d}h \mathrm{d}v\\
	&\leq c_1 r^{d-1} \exp(- c_2 t (r \wedge \sqrt{t})^{d-1}).
	\end{align*}
	This proves the induction step.
	\\~\\
	\textit{Proof of the induction step for $\la < \infty$.}
	Let us check that the proof above still holds. The only difference here is that the intensity of the process is no longer constant. However let us recall that this intensity has a density given by \eqref{eq:intensity process}
 so it is uniformly bounded from below for any $\Vert v \Vert \leq \frac{3}{4} \la^{\frac{1}{d+1}} \pi$ and $h \leq \frac{3}{4} \la^{\frac{2}{d+1}}$ by a constant that does not depend on $\la$ and is upper bounded by $1$. The same proof as in the case $\la = \infty$ shows that \begin{align*}
			\{ w \in \partial\Phi_n^{(\la)}((\mathcal{P}^{(\la)} \cup \lbrace w \rbrace) \cap C(r)) \}
	\subseteq \{ \exists i : \cP^{(\la)} \cap B_{i,n}(h,r) \subseteq [\Phi_{n}^{(\la)}((\cP^{(\la)}\cup\{ w\})\cap C(r))]^c \}.
	\end{align*}
	with $B_{i,n}(h,r)$ introduced at \eqref{eq:Bi}.
	If for a fixed $i$, $B_{i,n}(h,r)$ is empty then in particular the set
	$B_{i,n}(h,r) \cap \{ (v',h') : \frac{1}{2}h  \le h' \le \frac{3}{4}h \} \cap C(r/2^{n-1})$ is also empty and included in the region $\{ (v', h') :  \Vert v' \Vert \leq \frac{3}{4} \la^{\frac{1}{d=1}} \pi \text{ and } h' \leq \frac{3}{4} \la^{\frac{2}{d+1}} \}$ on which the density at \eqref{eq:intensity process} is bounded from below by a constant. Consequently, 
	\begin{align*}
	\P(\cP^{(\la)}\cap B_{i,n}(h,r) =\varnothing)&\le e^{-c_1 \text{Vol}_d(B_{i,n}(h,r))}\le e^{-c_2 h(r\wedge \sqrt{h})^{d-1}}.    
	\end{align*}
The use of the induction hypothesis remains unchanged so we still have \eqref{lemme hauteur max eq w fixe 2} in the case $\la < \infty$. To get the result we then follow the same steps as before except that we upper-bound the intensity density by $1$ after the use of Mecke's formula.
	\\~\\
	\textit{Proof of the base case $n=1$ for both $\lambda=\infty$ and $\lambda<\infty$.}
 We define for every $1\le i\le 2^{d-1}$
	$$\tilde{B}_{i,1}(h,r)=A_i\cap\{(v',h'):\frac12 h\le h'\le \frac34 h\}\cap C(3r/4).$$
which guarantees that the intensity measure of $\tilde{B}_{i,1}(h,r)$ is lower bounded by its Lebesgue measure up to a multiplicative constant. Using the inclusion 
$$\{w\in \partial\Phi_1\left((\mathcal{P}^{(\la)} \cup \lbrace w \rbrace) \cap C(r))\right \rbrace
	\subseteq \{ \exists i : \cP^{(\la)}\cap \tilde{B}_{i,1}(h,r) =\varnothing\},$$
we get an analogue of \eqref{lemme hauteur max eq w fixe 2} which, combined with Mecke's formula, proves the base case.

This completes the proof of Lemma \ref{lemme hauteur max}.
\end{proof}

We are now ready to prove a stabilization result for the $0-$faces, i.e. the extreme points of the $n-$th layer. It is a crucial step towards a general stabilization for $k-$faces. 
\begin{prop}\label{Stabilisation points}
	For all $n \geq 1$, there exist  $\lambda_0, c_1,c_2 > 0$  such that for any $h_0 > 0$, $\lambda \in \left[\lambda_0, +\infty\right]$ and $1 \leq r < \pi\la^{\frac{1}{d+1}}$ we have 
	\begin{equation}\label{eq:majoration queue rayon}
	 \mathbb{P}\left( R_{n,0}^{(\lambda)}(0,h_0) 
	\geq r \right) \leq c_1 \exp\left( -c_2 r^{d+1}  \right) .
	\end{equation}
\end{prop}
\begin{proof}
    We give a detailed proof for the case $\lambda = \infty$ and briefly describe at the end how to adapt the proof to make it work for finite $\lambda$. 

	We show this result by induction.
	The case $n=1$ corresponds to \cite[Lemma 6.1]{CSY}. We now fix $n\ge 2$ and assume that \eqref{eq:majoration queue rayon} is verified for all $m<n$. Let us show \eqref{eq:majoration queue rayon} for $n$.
	
	We first notice that
			\begin{align*}
		\{R_{n,0}^{(\lambda)}(0,h_0) 
	\geq r\} &= \bigcup_{s\geq r}\{ \xi_{n,0}^{(\infty)}((0,h_0), \mathcal{P} \cap C(s)) 
		\neq \xi_{n,0}^{(\infty)}((0,h_0), \mathcal{P})\}
		    \end{align*}
		    Let us introduce
		\begin{align*}
		E_1 &:= \bigcup_{s\geq r}\bigcup_{l < n} \{ (0,h_0) \in \partial\Phi_l(\mathcal{P} \cap C(s)) \cap \partial\Phi_n(\mathcal{P})\},\\
		E_2 &:= \bigcup_{s\geq r} \bigcup_{l > n} \{ (0,h_0) \in \partial\Phi_n(\mathcal{P} \cap C(s)) \cap \partial\Phi_l(\mathcal{P})
		\}.
	\end{align*}
	Since $\{R_{n,0}^{(\lambda)}(0,h_0) 
	\geq r\} = E_1 \cup E_2$, it is enough to prove  that for any $r\ge 1$,
	\begin{equation}\label{eq:majPE1}
	\P(E_1)\le c_1\exp(-c_2r^{d+1})
	\end{equation}
	and
	\begin{equation}\label{eq:majPE2}
	\P(E_2)\le c_1\exp(-c_2r^{d+1}).
	\end{equation}
	
	\noindent\textit{Decomposition of $E_1$}.
	The strategy is the following: we plan to select a down-paraboloid which contains $(0,h_0)$ on its boundary and a point $w$ in its interior to which we can apply the induction hypothesis, recalling \eqref{eq:translinvRnkla}.

	To do so, we introduce the two events
	\begin{align*}
	&F_1 := \{\ \exists (v_1, h_1) \in \partial\Pi^{\uparrow}(0,h_0),\ h_1 \leq r^2/32 \text{ : } \\&\hspace{3cm} \cP \cap C(r)\cap \Pi^{\downarrow}(v_1, h_1)  \subseteq 
	[\Phi_{n-1}(\cP\cap C(r))]^c\}
		\cap \{(0,h_0) \in \partial \Phi_n(\cP) \},\\
	&F_2 := \{ \ \exists (v_1, h_1) \in \partial \Pi^{\uparrow}(0,h_0),\ h_1 \geq r^2/32 \text{ : }\\& \hspace{3cm} \cP \cap C(r)\cap \Pi^{\downarrow}(v_1, h_1) \subseteq
	[\Phi_{n-1}(\cP\cap C(r))]^c\}
	\cap \{(0,h_0) \in \partial \Phi_n(\cP)\},
	\end{align*}
    see Figure \ref{fig:F2}.
	    \begin{figure}
        \centering
        \begin{overpic}[width=0.6\textwidth]{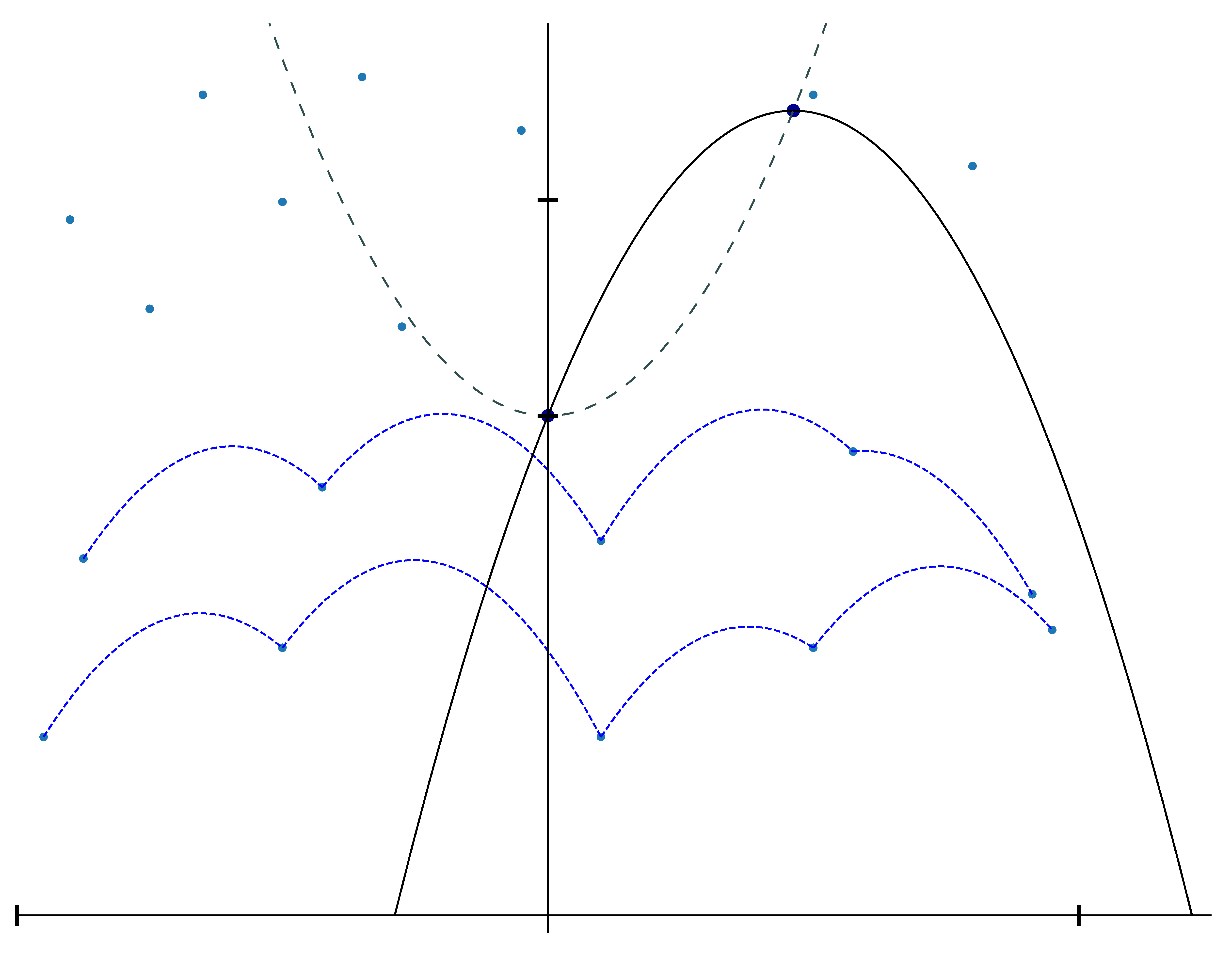}
 \put (45.7,41.3) {$\displaystyle h_0$}
 \put (37.8,61.7) {$\displaystyle \frac{r^2}{32}$}
 \put (52, 72) {$\displaystyle (v_1, h_1)$}
 \put (86.8, 1) {$\displaystyle r$}
 \put (-1, 1) {$\displaystyle -r$}

\end{overpic}
        \caption{Example of a situation where event $F_2$ from Proposition \ref{Stabilisation points} occurs when $(0,h_0)$ is on layer $3$ for the peeling of $\cP \cap C(r)$.}
        \label{fig:F2}
    \end{figure}
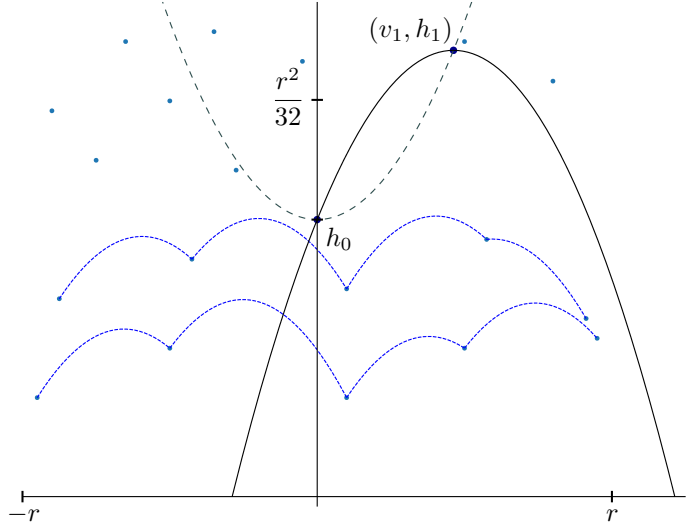

	When $n=2$, we replace the inclusion 
	$\cP \cap C(r)\cap \Pi^{\downarrow}(v_1, h_1) \subseteq
	[\Phi_{n-1}(\cP\cap C(r))]^c$
	in the definition of $F_1$ and $F_2$ above by $\cP\cap C(r)\cap \Pi^{\downarrow}(v_1,h_1)=\varnothing$.
	
	In particular, $E_1 \subseteq F_1 \cup F_2$. Indeed
	if  one of the events of the union in the definition of $E_1$ occurs for fixed $l < n$ and $s \geq r$, there exists
	$(v_1, h_1) \in \partial\Pi^{\uparrow}(0, h_0)$ such that for every $w\in \cP\cap C(r)\cap\Pi^{\downarrow}(v_1, h_1)$,  $\ell^{(\infty)}(w,\cP\cap C(s))\le (l-1)$. 
	Lemma \ref{lem:dalalreecrit} then implies that $\ell^{(\infty)}(w,\cP\cap C(r))\le (l-1) \le (n-2)$ for any such $w$.
	
	Consequently, it suffices to upper bound $\P(F_1)$ and $\P(F_2)$ which we do with two different strategies. 
	In the case of $F_1$, there is a downward paraboloid $\Pi^{\downarrow}(v_1,h_1)$ which is low enough to be contained in a cylinder smaller than $C(r)$. This implies that we can apply the induction hypothesis to a well chosen point in $\Pi^\downarrow(v_1,h_1)$. When on $F_2$, the downward paraboloid is high enough so that there is a high point $w$ with $\ell^{(\infty)}(w,\cP\cap C(r))\le n$ and we deduce from Lemma \ref{lemme hauteur max} that it happens with exponentially small probability.
	\\~\\
	\noindent\textit{Upper bound for $\P(F_1)$.}
		Let us fix $(v_1,h_1)$ with $h_1\le r^2/32$ as in the event $F_1$.
	Using that $(v_1,h_1)\in \partial \Pi^{\uparrow}(0,h_0)$, we get 
	\begin{equation}\label{eq:majv1}
	\Vert v_1 \Vert = \sqrt{2(h_1 - h_0)} \leq \sqrt{2h_1} = r/4.
	\end{equation}
	Furthermore if we take $(v_2, h_2) \in \Pi^\downarrow(v_1, h_1)$, the norm of $v_2$ is smaller than the norm of $v_1$ plus half of the width of the paraboloid $\Pi^\downarrow(v_1, h_1)$, so
	\begin{equation}\label{eq:majv2}
	\Vert v_2 \Vert \leq \Vert v_1 \Vert + \|v_2-v_1\|\le \frac{r}{4}+\sqrt{2h_1} \le r/2 .
	\end{equation}
	This implies that $\Pi^\downarrow(v_1, h_1) \subseteq C(r/2)$. In particular, when $n=2$, we get that $(0,h_0)$ is an extremal point of $\cP\cup \{(0,h_0)\}$ and subsequently that $F_1=\varnothing$. In the case $n\ge 2$, we proceed in the following way. Since $\ell^{(\infty)}((0,h_0),\cP)=n$, we can choose a point $w\in\Pi^{\downarrow}(v_1,h_1)$ such that $\ell^{(\infty)}(w,\cP)\ge (n-1)$.
	Then because we are on the event $F_1$, we also have $\ell^{(\infty)}(w,\cP\cap C(r))\le (n-2)$.
	By Lemma \ref{lem:dalalreecrit}, this implies that $\ell^{(\infty)}(w,\cP\cap C_w(r/2))\le (n-2)$, which means that 
	\begin{equation}\label{eq:stabrad}
	\exists m \leq (n-2) :
	R_{m,0}^{(\infty)}(w) \geq r/2.
	\end{equation}
	 Using the induction hypothesis
	$$\mathbb{P}(R_{m,0}^{(\infty)}(w) \geq r/2) \leq c_1 \exp(-c_2r^{d+1}).$$
	
	Since $w$ belongs to $\cP\cap C^{\le \frac{r^2}{32}}(r/2)$, we   rewrite 
	\begin{align*}
		\mathbb{P}\left( F_1 \right) &\leq \mathbb{P}\bigg(\bigcup_{m =1}^{n-2} \bigcup_{w \in \mathcal{P} \cap C^{\leq \frac{r^2}{32}}(\frac{r}2)} 
	\{ R_{m,0}^{(\infty)}(w, \mathcal{P}) \geq \frac{r}2 \} \bigg) 
	\le 
	\mathbb{E}\bigg[ \sum_{m=1}^{n-2} \sum_{w \in \mathcal{P} \cap C^{\leq \frac{r^2}{32}}(\frac{r}2)}\hspace*{-.3cm} 
	\mathds{1}_{\{R_{m,0}^{(\infty)}(w, \mathcal{P}) \geq \frac{r}2\}} \bigg].
	\end{align*}
	Now we use Mecke's formula to obtain
	\begin{align}\label{eq:findeF1}
	\mathbb{P}\left( F_1 \right) &\leq \sum_{m \leq n-2}\int_{C^{\leq r^2/32}(r/2)} \mathbb{P}\left( R_{m,0}^{(\infty)}(w, \mathcal{P}) \geq r/2 \right) \mathrm{d}w \nonumber \\
	&\leq c_1\exp(-c_2r^{d+1}).
	\end{align}
	
	\noindent\textit{Decomposition of $F_2$}.
	We rewrite $F_2=G_1\cup G_2$ where
	\begin{align*}
		G_1 := \{  \exists (v_1, h_1&),\ h_1 \geq r^2/32, \Vert v_1 \Vert \leq r/6 : \\& \cP\cap C(r)\cap \Pi^{\downarrow}(v_1, h_1)  \subseteq 
[\Phi_{n-1}(\cP\cap C(r))]^c\}
		\cap \{(0,h_0) \in \partial \Phi_n(\cP)\}.
		\\
		G_2 := \{  \exists (v_1, h_1&),\ h_1 \geq r^2/32,\ \Vert v_1 \Vert \geq r/6 : \\ & \cP\cap C(r)\cap \Pi^{\downarrow}(v_1, h_1) \subseteq 
		[\Phi_{n-1}(\cP\cap C(r))]^c\}
		\cap \{(0,h_0) \in \partial \Phi_n(\cP)\}.
	\end{align*}
	Again, when $n=2$, we replace the inclusion $\cP \cap C(r) \cap \Pi^{\downarrow}(v_1, h_1) \subseteq
	[\Phi_{n-1}(\mathcal{P}\cap C(r))]^c$ in the definition of $F_1$ and $F_2$ above by $\cP\cap C(r)\cap \Pi^{\downarrow}(v_1,h_1)=\varnothing$.
	\begin{figure}[H]
	    \begin{subfigure}{.5\textwidth}
		\centering
		\includegraphics[scale=0.86]{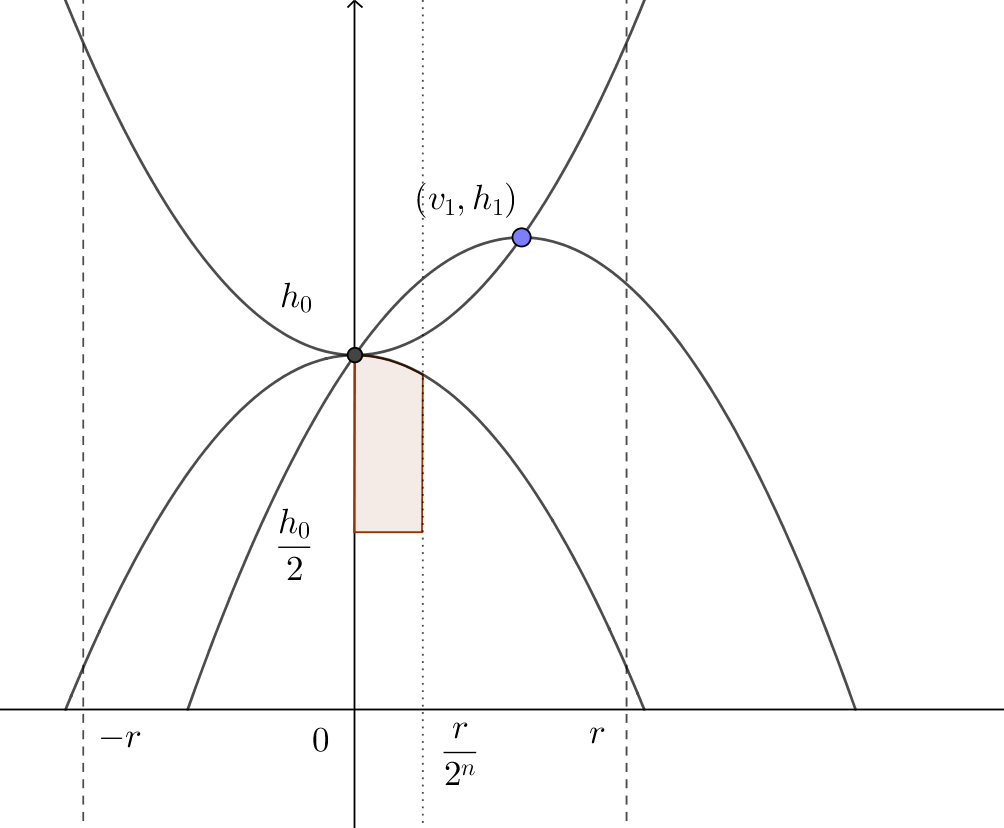}
		\end{subfigure}
		\begin{subfigure}{.5\textwidth}
		\centering
		\includegraphics[scale=0.86]{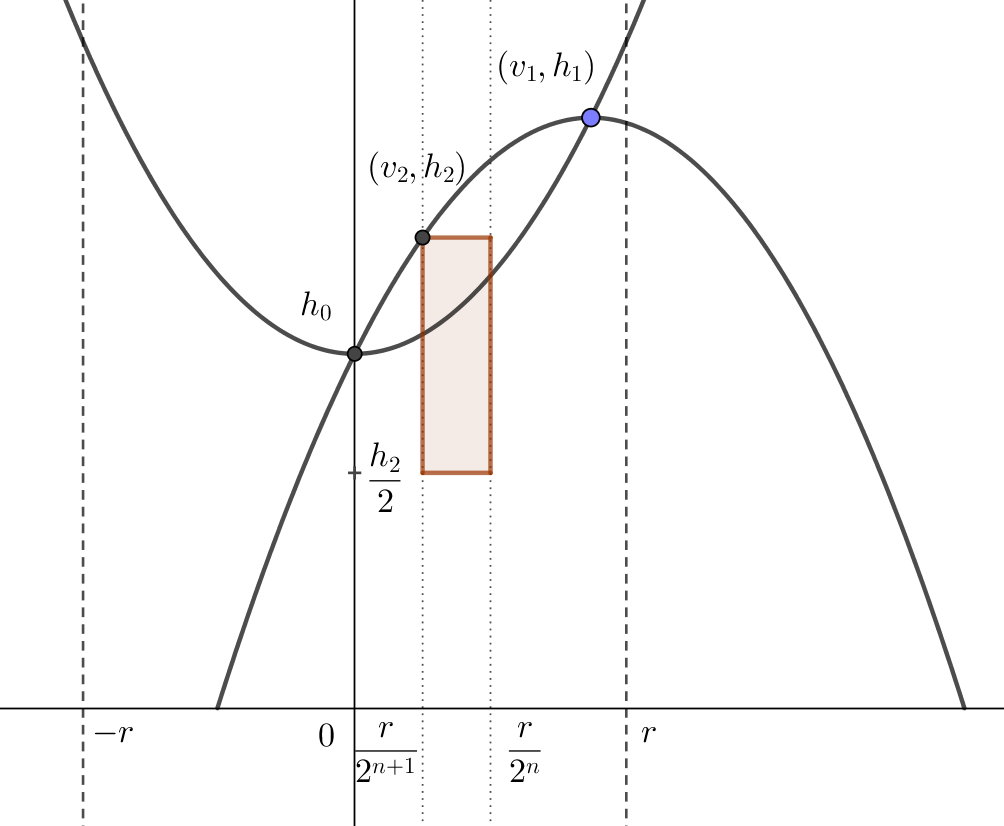}
		\end{subfigure}

		\caption{Illustration of the geometric constructions leading to the upper-bounds of $\P(G_1)$ (left) and $\P(G_2)$ (right)
		}
		\label{StabGen_fig1}
	\end{figure}
	
	\noindent\textit{Upper bound for $\P(G_1)$}.
	 We fix $(v_1,h_1)$ with $\Vert v_1 \Vert \leq r/6$ and $h_1\ge r^2/32$ as in the event $G_1$. In particular, we get
	 \begin{equation}\label{eq:5/288}
	 h_0=h_1-\frac{\|v_1\|^2}{2}\ge \frac{5}{288}r^2.    
	 \end{equation}
	 By Lemma \ref{lem:demiparabole}, the down-paraboloid
	$\Pi^{\downarrow}(v_1, h_1)$ contains half of $\Pi^{\downarrow}(0,h_0) \cap C^{\geq \frac{h_0}{2}}(\frac{r}{2^n})$, see Figure \ref{StabGen_fig1} (left).
	As in the proof of Lemma \ref{lemme hauteur max},  we can find a deterministic subset $A_i$  of $\Pi^\downarrow(0, h_0)$ for some $1\le i\le 2^{d-1}$ such that $A_i\subseteq\Pi^\downarrow(v_1, h_1)$ and 
	\begin{equation}\label{eq:volorthant}
	\text{Vol}_d \bigg(A_i  \cap C^{\geq \frac{h_0}{2}}\Big(\frac{r}{2^n}\Big)\bigg) = \frac{1}{2^{d-1}} \text{Vol}_d\bigg(\Pi^\downarrow(0, h_0) \cap C^{\geq \frac{h_0}{2}}\Big(\frac{r}{2^n}\Big) \bigg).    
	\end{equation}
	The set $A_i\cap C^{\ge \frac{h_0}{2}}(\frac{r}{2^n})$ is empty with probability bounded by $\exp(-c r^{d+1})$ thanks to \eqref{eq:5/288} and \eqref{eq:volorthant}. If $A_i\cap C^{\ge \frac{h_0}{2}}(\frac{r}{2^n})$ is not empty, which happens only when $n\ge 3$, it contains a point $w$ at height at least $\frac{h_0}{2} \geq cr^2$ with $\ell^{(\infty)}(w,\cP\cap C(r))\le (n-2)$.
	Using Lemma \ref{lemme hauteur max} this happens
	with probability smaller than $c_1 \exp(-c_2 r^{d+1})$.
	
	A union bound on the finite number of sets $A_1, \ldots A_{2^{d-1}}$ yields 
	\begin{equation}\label{eq:upperboundPG1}
	\P(G_1) \leq c_1 \exp(-c_2 r^{d+1}).
	\end{equation}
	\\~\\
	\noindent\textit{Upper bound for $\P(G_2)$}.
		 We fix $(v_1,h_1)$ with $\Vert v_1 \Vert \geq r/6$ and $h_1\ge r^2/32$ as in the event $G_2$. Let $P$ be the vertical plane containing the origin and $(v_1, h_1)$ and let $(v_2, h_2)$ be the highest intersection point between
	$\Pi^{\downarrow}(v_1, h_1)$ and $\partial C(r/2^{n+1})$ in $P$.
	Using $\|v_1\|=\sqrt{2(h_1-h_0)}$, we get
	\begin{equation}\label{Stabilisation points h2}
	    h_2 = h_1-\frac12(\|v_1\|-\frac{r}{2^{n+1}})^2=h_0 + \frac{r}{2^{n+1}} \sqrt{2(h_1 - h_0)} - \frac{r^2}{2(2^{n+1})^2}\ge 
	    cr^2. 
	\end{equation}
	We claim that $\Pi^{\downarrow}(v_1,h_1)\cap C^{(\frac{h_2}{2},h_2)}(\frac{r}{2^n})$ contains a deterministic cylinder $C_0$ with width proportional to $r$ and height proportional to $h_2$.\label{page:cylindre}
Indeed, for $n\ge 3$, the cylinder inscribed in $C(\frac{r}{2^n})\setminus C(\frac{r}{2^{n+1}})$
	of radius $\frac{r}{2^{n+2}}$
	between heights $\frac{h_2}2$ and $h_2$ and with axis included in $P$
	as in Figure \ref{StabGen_fig1} (right) is fully included in $\Pi^\downarrow(v_1,h_1)$, thanks to the inequality
	 $\Vert v_1 \Vert \geq \frac{r}{6} \geq \frac{r}{2^n}$. For $n = 2$, the inequality is not satisfied and that is why we take a thinner cylinder inscribed in $C(\frac{r}{6}) \setminus C(\frac{r}{8})$ instead.
	 
	The cylinder $C_0$ constructed above is empty with probability smaller than 
	$ \exp(-ch_2 r^{d-1})$. If $C_0$ is not empty, which happens only when $n\ge 3$, it contains a point $w$ of height at least $\frac{h_2}2$ with $\ell^{(\infty)}(w,\cP\cap C(r))\le (n-2)$. By Lemma \ref{lemme hauteur max} combined with \eqref{Stabilisation points h2}, this happens with probability smaller than $c_1 r^{d-1}\exp(-c_2 h_2 r^{d-1})$.
	
	Finally, discretizing $\partial \Pi^\uparrow(0, h_0)$ for 
	\begin{equation}\label{eq:h1plusgrandqueh0}
	h_1 \geq h_0 + \frac{1}{2} \bigg(\frac{r}6\bigg)^2,
	\end{equation}
	we get
	\begin{align}\label{eq:upperboundPG2}
	\mathbb{P}(G_2) &\leq c_1 \int_{h_0 + r^2/72}^{\infty} (h_1 - h_0)^{(d-2)/2} 
	r^{d-1}e^{-c_2 (h_0 + \frac{r}{2^{n+1}} \sqrt{2(h_1 - h_0)} - \frac{r^2}{2(2^{n+1})^2}) r^{d-1}} \mathrm{d}h_1\notag\\
	&\leq c_1 \int_{h_0 + r^2/72}^{\infty} (h_1 - h_0)^{(d-2)/2} 
	r^{d-1}e^{-c_2  \frac{r}{2^{n+2}} \sqrt{2(h_1 - h_0)}  r^{d-1}} \mathrm{d}h_1\notag\\
	&\leq c_1 \exp(-c_2 r^{d+1}).
	\end{align}
	
	\noindent\textit{Conclusion for $\la = \infty$.}
		Using the inclusion $E_1\subset F_1\cup G_1\cup G_2$ and the estimates for $\P(F_1)$, $\P(G_1)$ and $\P(G_2)$ obtained above, we deduce \eqref{eq:majPE1}.
The estimate \eqref{eq:majPE2}
 is obtained in a very similar fashion, where $n$ plays the role of $l$, by considering the decomposition $E_2\subset F_1'\cup F_2'$ where
\begin{align*}
F_1'&=\{ \exists (v_1, h_1) \in \partial\Pi^{\uparrow}(0,h_0),\ h_1 \leq r^2/32 \text{ : } \\&\hspace{1cm} \cP\cap C(r) \cap \Pi^{\downarrow}(v_1, h_1)  \subseteq 
[\Phi_n(\cP\cap C(r))]^c\}
		\cap \{\exists l > n,\ (0,h_0) \in \partial \Phi_l(\cP) \},\\
	F_2'&= \{ \exists (v_1, h_1) \in \partial \Pi^{\uparrow}(0,h_0),\ h_1 \geq r^2/32 \text{ : }\\& \hspace{1cm} \cP \cap C(r) \cap \Pi^{\downarrow}(v_1, h_1) \subseteq
[\Phi_n(\cP\cap C(r))]^c\}
	\cap \{\exists l > n,\ (0,h_0) \in \partial \Phi_l(\cP)\}.    
\end{align*}
For the sake of brevity, the proof of \eqref{eq:majPE2} is omitted. Combining \eqref{eq:majPE1} and \eqref{eq:majPE2}, we complete the proof of Proposition \ref{Stabilisation points} when $\lambda=\infty$.\\ 
	
	\noindent\textit{Case $\la<\infty$.}
		We recall the two necessary updates for finite $\la$. 
	\begin{enumerate}
	    \item \textit{Density}. The density of the intensity measure of $\cP^{(\la)}$ is lower bounded by a positive constant in a compact subset of $W_\la$ only.
	    \item \textit{Quasi-paraboloids}. The calculations that have been done with paraboloids are valid  for quasi-paraboloids up to a small error.
	\end{enumerate}
	The second update above implies that for $\la$ large enough, \eqref{eq:majv1} and \eqref{eq:majv2} can be replaced by $\Vert v_1 \Vert \leq \frac{r}{3}$ and
	$\Vert v_2 \Vert \leq \frac{2r}{3}$ respectively. The assertion \eqref{eq:stabrad} is in turn replaced by  $R_{m,0}^{(\la)}(w) \geq r/3$ for some $m\le (n-2)$ and we proceed with the same reasoning as before to get \eqref{eq:findeF1}. 
	
	Regarding the upper bounds of $\P(G_1)$ and $\P(G_2)$, we  make the following modifications.
	\begin{itemize}
	\item Because of the second update, the inequalities \eqref{eq:5/288}, \eqref{Stabilisation points h2} and \eqref{eq:h1plusgrandqueh0} deduced from the actual equation of a paraboloid become $h_0\ge cr^2$, $h_2\ge cr^2$ and $h_1\ge h_0+ cr^2$ for $\la$ large enough thanks to \eqref{eq:cvusurtoutcompact}. Moreover, the second equality at \eqref{Stabilisation points h2} stays the same up to a multiplicative constant.
	    \item 	In view of the first update, we replace 
	$A_i\cap C^{\ge \frac{h_0}2}(\frac{r}{2^n}) $ by
	$A_i\cap C^{(\frac{h_0}2,\frac{3h_0}{4})}(\frac{r}{2^n})$ in \eqref{eq:volorthant} and in the rest of the proof of the upper bound for $\P(G_1)$. 
	
	Moreover, the deterministic cylinder $C_0$ constructed in the proof of the upper bound for $\P(G_2)$ on page \pageref{page:cylindre} is replaced by a cylinder included in  $C^{(\frac{h_2}{2},\frac{3h_2}{4})}(\frac{r}{2^n})$ and with radius $\frac{r}{2^{n+3}}$, say. 
	\end{itemize}
We then obtain \eqref{eq:upperboundPG1} and \eqref{eq:upperboundPG2} and conclude as in the case $\la=\infty.$
	
\end{proof} 
Proposition \ref{Stabilisation points} is a general stabilization result for the score at one fixed point and this stabilization is lateral, meaning that the point process is intersected with a cylinder. Lemma \ref{lemme stab rect} below is a complementary stabilization result in a cylinder and the stabilization there is in height, meaning that the point process is intersected with a horizontal strip. 
Combining Lemma \ref{lemme stab rect} with Proposition \ref{Stabilisation points}, we can deduce a general stabilization result both in width and height. The stabilization in height is required to restrict the peeling to a cylinder bounded in height later on and use the continuous mapping theorem, see Lemma \ref{lem:convesp1}. This will ultimately imply in particular a convergence result for the mean of the functional $\xi_{n,k}^{(\la)}$, see Proposition \ref{expectation convergence xi k-faces}. An extra refinement contained in Lemma \ref{lemme stab rect} is that the stabilization in height is proved to be uniform for all the points inside a small cylinder. This will be needed for getting the stabilization in height of the $k$-face score, see Lemma \ref{localisation hauteur k-faces}.

\begin{lem}\label{lemme stab rect}
	For all $n \geq 1$, there exist $\lambda_0, c_1, c_2 > 0$ such that for all $\lambda \in [\la_0, \infty]$ and $1 \le r < \pi\la^{\frac{1}{d+1}}$, we have
	\begin{align*} \mathbb{P}\left(\exists w \in \cP^{(\la)} \cap C^{\leq l_n r^2}\left(\frac{r}{2^{n+1}}\right) : \xi_{n, 0}^{(\lambda)}(w, \cP^{(\la)}\cap C(r)) \neq \xi_{n, 0}^{(\lambda)}(w, \cP^{(\la)}\cap C^{\le l_n r^2}(r)) \right)\\
	\leq c_1\exp(-c_2 r^{d+1})
	\end{align*}
	with $l_n = \frac{1}{2^{2n+7}}$.
\end{lem}
\begin{proof}
	As in the previous proofs, we proceed in the case $\la=\infty$ and explain at the end how to adapt the arguments in the case $\la<\infty$. For fixed $n$, we prove by induction on $m$ that for all $m \leq n$
	there exists $c_1, c_2 > 0$ such that for all $r \ge 1$ we have
	\begin{align}\label{eq:lemencoreplusgeneral} \mathbb{P}\left(\exists w \in \cP\cap C^{\leq l_n r^2}\left(\frac{r}{2^{m+1}}\right) : \xi_{m, 0}^{(\infty)}(w, \cP\cap C(r)) \neq \xi_{m, 0}^{(\infty)}(w, \cP\cap C^{\le l_n r^2}(r) \right)\notag\\
	\leq c_1 \exp(-c_2 r^{d+1}).
    \end{align}
    Lemma \ref{lemme stab rect} is then derived from \eqref{eq:lemencoreplusgeneral} by taking $m=n$.\\
	
	\noindent\textit{Proof of the base case $m=1$ for $\la=\infty$.}
	Let $w \in C^{\leq l_n r^2}\left(\frac{r}{4}\right)$ and assume that 
	$\xi_{1,0}^{(\infty)}(w, \cP \cap C(r)) \neq \xi_{1,0}^{(\infty)}(w, \cP \cap C^{\le l_n r^2}(r))$. 
	Then there exists a downward paraboloid $\Pi^{\downarrow}(v_1,h_1)$, whose boundary contains 
	$w$, that contains no point of $\cP \cap C^{\leq l_n r^2}\left(r\right)$ and contains at least one point of
	$\cP \cap C\left(r\right)$. 
	
	If $h_1 \leq l_n r^2$, we observe that thanks to $l_n = \frac{1}{2^{2n+7}}$ and the fact that $w \in C\left(\frac{r}{4}\right)$, we get for any $(v',h')\in \Pi^{\downarrow}(v_1,h_1)$,
	\begin{align}\label{eq:basecasedetaille}
	 \|v'\|&\le \|v'-v_1\|+\|v_1\|\le \sqrt{2h_1}+\|v-v_1\|+\|v\|\le 2\sqrt{2h_1}+\frac{r}{4}\nonumber\\
	 & \le 2\sqrt{2l_n r^2}+\frac{r}{4}\le r.
	\end{align}
	The last inequality in \eqref{eq:basecasedetaille} implies that  $\Pi^\downarrow(v_1, h_1)$ is contained
	in $C^{\leq l_n r^2}\left(r\right)$ which is excluded. 
	
	If $h_1 > l_n r^2$,
	we claim that the intersection between $\Pi^\uparrow(w)$ and the vertical plane containing $w$ and $(v_1, h_1)$ contains exactly two points at height equal to $l_n r^2$ and we call $(v_2, h_2)$ the one which is closer to $(v_1,h_1)$, see Figure \ref{stab_rect_fig1}. 
	\begin{figure}
		\centering
		\includegraphics[trim={0cm 2.2cm 0 0},clip]{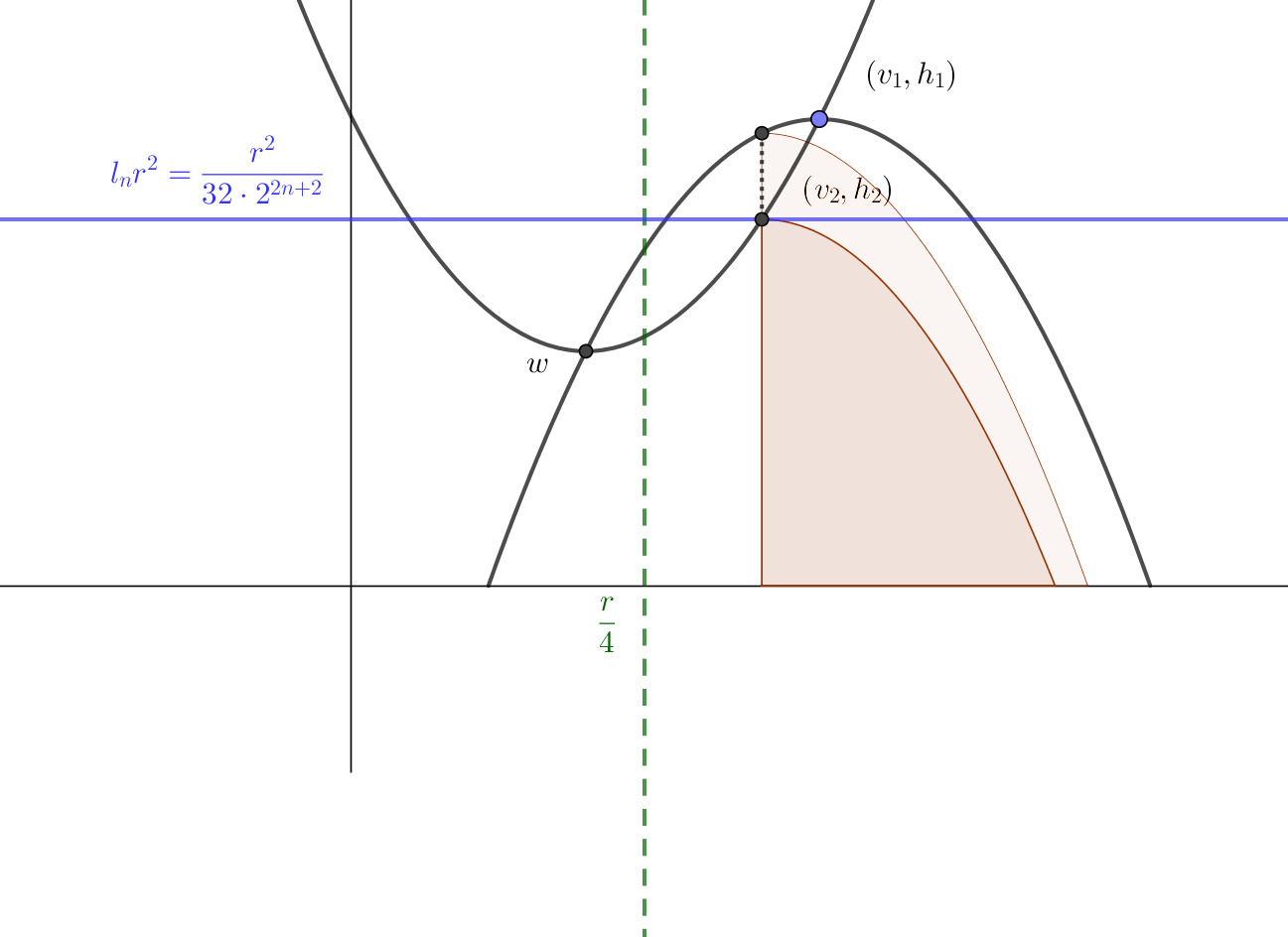}
		\caption{Case $m=1$ of Lemma \ref{lemme stab rect}.}
		\label{stab_rect_fig1}
	\end{figure}
	Thanks to Lemma \ref{lem:demiparabole}, the paraboloid $\Pi^{\downarrow}(v_1, h_1)$ contains half of the down paraboloid with apex at the vertical projection of $(v_2,h_2)$ onto $\partial \Pi^\downarrow(v_1,h_1)$. Consequently, $\Pi^{\downarrow}(v_1, h_1)$ also contains   $\Pi^{\downarrow}_+(v_2, h_2) \cap C^{\leq l_n r^2}(r)$ where $\Pi^{\downarrow}_+(v_2, h_2)$ denotes half of $\Pi^\downarrow(v_2, h_2)$.
	This latter set has volume $c r^{d+1}$ so we can show by using deterministic orthants as in the proof of Lemma \ref{lemme hauteur max} that 
	\begin{equation}\label{eq:probavidebasecase}
	 \P(\cP\cap \Pi^{\downarrow}_+(v_2,h_2)\cap C^{\le l_n r^2}(r)=\varnothing)\le \exp(-cr^{d+1}).    
	\end{equation}
	Denoting by $D$ the set of points of $\partial\Pi^\uparrow(w)$ at height $l_n r^2$ and discretizing $D$, we obtain
	\begin{equation*}
	\P(\exists (v_2,h_2)\in D: \cP\cap \Pi^{\downarrow}_+(v_2, h_2) \cap C^{\leq l_n r^2}(r)=\varnothing)\le c_1 \exp(-c_2r^{d+1}).
	\end{equation*}

	 Using Mecke's formula we get 
		\begin{align*}
		&\mathbb{P}\left(\exists w \in \cP^{(\la)}\cap C^{\leq l_n r^2}\left(\frac{r}{4}\right) : \xi_{1,0}^{(\infty)}(w, \cP^{(\la)} \cap C(r)) \neq \xi_{1,0}^{(\infty)}(w, \cP^{(\la)} \cap C^{\le l_n r^2}(r)) \right) 
		\\&\leq c_1\text{Vol}_d\left(C^{\leq l_n r^2}\left(\frac{r}{4}\right)\right)\exp(-c_2 r^{d+1})
			\le c_1 \exp(-c_2 r^{d+1}). 
	\end{align*}
	This proves the base case for $\la=\infty.$\\
	
	\noindent\textit{Proof of the induction step for $\la=\infty$.}
	Now for fixed $2\le m \le n$ we assume that the result holds for any $p < m$ and we show that it remains true for $m$.
	Let $w \in C^{\leq l_n r^2}\left(\frac{r}{2^{m+1}}\right)$.
	We denote by $E_w$ the event
	$$E_w := \{\xi_{m,0}^{(\infty)}(w, \cP \cap C(r)) \neq \xi_{m,0}^{(\infty)}(w, \cP \cap C^{\le l_n r^2}(r))\} .$$
	When on $E_w$, we also assume that $\xi_{m,0}^{(\infty)}(w, \cP \cap C^{\le l_n r^2}(r))=1 $, i.e. $\ell^{(\infty)}(w,\cP\cap C^{\le l_n r^2}(r))=m$.
	Indeed, the case $\xi_{m,0}^{(\infty)}(w, \cP \cap C(r))=1$
	can be treated in a similar way, see what we did when dealing with events $E_1$ and $E_2$ in the proof of Proposition \ref{Stabilisation points}.
	In particular, when $\ell^{(\infty)}(w, \cP\cap C^{\le l_n r^2}(r))=m$, the depth $\ell^{(\infty)}(w, \cP\cap C(r))$ is larger than $(m+1)$ thanks to Lemma \ref{lem:dalalreecrit}.
	
	
	In other words, there exists a downward paraboloid $\Pi^{\downarrow}(v_1,h_1)$ whose boundary contains $w$ and that only contains points on a layer of order at most
	$(m-1)$ for the peeling in $C^{\leq l_n r^2}\left(r\right)$ and at least one point denoted by $(v_3,h_3)$ such that 
	\begin{equation}\label{couchedev3h3}
	\ell^{(\infty)}((v_3,h_3), \cP\cap C(r))\ge m\mbox{ and }
	\ell^{(\infty)}((v_3,h_3), \cP\cap C^{\leq l_n r^2}(r))\le m-1.
	\end{equation}
	
	If $h_1 \leq l_n r^2$, then for $(v',h')\in \Pi^{\downarrow}(v_1,h_1)$, we obtain by the same method as in \eqref{eq:basecasedetaille} that
	\begin{align}\label{eq:calibration ln}
	\|v'\|\le 2\sqrt{2h_1}+\frac{r}{2^{m+1}}\le     2 \sqrt{2l_n r^2} +\frac{r}{2^{m+1}}\leq \frac{r}{2^{m}}.
	\end{align}
	We notice that the last inequality in \eqref{eq:calibration ln} still holds when $l_n$ is replaced by $\frac{1}{2^{2n+5}}$, which means that our current calibration takes into account the case $\la<\infty$ which is discussed at the end of the proof. The estimate \eqref{eq:calibration ln} shows that the paraboloid $\Pi^{\downarrow}(v_1,h_1)$ stays in $C^{\leq l_n r^2}\left(\frac{r}{2^m}\right)$. Consequently the point $(v_3, h_3)$ introduced above belongs to $C^{\leq l_n r^2}\left(\frac{r}{2^m}\right)$ and by \eqref{couchedev3h3}, it satisfies
	$\xi_{p,0}^{(\infty)}((v_3, h_3), \cP\cap C^{\le l_n r^2}(r)) \neq \xi_{p,0}^{(\infty)}((v_3,h_3), \cP\cap C(r))$ for $p=\ell^{(\infty)}((v_3,h_3),\cP\cap C^{\le l_n r^2}(r))<m$. Using the induction hypothesis, this happens
	with probability smaller than $c_1 \exp(-c_2 r^{d+1})$.
	
	If $h_1 > l_n r^2$ we proceed as in the case $m=1$ by using the construction described in Figure \ref{stab_rect_fig1}. Namely, we show that there exists half of a paraboloid  whose apex is on $\partial\Pi^\uparrow(w)$ at height $l_n r^2$ and that only contains points
	on a layer of order at most $(m-1)$ for the peeling of $\cP \cap C^{\le l_n r^2}(r)$. Let us denote by $A$ this half-paraboloid. Thanks to \eqref{eq:calibration ln}, the set $A$ is included in $C\left(\frac{r}{2^m} \right)$ which implies that 
	\begin{equation}\label{eq:probadevide}
	\P\left(\cP\cap     A \cap C^{\ge l_n r^2/2}\left(\frac{r}{2^m}\right)=\varnothing\right)\le \exp(-c r^{d+1}).
	\end{equation}
	If $\cP\cap A\cap C^{\ge l_n r^2/2}\left(\frac{r}{2^m}\right)\ne \varnothing$, we choose $(v_4,h_4)$ in that set and two cases arise.

	If $\ell^{(\infty)}((v_4,h_4),\cP\cap C(r))=\ell^{(\infty)}((v_4,h_4),\cP\cap C^{\le l_n r^2}(r))$, then $\ell^{(\infty)}((v_4,h_4),\cP\cap C(r))\le (m-1)$ and  $h_4\ge l_n r^2/2 = c r^2$. Using Lemma \ref{lemme hauteur max}, this happens with probability smaller than $c_1 \exp(-c_2 r^{d+1}).$
	
	If $\ell^{(\infty)}((v_4,h_4),\cP\cap C(r))\ne \ell^{(\infty)}((v_4,h_4),\cP\cap C^{\le l_n r^2}(r))$, we can use the induction hypothesis and a union bound for $p < m$ to prove that this happens with probability smaller than
    \sloppy$c_1\exp(-c_2 r^{d+1})$.
	
	To sum up, we have shown that for any $w \in C^{\le l_nr^2}\left(\frac{r}{2^n} \right)$  
	$$\P(E_w) \leq c_1 r^{(m-1)(d+1)} \exp(-c_2 r^{d+1}) .$$ Using Mecke's formula, we finally obtain
	\begin{align*}
	 \mathbb{P}\left(\exists w \in C^{\leq l_n r^2}\left(\frac{r}{2^{m+1}}\right) : \xi_{m,0, }^{(\infty)}(w, \cP^{(\la)}\cap C(r)) \neq \xi_{m,0, }^{(\infty)}(w, \cP^{(\la)}\cap C^{\le l_n r^2}(r)) \right)\\
	\leq c_1 r^{m(d+1)} \exp(-c_2 r^{d+1}).
	\end{align*}
	
\noindent\textit{Case $\la<\infty$.} We have to adapt the arguments where either the actual equation of a paraboloid  or a lower bound of the intensity measure of $\cP$ is used, namely \eqref{eq:basecasedetaille}, \eqref{eq:probavidebasecase}, \eqref{eq:calibration ln} and \eqref{eq:probadevide}. Thanks to \eqref{eq:cvusurtoutcompact}, for $\la$ large enough, the series of estimates leading to \eqref{eq:basecasedetaille} can be replaced by
\begin{equation}\label{eq:inegalités lambda fini}
\|v'\|\le \ldots \le 4\sqrt{2h_1}+\frac{r}{4}\le 4\sqrt{2l_nr^2}+\frac{r}{4}\le r.    
\end{equation}
The adaptation of \eqref{eq:calibration ln} is identical.

In order to show \eqref{eq:probavidebasecase} for $\la<\infty$, we use \eqref{eq:inegalités lambda fini} to show that $[\Pi^{\downarrow}]^{(\la)}(v_2,h_2)$ is included in $C(\frac{r}{2})$ for $\la$ large enough. In particular, the set $[\Pi^{\downarrow}]^{(\la)}(v_2,h_2)\cap C^{\le l_n r^2}(r)$ is contained in $C^{\le \la^{\frac2{d+1}}/2}(\frac{1}{2}\la^{\frac1{d+1}}\pi)$ which is a domain where the intensity measure of $\cP^{(\la)}$ is bounded from below. Consequently, \eqref{eq:probavidebasecase} occurs for $\la<\infty$ as well. In the same way, the estimate \eqref{eq:probadevide} holds for $\la<\infty$ as well because $A\cap C^{\ge l_n r^2/2}\left(\frac{r}{2^m}\right)$ is included in $C^{\le \la^{\frac2{d+1}}/2}(\frac{1}{2}\la^{\frac1{d+1}}\pi)$.
\end{proof}

\subsection{Stabilization for k-faces}\label{sec:stabfaces}
Henceforth, we fix $n\ge 1$ and $k \in \{0, \ldots, d-1 \}$. We aim at proving Proposition \ref{stab k-faces} which states a stabilization result for the quantities $\xi_{n,k}^{(\la)}$ introduced at \eqref{eq:defxink}. To do so, we start with a intermediary lemma on the distribution tail of the height of the parabolic facets containing a fixed point from the $n$-th layer.

Let us recall the definition of the set $\cF_{n,d-1}^{(\la)}(w,Y)$ given on page
\pageref{page:defFnk}. For any locally finite set $Y\subset \R^{d-1}\times \R_+$ and $w=(v,h)\in \R^{d-1}\times \R_+$, we introduce 
the maximal height of the facets from $\cF_{n,d-1}^{(\la)}(w,Y)$, i.e. $H_n^{(\la)}(w,Y)=0$ if $w\not\in \partial \Phi_n(Y)$ and otherwise,
\begin{equation}\label{def:H} H_{n}^{(\la)}(w,Y)=\sup\{h_1>0: \exists v_1\in\R^{d-1}, F\in \cF_{n,d-1}^{(\la)}(w,Y) \mbox{ s.t. } 
F\subset\partial[\Pi^{\downarrow}]^{(\la)}(v_1,h_1)
\}.\end{equation}


\begin{lem}\label{stab hauteur}
	For all $n \geq 1$ there exist $\lambda_0, c_1, c_2 > 0$ such that for all $\lambda \in [\lambda_0, +\infty]$, $h_0\in (0,\la^{\frac2{d+1}})$, $t>0$ and $1 \leq r < \pi\la^{\frac1{d+1}}$, we have 
	\begin{equation} \label{stab hauteur eq1}
	\mathbb{P}(\exists s\ge r: H_{n}^{(\lambda)}((0,h_0),\mathcal{P}^{(\lambda)} \cap C(s)) \geq t) \leq 
	c_1  \exp(-c_2\sqrt{t} (r\wedge \sqrt{t})^{d})
	\end{equation}
	and
	\begin{equation}\label{stab hauteur eq2}
	\mathbb{P}(H_{n}^{(\lambda)}((0,h_0),\mathcal{P}^{(\lambda)}) \geq t) \leq c_1 \exp(-c_2 t^{\frac{d+1}{2}}).
	\end{equation}
\end{lem}
\begin{proof}
    We only prove \eqref{stab hauteur eq1} as the proof of \eqref{stab hauteur eq2} is almost identical.\\ 
    
    \noindent\textit{Case $\la = \infty.$}
	For $(v_1,h_1) \in \R^{d-1}\times \R_+$, we introduce the events $$ E_{(v_1,h_1)} = \{ \exists s\ge r, F \in \cF_{n,d-1}((0,h_0),\cP\cap C(s)), F \subseteq \partial \Pi^\downarrow(v_1, h_1) \} $$
	and
	$$E = E(t) := \bigcup_{(v_1, h_1) \in \Pi^\uparrow (0,h_0), h_1 \geq t } E_{(v_1,h_1)} .$$
	It is enough to prove that 
	\begin{equation}\label{stab hauteur cas1}
	    \P(E) \leq c_1 \exp(-c_2 \sqrt{t} (r \wedge \sqrt{t})^{d}).
	\end{equation}
	In the same spirit as in the proof of Proposition \ref{Stabilisation points}, we consider two cases  depending on the value of $h_0$. The left- (resp. right-) hand side of Figure \ref{StabGen_fig1} reflects the first case $h_0\ge t/2$ (resp. $h_0\le t/2$). 
	\\~\\
	\noindent\textit{Case $h_0 \geq t/2$}.
	As in the proof of Lemma \ref{lemme hauteur max}, we start by decomposing $\Pi^{\downarrow}(0,h_0)$ into $2^{d-1}$ deterministic subparts $A_1,\ldots,A_{2^{d-1}}$ such that 
	\begin{align}\label{eq:minvolAi}
	    \text{Vol}_d(A_i  \cap C^{\ge h_0/2}((r\wedge \sqrt{t})/2^{n})) &= \frac{1}{2^{d-1}} \text{Vol}_d(\Pi^\downarrow((0,h_0))  \cap C^{\ge h_0/2}((r\wedge\sqrt{t})/2^{n}))\notag\\&\ge ct(r\wedge\sqrt{t})^{d-1}.
	\end{align}
	
	Let $(v_1, h_1) \in \Pi^\uparrow(0,h_0)$ such that $h_1 \geq t$ and $\partial\Pi^\downarrow(v_1,h_1)$ contains a facet of $\partial \Phi_n(\cP\cap C(s))$ going through $(0,h_0)$ for some $s\ge r$.
	By Lemma \ref{lem:demiparabole}, $\Pi^{\downarrow}(v_1, h_1)$ contains half  of
	$\Pi^{\downarrow}(0,h_0)$, which implies that it contains a set $A_i\cap C^{\geq h_0 /2}((r \wedge \sqrt{t})/2^n)$ for some $i$. Consequently,
	\begin{align}\label{eq:redecomposeA}
	\P(E)&=\P\bigg(\bigcup_{(v_1, h_1) \in \Pi^\uparrow (0,h_0), h_1 > t } E_{(v_1,h_1)}\bigg)\nonumber\\&\le\P\bigg(\bigcup_{s\ge r}\bigcup_{i=1}^{2^{d-1}}\{\cP \cap A_i\cap C^{\geq h_0 /2}((r \wedge \sqrt{t})/2^n)\subset [\Phi_n(\cP\cap C(s))]^c\}\bigg)\nonumber\\
	&\le\P\bigg(\bigcup_{i=1}^{2^{d-1}}\{\cP \cap A_i\cap C^{\geq h_0 /2}((r \wedge \sqrt{t})/2^n)\subset [\Phi_n(\cP\cap C(r))]^c\}\bigg)\nonumber\\
	&\le 2^{d-1}\P(\cP \cap A_1\cap C^{\geq h_0 /2}((r \wedge \sqrt{t})/2^n)\subset [\Phi_n(\cP\cap C(r))]^c).
	\end{align}
where the inclusion $\Phi_n(\cP\cap C(r))\subset\Phi_n(\cP\cap C(s))$ is due to Lemma \ref{lem:dalalreecrit}.	
	By \eqref{eq:minvolAi},
	$\cP$ does not meet $A_1\cap C^{\geq h_0 /2}((r \wedge \sqrt{t})/2^n)$ with probability smaller 
	than $\exp(-c t (r\wedge \sqrt{t})^{d-1})$.
	Otherwise it contains a point $(v_3,h_3)$  with $h_3\ge h_0 /2 \geq t/4$ and such that $\ell^{(\infty)}((v_3,h_3),\cP\cap C(r))\le (n-1)$. Using Lemma \ref{lemme hauteur max}, we obtain that
	$$\P(\cP \cap A_1\cap C^{\geq h_0 /2}((r \wedge \sqrt{t})/2^n)\subset \Phi_n^c(\cP_\la\cap C(r))) \leq c_1 \exp(-c_2 t(r\wedge \sqrt{t})^{d-1}),$$
	which implies \eqref{stab hauteur cas1} thanks to \eqref{eq:redecomposeA} and the inequality $t(r\wedge \sqrt{t})^{d-1}\ge \sqrt{t}(r\wedge \sqrt{t})^d$.\\
	

\noindent\textit{Case $h_0 \leq t/2$}.
		Let $(v_1, h_1) \in \partial \Pi^\uparrow(0,h_0)$ such that $h_1 \geq t$ and $\partial\Pi^\downarrow(v_1,h_1)$ contains a facet of $\partial \Phi_n(\cP\cap C(s))$ going through $(0,h_0)$ for some $s\ge r$.
	For $u > 0$, the height of the highest point of intersection between 
	$\Pi^{\downarrow}(v_1, h_1)$ and $C(u/2^{n+1})$ in the vertical plane containing $(v_1,h_1)$ and the origin
	is 
	\begin{equation}\label{eq:calculh2}
	    h_2 := h_0 + \sqrt{2(h_1 - h_0)}\frac{u}{2^{n+1}} - \frac{u^2}{2^{2n+3}}.
	\end{equation}
We take $u = \sqrt{2(h_1 - h_0)} \wedge r$
	and fit a cylinder $C_0$ of radius $\frac{u}{2^{n+2}}$ between height $h_2/2$ and $h_2$
	in $C(\frac{u}{2^n}) \cap \Pi^{\downarrow}(v_1, h_1)$, see the right-hand side of
	Figure \ref{StabGen_fig1} with $r$ replaced by $u$. 
	Either the point process $\cP$ does not meet the cylinder $C_0$ which happens with probability $\exp(-c h_2 u^{d-1})$ or it contains a point $(v_3,h_3)$ with $h_3\ge h_2/2$ and such that $$\ell^{(\infty)}((v_3,h_3),\cP\cap C(r))\le \ell^{(\infty)}((v_3,h_3),\cP\cap C(s))\le (n-1).$$ 
    Using Lemma \ref{lemme hauteur max}, this happens with probability smaller than
	$c_1\exp\left( - c_2h_2(u \wedge \sqrt{h_2})^{d-1} \right)$. Consequently, we get
\begin{equation}\label{eq:A_{(v_1,h_1)}}
\P(E_{(v_1,h_1)})\le c_1u^{d-1}\exp\left( - c_2h_2(u \wedge \sqrt{h_2})^{d-1} \right).    
\end{equation}	
It remains to make explicit the right-hand side of \eqref{eq:A_{(v_1,h_1)}} in the two cases $u=\sqrt{2(h_1-h_0)}$ and $u=r$. 
	
	When $u = \sqrt{2(h_1-h_0)} \leq r$, we deduce from \eqref{eq:calculh2} and the fact that $h_1 \geq t$ that
	\begin{equation}\label{eq:doubleminoranth2}
	h_2 = h_0 + \frac{h_1 - h_0}{2^n} - \frac{h_1 - h_0}{2^{2(n+1)}}\ge c(h_1-h_0)\ge c't.
	\end{equation}
	Combining \eqref{eq:A_{(v_1,h_1)}} with  \eqref{eq:doubleminoranth2}, we obtain
	\begin{equation}\label{eq:1ereintegrande}
	\P(E_{(v_1,h_1)} ) \le  c_1 \exp(-c_2 (h_1 - h_0) \sqrt{t}^{d-1}) .
	\end{equation}
	
	When $u = r \leq \sqrt{2(h_1-h_0)}$, we obtain in the same way
	\begin{equation}\label{eq:2emeintegrande}
	\P(E_{(v_1,h_1)}) \le  c_1 \exp(-c_2 (h_1 - h_0) r^{d}) .
	\end{equation}
	Discretizing and integrating the right-hand sides of \eqref{eq:1ereintegrande} and \eqref{eq:2emeintegrande} over the set $\{(v_1,h_1)\in \Pi^{\uparrow}(0,h_0): h_1 \geq t\}$, we deduce \eqref{stab hauteur cas1} when $h_0\le t/2$. 
	
	This completes the proof of \eqref{stab hauteur eq1} in the case $\la=\infty$.\\
	
	\noindent\textit{Case $\la<\infty$.} In the case $h_0\ge t/2$, we need to replace $A_i\cap C^{\ge h_0/2}((r\wedge\sqrt{t})/2^n)$ with $A_i\cap C^{(h_0/2, \frac{3h_0}{4})}((r\wedge\sqrt{t})/2^n)$ which is included in $C^{\le \frac34\la^{\frac2{d+1}}}(\frac{\pi}{2}\la^{\frac1{d+1}})$ so that we can lower bound by a constant the density of the intensity measure of $\cP^{(\la)}$ when on that particular set. We adapt the case $h_0\le t/2$ in the exact same way as in the proof of Proposition \ref{Stabilisation points} by replacing the equality \eqref{eq:calculh2} by an inequality up to a multiplicative constant and reducing the cylinder $C_0$ so that it is included in $C^{(h_2/2,\frac34 h_2)}(\frac{u}{2^n})$, which makes it possible to lower bound by a constant the density of the intensity measure of $\cP^{(\la)}$ on $C_0$.
\end{proof}
\begin{prop}\label{stab k-faces}
	For any $n \geq 1$ there exists $\lambda_0, c_1,c_2 > 0$  such that for any $h_0 >0$, $\lambda \in [\lambda_0, +\infty]$ and $1\le r < \pi\la^{\frac1{d+1}}$,  we have 
	$$ \mathbb{P}\left( R_{n,k}^{(\lambda)}(0,h_0) 
	\ge r \right) \leq c_1 \exp\left( -c_2r^{d+1} \right) .$$
\end{prop}
\begin{proof}
Let us assume that $ R_{n,k}^{(\la)}(0,h_0) 	> r $.
	
	First, we can also assume that $R_{n, 0}^{(\la)}(0, h_0) < r$ since the complement
	occurs with probability smaller than $c_1 \exp\left( -c_2r^{d+1} \right)$ by Proposition \ref{Stabilisation points}.
	In particular we have $\xi_{n,0}^{(\la)}((0,h_0), \mathcal{P}^{(\la)}) = \xi_{n,0}^{(\la)}((0,h_0), \mathcal{P}^{(\la)} 
		\cap C(s))$ for every $s\ge r$. We can also assume that both are different from $0$, or equivalently that $(0,h_0)$ is on
	the n-th layer for both peelings because otherwise we would have $R_{n, k}^{(\la)}(0, h_0) \leq r$.
	
	Thanks to Lemma \ref{stab hauteur}, we have $\min(H_{n}((0,h_0),\mathcal{P}^{(\la)}),\max_{s\ge r}H_{n}((0,h_0),\mathcal{P}^{(\la)}\cap C(s))) > r^2/32$ with probability smaller than
	$c_1 \exp\left( -c_2r^{d+1} \right)$. Consequently, we can assume that for any $s\in [r,\infty]$, $H_{n}((0,h_0),\mathcal{P}^{(\la)}\cap C(s)) \leq r^2/32$. 
	
	Let 
	\begin{equation}\label{eq:defEnsembleU}
	\mathcal{U} = \mathcal{U}(h_0,r^2/32) := \bigcup_{(v_1, h_1) \in \partial[\Pi^{\uparrow}]^{(\la)}(0,h_0), h_1 \leq r^2/32} \Pi^{\downarrow}(v_1,h_1).
	\end{equation}
	When $\la=\infty$, a point $(v,h) \in \mathcal{U}$ verifies $\Vert v \Vert \leq \sqrt{2(r^2/32 - h_0)} + \sqrt{2 r^2/32}
		\leq 2\sqrt{2r^2/32} \leq r/2$. By \eqref{eq:cvusurtoutcompact}, this implies that for $\la$ large enough, a point $(v,h)\in {\mathcal U}$ satisfies $\|v\|\le \frac34r.$

	The set $\mathcal{U}$ is designed to include all points of $\cP^{(\la)}$ which lie on a common $k$-face  of $\partial\Phi_n(\cP^{(\la)}\cap C(s))$ with $(0,h_0)$ for every $s\in [r,\infty]$. We assert that
	\begin{equation}\label{eq:inclusionRnk}
	\{R_{n,k}^{(\la)}((0,h_0), \cP^{(\la)})\ge r\}\subset\{\exists w \in \mathcal{P}^{(\la)}\cap\mathcal{U}  : R_{n,0}^{(\la)}(w, \cP^{(\la)}) \ge r/4\}.
	\end{equation}
		Indeed, if every point $w \in \mathcal{P}^{(\la)}\cap\mathcal{U} $ verifies $R_{n,0}^{(\la)}(w, \mathcal{P}^{(\la)}) \leq r/4$, then the status of these points with respect to the $n$-th layer is the same for both $\cP^{(\la)}$ and $\cP^{(\la)}\cap C(s),$ for any $s\ge r$. Consequently, the $k$-faces of the $n$-th layer containing $(0,h_0)$ are the same for the peeling of any $\cP^{(\la)}\cap C(s)$, $s\ge r$, which implies that $R_{n,k}^{(\la)}((0,h_0), \cP^{(\la)})\le r$.  
		 Using consecutively \eqref{eq:inclusionRnk}, Mecke's formula and the fact that the density of the intensity measure of $\cP^{(\la)}$ is upper bounded by $1$, we have
	\begin{align*}
		\mathbb{P}( R_{n,k}^{(\la)}(0,h_0))\ge r
		)
			&\leq \mathbb{E}\left[ \sum_{x \in \mathcal{U}\cap \mathcal{P}^{(\la)}} \mathds{1}_{R_{n,0}^{(\la)}(x) \ge r/4} \right]\\
			&\leq \int_{\mathcal{U}} \mathbb{P}(R_{n,0}^{(\la)}(x) \ge r/4) \mathrm{d}x\\
			&\leq \text{Vol}_d(\mathcal{U}) c_1 \exp(-c_2 r^{d+1}) \mathrm{d}x.\\
			&\leq c_1 \exp(-c_2r^{d+1}).
	\end{align*}
	This yields the result.
\end{proof}
The final result of this section is a slight analogue of Proposition \ref{stab k-faces} for the stabilization in height, i.e. we prove that with high probability the calculation of the score of $(0,h_0)$ inside the cylinder $C(r)$ does not depend on the points of the point process which are higher than $r^2$ up to a multiplicative constant. The statement of that result uses the following notation: for $w=(v,h)\in W_\la$, $r,h>0$, $n\ge 1$ and $k\in\{0,\ldots,d-1\}$, we denote by $\xi_{n,k,[r,h]}^{(\la)}(w)$ the quantity
$$\xi_{n,k,[r,h]}^{(\la)}(w,\cP^{(\la)})=\xi_{n,k}^{(\la)}(w,\cP^{(\la)}\cap C_v^{\le h}(r)).$$
Let us also recall the notation $\xi_{n,k,[r]}^{(\la)}(w,\cP^{(\la)})$ introduced at \eqref{eq:xinkr}.

\begin{lem}\label{localisation hauteur k-faces}
	For all $n \geq 1$ and there exists $\lambda_0, c > 0$ such that for all $h_0>0$, $\lambda \in [\lambda_0, +\infty]$ and $1\le r < \pi\la^{\frac1{d+1}}$, we have
	$$ \mathbb{P}\left(\xi_{n,k, [r]}^{(\lambda)}((0, h_0),\cP^{(\la)}) \neq \xi_{n,k, [r,l_n r^2]}^{(\lambda)}((0, h_0),\cP^{(\la)}) \right)
	\leq c_1  \exp(-c_2r^{d+1})$$
	with $l_n = \frac{1}{ 2^{2n+7}}$.
\end{lem}
\begin{proof}
	We denote by $A$ the event $\lbrace H_{n}((0,h_0), \cP^{(\la)}\cap C(r)) > l_n r^2 \rbrace$.
	Using Lemma \ref{stab hauteur} with the choice $r$ for $t$ and $l_n r^2$ for $r$, we get
	\begin{equation} \label{eq1 stab hauteur k-faces} \mathbb{P}(A) \leq c_1 \exp(-c_2r^{d+1}).\end{equation}
	On the event $A^c$, every paraboloid containing a facet going through
	 $(0,h_0)$ of $\partial \Phi_n(\cP^{(\la)}\cap C(r))$ 
	 has height smaller than $l_n r^2$.
	This implies that every point that shares a facet of $\partial \Phi_n(\cP^{(\la)}\cap C(r))$
	with $(0, h_0)$ is on the boundary of such a paraboloid and is consequently included in
	$C^{\leq l_n r^2}\left(\frac{r}{2^{n+1}}\right)$ by the same method as in \eqref{eq:probavidebasecase}. Using Lemma \ref{lemme stab rect}, we obtain that
	with probability larger than
	$1 - c_1 \exp(-c_2r^{d+1})$, any 
	$w \in C^{\leq l_n r^2}\left(\frac{r}{2^{n+1}}\right)$ verifies
	$\xi_{n,0, [r]}^{(\la)}(w, \cP^{(\la)}) = \xi_{n, 0, [r,l_n r^2]}^{(\la)}(w, \cP^{(\la)})$.
	Thus we deduce that with probability larger than $1 - c_1 \exp(-c_2r^{d+1})$, the 
	$k-$faces containing $w$ of $\partial \Phi_n(\mathcal{P}^{(\la)} \cap C(r))$ coincide with those of 
	$\partial \Phi_n(\mathcal{P}^{(\la)}\cap C^{\leq l_n r^2}(r))$.
	In other terms we have proved that
	\begin{equation}\label{eq2 stab hauteur k-faces}
		\mathbb{P}\left(\{\xi_{n,k, [r]}^{(\la)}((0, h_0),\cP^{(\la)}) \neq \xi_{n,k, [r,l_n]}^{(\la)}((0, h_0), \cP^{(\la)})\} \cap A^c \right)
			\leq c_1 \exp(-c_2r^{d+1}).
	\end{equation}
	Combining \eqref{eq1 stab hauteur k-faces} and \eqref{eq2 stab hauteur k-faces}, we obtain the result.
\end{proof}

\section{$L^p$ bounds and pointwise convergences}\label{sec:lpandconv}
\
In this section, we prove intermediary results on the functionals $\xi_{n,k}^{(\la)}$ and their variations which depend on the stabilization properties of Section \ref{sec:stab} and pave the way to the proofs of Theorems \ref{lim esperance precis} and \ref{lim variance precis}. More precisely, Lemma \ref{borne Lp} states some moment bounds, Lemma \ref{lem:convesp1} and Proposition \ref{expectation convergence xi k-faces} show in two steps the convergence of the expectation of $\xi_{n,k}^{(\la)}$ to the expectation of $\xi_{n,k}^{(\infty)}$. We then deduce similar results for covariances of scores in Proposition \ref{correlation convergence} and Lemma \ref{Lp bound correlation}.
\begin{lem}\label{borne Lp}
	For any $p \in \left[1, +\infty \right)$, there exist  constants $\lambda_0, c>0$  such that for any
	$(v,h) \in \mathbb{R}^{d-1}\times\R_+$  , $\lambda \in [\lambda_0, \infty]$ and $1 \leq r < \pi \la^{\frac{1}{d+1}}$
	\begin{equation}
	\mathbb{E}\left[ \xi_{n,k}^{(\la)}((v,h), \mathcal{P^{(\lambda)}}) ^p \right] \leq c_1 h^{pk} \exp(- c_2h^{\frac{d+1}{2}}),
	\end{equation}
	\begin{equation}\label{eq:moment2}
	\mathbb{E}\left[ \xi_{n,k, [r]}^{(\la)}((v,h), \mathcal{P^{(\lambda)}}) ^p \right] \leq 
	cr^{pk(d-1)+1}(h\vee r^2)^{pk},
	\end{equation}
	\begin{equation}\label{eq:moment3}
	\mathbb{E}\left[ \xi_{n,k, [r, l_n r^2]}^{(\la)}((v,h), \mathcal{P}^{(\lambda)}) ^p \right] \leq 
	cr^{pk(d+1)+1}.
	\end{equation}
	with $l_n = \frac{1}{ 2^{2n+7}}$.
\end{lem}
\begin{proof}
The method is almost identical to \cite[Lemma 4.4]{CY2}. For the sake of completeness and because the variants at \eqref{eq:moment2} and \eqref{eq:moment3} are new in comparison to \cite{CY2}, we provide as an example the proof of \eqref{eq:moment2}.
	
We assert that $\xi^{(\la)}_{n,k,[r]}((v,h),\cP^{(\la)})\overset{d}{=}\xi^{(\la)}_{n,k,[r]}((0,h),\cP^{(\la)})$ because of the rotation-inva\-riance of the initial model in the unit ball. 

	
	We introduce the variables $H := H_n^{(\la)}((0,h), \mathcal{P}^{(\la)}\cap C(r)))$ as defined at
	\eqref{def:H} and $R$ as the smallest $s>0$ such that $C(s)$ contains $\bigcup_{(v_1,h_1)\in\cP^{(\la)}\cap\cU(h,H)}C_{v_1}(R^{(\la)}_{n,k}(v_1,h_1))$ where $\cU(\cdot,\cdot)$ has been introduced at \eqref{eq:defEnsembleU}. We assert that the proof of Proposition \ref{stab k-faces} implies that 
	\begin{equation}\label{eq:stabkfacesamelioree}
	\P(R\ge r)\le c_1e^{-c_2r^{d+1}}.
	\end{equation}
	In particular, only the points in $C^{\le H}(R)$ can be part of a potential facet containing $(0,h)$ on $\partial \Phi_n^{(\la)}(\cP^{(\la)})$. This implies that
	$$\xi_{n,k}^{(\la)}((0,h), \mathcal{P}^{(\la)}) \leq \frac{1}{k+1} \binom{N}{k} $$
	where $N=\text{card}(C^{\le H}(R)\cap \cP^{(\la)}).$
Consequently, it is enough to show that there exist $c>0$ such that
	\begin{equation}\label{eq:momentN}
	\mathbb{E}[N^{pk}] \leq cr^{pk(d-1)+1}(h\vee r^2)^{pk}.
	\end{equation}
	We decompose the expectation in the following way:
	\begin{align}\label{eq:majENpk}
		\mathbb{E}[N^{pk}] &= 
			\sum_{i = 1}^{\lceil r \rceil} \sum_{j = \lfloor h\rfloor+1}^{\infty} \mathbb{E}[N^{pk} \mathds{1}_{(i-1) \leq R <i} \mathds{1}_{(j-1) \leq H < j}]\nonumber\\
			&\le \sum_{i = 1}^{\lceil r \rceil} \sum_{j =  \lfloor h\rfloor +1}^{\infty}\mathbb{E}[\text{card}(C^{\le j}(i)\cap \cP^{(\la)})^{pk} \mathds{1}_{(i-1) \leq R <i} \mathds{1}_{(j-1) \leq H < j}]\nonumber\\
			&\le \sum_{i = 1}^{\lceil r \rceil} \sum_{j =  \lfloor h\rfloor +1}^{\infty}\mathbb{E}[\text{card}(C^{\le j}(i)\cap \cP^{(\la)})^{3pk}]^{1/3}\mathbb{P}(R \geq i-1)^{1/3} \mathbb{P}(H \geq j-1)^{1/3}
			\end{align}
				where the last line is a consequence of H\"older's inequality. For any $i,j\ge 0$, we observe that $C^{\leq j}(i)$ has volume $c i^{d-1}j$ and thanks to \eqref{eq:intensity process}, the $d\cP^{(\la)}$-measure of $C^{\le j}(i)$ is bounded by its volume. Consequently, the variable $\text{card}(\mathcal{P} \cap C^{\leq j}(i))$ is stochastically dominated by a Poisson variable $\text{Po}(ci^{d-1}j)$. Combining this fact, the moment bound 	$\mathbb{E}[Po(\lambda)^r] \leq c \lambda^r$ for any $r \geq 1$ and \eqref{eq:majENpk}, we obtain
	\begin{align*}
		\mathbb{E}[N^{pk}] \leq \sum_{i = 1}^{\lceil r\rceil} \sum_{j =  \lfloor h\rfloor+1}^{\infty} ci^{pk(d-1)}j^{pk} 
			\mathbb{P}(R \geq (i-1))^{1/3} \mathbb{P}(H \geq (j-1))^{1/3}.
	\end{align*}
Let us assume that $h\le r^2$. We can now use Lemma \ref{stab hauteur} and \eqref{eq:stabkfacesamelioree}
	to get
		\begin{align*}
	\mathbb{E}[N^{pk}] &\leq \sum_{i = 1}^{\lceil r \rceil}  c_1i^{pk(d-1)} 
	e^{-c_2i^{d+1}} \left(\sum_{j = \lfloor h\rfloor+1}^{\lfloor r^2\rfloor} j^{pk}e^{-c_3j^{\frac{d+1}{2}}}+\sum_{j=\lfloor r^2\rfloor+1}^{\infty} j^{pk} 
	e^{-c_4\sqrt{j}r^d}\right)\\
	&\leq c_1 r^{pk(d-1)+1}r^{2pk}\left(e^{-c_2h^{\frac{d+1}{2}}}+e^{-c_3r^{d+1}}\right).
	\end{align*}
	This shows \eqref{eq:momentN} and consequently \eqref{eq:moment2}. We proceed in the same way when $h\ge r^2$. 
\end{proof}

The next lemma is a first step towards the convergence of the expectation of $\xi_{n,k}$ stated in Proposition \ref{expectation convergence xi k-faces}. It relies on the application of the continuous mapping theorem to the intersection of the point process $\cP^{(\la)}$ with the compact set $C^{\le l_n r^2}(r)$.
\begin{lem}\label{lem:convesp1}
	There exists $c > 0$ such that for all $h_0 \ge 0$ and $r \ge 1$ we have
	\begin{align*} \lsup\limits_{\lambda \rightarrow +\infty} |\mathbb{E}[\xi_{n,k, [r]}^{(\lambda)}((0, h_0), \mathcal{P}^{(\lambda)})] - \mathbb{E}[\xi_{n,k, [r]}^{(\infty)}((0,h_0), \mathcal{P})] | \leq c_1  r^{k(d-1) + 1/2} (h_0 \vee r^2)^k e^{-c_2 r^{d+1}}  . 
	\end{align*}
	 
\end{lem}
\begin{proof}
For sake of simplicity, we use the following abbreviations:
$$\xi_{n,k, [r]}^{(\lambda)}:=\xi_{n,k, [r]}^{(\lambda)}((0,h_0), \mathcal{P}^{(\lambda)})\; \mbox{ and }\;\xi_{n,k, [r]}^{(\infty)}:=\xi_{n,k, [r]}^{(\infty)}((0,h_0), \mathcal{P}).$$
Recalling $l_n = \frac{1}{  2^{2n+7}}$, we use similar notations $\xi_{n,k, [r, l_n r^2]}^{(\lambda)}$ and $\xi_{n,k, [r, l_n r^2]}^{(\infty)}$ and obtain 
	\begin{align}\label{eq:lem convesp1 eq1}
		&\hspace*{-0.25cm}\bigl\vert\mathbb{E}[\xi_{n,k, [r]}^{(\lambda)}] - \mathbb{E}[\xi_{n,k, [r]}^{(\infty)}]\bigl| \nonumber
		 \\&\hspace*{-0.25cm}\leq \bigl|\mathbb{E}[\xi_{n,k, [r]}^{(\lambda)}
			 - \xi_{n,k, [r, l_n r^2]}^{(\lambda)}]\bigl|
			 + \bigl|\mathbb{E}[\xi_{n,k, [r, l_n r^2]}^{(\lambda)}]
			 - \mathbb{E}[\xi_{n,k, [r, l_n r^2]}^{(\infty)}]\bigl|
			  + \bigl|\mathbb{E}[\xi_{n,k, [r, l_n r^2]}^{(\infty)}
			 - \xi_{n,k, [r]}^{(\infty)}]\bigl\vert .
	\end{align}
	We start by bounding the first term in the rhs of \eqref{eq:lem convesp1 eq1}. Using the Cauchy-Schwarz inequality, we get
	\begin{align*}
	 \bigl\vert\E[ \xi_{n,k, [r]}^{(\lambda)} - \xi_{n,k, [r, l_n r^2]}^{(\la)} ]\bigl\vert 
	 &= \bigl\vert\E[ (\xi_{n,k, [r]}^{(\lambda)} - \xi_{n,k, [r, l_n r^2]}^{(\la)}) \mathds{1}_{\{\xi_{n,k, [r]}^{(\lambda)} \neq \xi_{n,k, [r, l_n r^2]}^{(\lambda)}\}} ]\bigl\vert
	\\ &\leq  \E[ (\xi_{n,k, [r]}^{(\lambda)} - \xi_{n,k, [r, l_n r^2]}^{(\la)})^2]^{1/2} \P (\xi_{n,k, [r]}^{(\lambda)} \neq \xi_{n,k, [r, l_n r^2]}^{(\lambda)} )^{1/2}.
	\end{align*}
	Combining Lemma \ref{localisation hauteur k-faces} and Lemma \ref{borne Lp}, we obtain for $\la$ large enough
	\begin{equation}\label{eq:boundterm1_37}
	\bigl\vert\E[ \xi_{n,k, [r]}^{(\lambda)} - \xi_{n,k, [r, l_n r^2]}^{(\la)} ]\bigl\vert\le 	 c_1 r^{k(d-1) + 1/2} (h_0 \vee r^2)^k \exp(-c_2 r^{d+1})  .
	\end{equation}
	In the same way, we bound the third term in the rhs of \eqref{eq:lem convesp1 eq1} to get 
	\begin{equation}\label{eq:boundterm3_37}
	\bigl\vert\E[ \xi_{n,k, [r]}^{(\infty)} -\xi_{n,k, [r, l_n r^2]}^{(\infty)} ]\bigl\vert\le 	 c_1 r^{k(d-1) + 1/2} (h_0 \vee r^2)^k \exp(-c_2 r^{d+1} )  .
	\end{equation}
	Let us now prove that
	\begin{equation}\label{eq:cvterm2_37}
	 \lim\limits_{\lambda \rightarrow +\infty} \mathbb{E}[\xi_{n,k, [r, l_n r^2]}^{(\lambda)}] = \E[\xi_{n,k, [r, l_n r^2]}^{(\infty)}].
	\end{equation}
	 To do so, we first prove the convergence in distribution of $\xi_{n,k, [r, l_n r^2]}^{(\lambda)}$
	to $\xi_{n,k, [r, l_n r^2]}^{(\infty)}$.
	We denote by $\mathcal{X}(r, l_n r^2)$ the set  of finite points sets in $C^{\leq l_n r^2}(r)$ and we endow it with the discrete topology.
	A sequence of point sets $(\eta_i)_i$ of $\mathcal{X}(r, l_n r^2)$ converges to a point set $\eta$ if 
	$\eta_i = \eta$ for all $i \geq i_0$ for some $i_0 \in \mathbb{N}$.
	
	We define for all $w \in W_\lambda$ and 
	$\eta \in \mathcal{X}(r, l_n r^2)$
	$$g_{n,\lambda}(w, \eta) := \xi_{n,k, [r, l_n r^2]}^{(\lambda)}(w, \eta).$$
	Considering that $\mathcal{P}^{(\lambda)} \rightarrow \mathcal{P}$ %
	 in distribution in $C^{\leq l_n r^2}(r)$ by Lemma \ref{convergence rescaled}, we intend to show the convergence in distribution $g_{n,\lambda}(w_0, \mathcal{P}^{(\lambda)}) \rightarrow g_{n, \infty}(w_0, \mathcal{P})$ by using \cite[Theorem 5.5, p. 34]{B68}. To do so, we observe that since $\mathcal{P}$ is in general position with probability
	$1$, it is enough to prove 
	that for any $\eta \in \mathcal{X}(r,l)$ in general position,
	$g_{n, \lambda}(w_0, \eta) = g_{n, \infty}(w_0, \eta)$ for all $\lambda$ large enough. 
	
	Let $\eta \in \mathcal{X}(r, l_n r^2)$ be in general position. We explain here why for each point $w_1$ of $\eta$, the number of the layer containing $w_1$ and the local structure of that layer around $w_1$ are fixed for $\la$ large enough. We start by considering an extreme point $w_1$ for the parabolic hull peeling of
	$\eta$. We choose a $(d-1)$-dimensional parabolic face containing $w_1$ which is generated by the extreme points $w_1, w_2,\ldots,w_d$ and we also denote by $\Pi^\downarrow(v,h)$ the downward paraboloid containing that face on its boundary. For $\varepsilon>0$ small enough, the intersection of  
	$\Pi^\downarrow(v, h + \varepsilon)$ with $\eta$ is $\{w_1,\cdots,w_d\}$.
	Since the quasi-paraboloids converge to the real paraboloids as $\lambda$ goes to infinity, for $\lambda$
	large enough the quasi-paraboloid with $w_1, \ldots, w_d$ on its boundary is contained in $\Pi^\downarrow(v, h + \varepsilon)$ and does not contain any point of $\eta$ in its interior. This means that for $\lambda$ large enough,
	$w_1$ is extreme and the facets containing it are generated by the same points. Applying this to every extreme point, we deduce that the first layer is stable for $\la$ large enough. By induction on the number of the layer, we prove similarly that the subsequent layers of the quasi-parabolic hull peeling of $\eta$ are stable for $\lambda$ large enough.
	This means that $g_{n, \lambda}(w_0, \eta) = g_{n, \infty}(w_0, \eta)$ for all $\lambda$ large enough and as a consequence completes the proof of the convergence in distribution of $\xi_{n,k, [r, l_n r^2]}^{(\lambda)}(w_0, \mathcal{P}^{(\lambda)})$
	to $\xi_{n,k, [r, l_n r^2]}^{(\infty)}(w_0, \mathcal{P})$.
This extends to \eqref{eq:cvterm2_37} by
	\cite[Theorem 5.4]{B68} since the considered sequence is bounded in $L^p$ with $p > 1$ thanks to Lemma \ref{borne Lp}.
	
	Inserting the results \eqref{eq:boundterm1_37}, \eqref{eq:boundterm3_37} and \eqref{eq:cvterm2_37} into \eqref{eq:lem convesp1 eq1}, we deduce Lemma \ref{lem:convesp1}.
\end{proof}
We are now ready to state the required convergence in expectation in Proposition \ref{expectation convergence xi k-faces} below.

\begin{prop}\label{expectation convergence xi k-faces}
	For any $n \geq 1$ and all $h_0 > 0$ we have 
	$$ \lim\limits_{\lambda \rightarrow +\infty } \mathbb{E}\left[\xi_{n,k}^{(\lambda)}((0, h_0), \mathcal{P}^{(\lambda)})\right]
	= \mathbb{E}\left[\xi_{n,k}^{(\infty)}((0, h_0), \mathcal{P})\right]. $$
\end{prop}
\begin{proof} With the same notation as in the proof of Lemma \ref{lem:convesp1}, we get from the triangle inequality that
	\begin{align}\label{eq1 lemme lim expectation kfaces}
	\big\vert\mathbb{E}[\xi_{n,k}^{(\lambda)}] - \mathbb{E}[ \xi_{n,k}^{(\infty)}]\big|
	\leq& \big|\mathbb{E}[\xi_{n,k}^{(\lambda)}]
	- \mathbb{E}[ \xi_{n,k, [r]}^{(\lambda)}]\big|
	+ \big|\mathbb{E}[\xi_{n,k, [r]}^{(\lambda)}]
	- \mathbb{E}[\xi_{n,k, [r]}^{(\infty)}]\big|
	+ \big|\mathbb{E}[\xi_{n,k, [r]}^{(\infty)}]
	- \mathbb{E}[ \xi_{n,k}^{(\infty)}]\big\vert.
	\end{align}

	Using the same method as in the proof of \eqref{eq:boundterm1_37}, we obtain thanks to Proposition \ref{stab k-faces} and Lemma \ref{borne Lp} that for $r$ large enough,
	\begin{equation}\label{eq:1er3emetermes}
	\max(\big|\mathbb{E}[\xi_{n,k}^{(\lambda)}]
	- \mathbb{E}[ \xi_{n,k, [r]}^{(\lambda)}]\big|, \big|\mathbb{E}[\xi_{n,k, [r]}^{(\infty)}]
	- \mathbb{E}[ \xi_{n,k}^{(\infty)}]\big\vert)\le
	r^{k(d-1)+\frac12+2k}e^{-c r^{\frac{d+1}{2}}}.
	\end{equation}
	Inserting \eqref{eq:1er3emetermes} and the result of Lemma \ref{lem:convesp1} into \eqref{eq1 lemme lim expectation kfaces}, we obtain that for any $\varepsilon>0$, there exists $r$ large enough such that $\lsup_{\la\to\infty}\big\vert\mathbb{E}[\xi_{n,k}^{(\lambda)}] - \mathbb{E}[ \xi_{n,k}^{(\infty)}]\big|\le \varepsilon$. This completes the proof of Proposition \ref{expectation convergence xi k-faces}.
\end{proof}
In view of the required variance estimates, we now aim at extending Proposition \ref{expectation convergence xi k-faces} when the expectation of a score is replaced by the covariance of two scores $\xi_{n,k}^{(\la)}((0,h_0),\cP^{(\la)})$ and $\xi_{n,k}^{(\la)}((v_1,h_1),\cP^{(\la)})$ at two points $(0,h_0)$ and $(v_1,h_1)$ belonging to $W_\la$. To do so, we study the two-point correlation function defined at \eqref{eq:defcnkla}.

\begin{prop}\label{correlation convergence}
For all $h_0 \ge 0$ and $(v_1, h_1) \in \mathbb{R}^{d-1} \times \mathbb{R}_+$ we have 
	$$ \lim\limits_{\lambda \rightarrow \infty} c_{n,k}^{(\lambda)}((0, h_0), (v_1, h_1)) =
	c_{n,k}^{(\infty)}((0, h_0), (v_1, h_1))   .$$ 
\end{prop}
\begin{proof}
For sake of simplicity, let us use the following reduced notation for $\la\in [\la_0,\infty]$:
\begin{equation}\label{eq:reducednotation}
\xi^{(\la)}(0):=\xi_{n,k}^{(\lambda)}\left((0, h_0), \mathcal{P}^{(\lambda)} \cup \lbrace (v_1, h_1) \rbrace
	\right),\xi^{(\la)}(1):=\xi_{n,k}^{(\la)}\left((v_1, h_1), \mathcal{P}\cup\{(0,h_0)\} \right).   
\end{equation}
	By Proposition \ref{expectation convergence xi k-faces}, we get 
	\begin{align*} 
\lim\limits_{\lambda \rightarrow \infty} 
	\mathbb{E}\left[\xi_{n,k}^{(\lambda)}\left((0, h_0), \mathcal{P}^{(\lambda)}   \right)  \right]
	\mathbb{E}\left[\xi_{n,k}^{(\lambda)}\left((v_1, h_1), \mathcal{P}^{(\lambda)}   \right)  \right]&
	= \mathbb{E}\left[\xi_{n,k}^{(\infty)} \left((0, h_0), \mathcal{P}  \right)  \right]
	\mathbb{E}\left[\xi_{n,k}^{(\infty)}\left((v_1, h_1), \mathcal{P} \right)  \right]&
	\end{align*}
	so we only need to prove 
	\begin{align}\label{eq:resultconvcrossed} 
	\lim\limits_{\lambda \rightarrow \infty} \mathbb{E}\left[\xi^{(\lambda)}(0)
	 \xi^{(\lambda)}(1)  \right]
	 = \mathbb{E}\left[\xi^{(\infty)}(0)
	 \xi^{(\infty)}(1)\right].
	\end{align}
	
	Let us fix $\varepsilon>0$. We can show that for $r$ large enough
	\begin{align}
	\lsup_{\la\to\infty}\big|\E\left[\xi^{(\lambda)}(0)
	 \xi^{(\lambda)}(1)  \right]-\E[\left[\xi_{[r]}^{(\lambda)}(0)
	 \xi_{[r]}^{(\lambda)}(1)  \right] \big| &\le \varepsilon,\label{eq:groscalcul1}\\
	 \lsup_{\la\to\infty}\big|\E\left[\xi_{[r]}^{(\lambda)}(0)
	 \xi_{[r]}^{(\lambda)}(1)  \right]-\E[\left[\xi_{_{[r,l_nr^2]}}^{(\lambda)}(0)
	 \xi_{[r,l_nr^2]}^{(\lambda)}(1)  \right] \big| &\le \varepsilon,\label{eq:groscalcul2}\\
	 \lim_{\la\to\infty}\big|\E[\left[\xi_{_{[r,l_nr^2]}}^{(\lambda)}(0)
	 \xi_{[r,l_nr^2]}^{(\lambda)}(1)  \right]-\E[\left[\xi_{_{[r,l_nr^2]}}^{(\infty)}(0)
	 \xi_{[r,l_nr^2]}^{(\infty)}(1)  \right]\big|&=0,\label{eq:groscalcul3}\\
	 \big|\E\left[\xi_{[r]}^{(\infty)}(0)
	 \xi_{[r]}^{(\infty)}(1)  \right]-\E[\left[\xi_{_{[r,l_nr^2]}}^{(\infty)}(0)
	 \xi_{[r,l_nr^2]}^{(\infty)}(1)  \right] \big| &\le \varepsilon,\label{eq:groscalcul4}\\
	  \big|\E\left[\xi^{(\infty)}(0)
	 \xi^{(\infty)}(1)  \right]-\E[\left[\xi_{[r]}^{(\infty)}(0)
	 \xi_{[r]}^{(\infty)}(1)  \right] \big| &\le \varepsilon\label{eq:groscalcul5}
	\end{align}
	where we have used the reduced notation for $\xi_{n,k,[r]}^{(\la)}$ and $\xi^{(\la)}_{n,k,[r,l_n r^2]}$ and their counterparts in the limit model similar to the one introduced at \eqref{eq:reducednotation}.
	These five assertions imply \eqref{eq:resultconvcrossed} so it is enough to show each of them.  We claim that \eqref{eq:groscalcul1}, \eqref{eq:groscalcul2}, \eqref{eq:groscalcul4} and \eqref{eq:groscalcul5} are obtained by similar methods. Consequently we omit the proofs of  \eqref{eq:groscalcul2}, \eqref{eq:groscalcul4} and \eqref{eq:groscalcul5} and concentrate on getting \eqref{eq:groscalcul1}.
	\begin{align}\label{eq:demo4.16}
	&\big|\E\left[\xi^{(\lambda)}(0)
	 \xi^{(\lambda)}(1)  \right]-\E[\left[\xi_{[r]}^{(\lambda)}(0)
	 \xi_{[r]}^{(\lambda)}(1)  \right] \big|\notag\\&\le 
	 \big|\E\left[(\xi^{(\lambda)}(0)-\xi_{[r]}^{(\lambda)}(0))
	 \xi^{(\lambda)}(1)  \right] \big|+\big|\E\left[\xi_{[r]}^{(\lambda)}(0)(
	 \xi^{(\lambda)}(1)-  
	 \xi_{[r]}^{(\lambda)}(1))  \right] \big|.
	 \end{align}
	 We derive an upper bound for the first term in the rhs of \eqref{eq:demo4.16}, the second term being treated identically. Using H\"older's inequality, we obtain
	 \begin{align*}
	 &\big|\E\left[(\xi^{(\lambda)}(0)-\xi_{[r]}^{(\lambda)}(0))
	 \xi^{(\lambda)}(1)  \right] \big|\\&=\big|\E\left[(\xi^{(\lambda)}(0)-\xi_{[r]}^{(\lambda)}(0))
	 \xi^{(\lambda)}(1)\mathds{1}_{\{\xi^{(\la)}(0)\ne \xi_{[r]}^{(\la)}(0)\}}  \right] \big|\\
	 &\le \E[|\xi^{(\lambda)}(0)-\xi_{[r]}^{(\lambda)}(0)|^3]^{1/3}\E[\xi^{(\lambda)}(1)^3]^{1/3}\P(\xi^{(\la)}(0)\ne \xi_{[r]}^{(\la)}(0))^{1/3}.
	\end{align*}
	The estimate \eqref{eq:groscalcul1} is then a consequence of Proposition \ref{stab k-faces} and Lemma \ref{borne Lp}. 
	The convergence \eqref{eq:groscalcul3} is derived analogously to \eqref{eq:cvterm2_37} by using \cite[Theorem 5.4]{B68} and Lemma \ref{borne Lp}. This proves Proposition \ref{correlation convergence}.
	\end{proof}
Lemma \ref{Lp bound correlation} below asserts that the correlation function $c_{n,k}^{(\la)}$ decays exponentially fast as a function of the heights of the two points and of the distance between them. This means in particular that the scores get closer to being independent as the distance between the points increases. The proof of Lemma \ref{Lp bound correlation} relies on Proposition \ref{stab k-faces} and Lemma \ref{borne Lp} and is identical to the proof of \cite[Lemma 4.8]{CY2}. For that reason, it is omitted.
\begin{lem}\label{Lp bound correlation}
	For all $n \geq 1$, there exist  $c_1,\ldots,c_4>0$ such that for any $h > 0$, $(v_1, h_1) \in \mathbb{R}^{d-1} \times \mathbb{R}_+$, $p \geq 1$ and $\lambda \in \left[1, \infty \right]$ we have 
	$$ |c_{n,k}^{(\lambda)}((0, h_0), (v_1, h_1))| \leq c_1 h_0^{c_2} h_1^{c_3} \exp\left(-c_4 \left( \Vert v_1 \Vert^{d+1} + h_0^{\frac{d+1}{2}} + h_1^{\frac{d+1}{2}} \right) \right)
	.$$ 
\end{lem}

\section{Proofs of the main results}\label{sec:proofs}
In this section, Theorems \ref{lim esperance precis}--\ref{lim variance precis vol} are proved. We recall that they include the statements on the limiting expectations and variances of Theorems \ref{thm:principalintro} and \ref{thm:principalintrovol}. We have chosen to omit the proof of the central limit theorems as it is 
    the exact replicate line by line of \cite[pages 93--98]{CSY} when the convex hull is replaced by the $n$-th layer of the peeling.
    
\subsection{Results on $k$-dimensional faces}
\begin{proof}[Proof of Theorem \ref{lim esperance precis}.]

~\\~\\{\textit{Proof of the convergence of the normalized expectation.}}
	The first step consists in taking the expectation in \eqref{eq:decompavantrescaling} and using the Mecke formula to get  	\begin{equation*}
		\mathbb{E}[N_{n,k,\lambda}] = \lambda \int_{\mathbb{B}^d} \mathbb{E}[\xi_{n,k}(x, \mathcal{P}_\lambda)] \mathrm{d}x .
	\end{equation*}
	Let us introduce $e_d=(0, \ldots, 0, 1)$. By rotation-invariance of  $\mathbb{E}[\xi_{n,k}(x, \mathcal{P}_\lambda)]$, we obtain
	$$\mathbb{E}[\xi_{n, k}(x, \mathcal{P}_\lambda)] = \mathbb{E}[\xi_{n,k}(|x| e_d, \mathcal{P}_\lambda)].$$
	Then we apply a change of coordinates in spherical coordinates to get 
	\begin{equation}\label{eq change spherical} \mathbb{E}[N_{n, k,\lambda}] = \lambda \int_{\mathbb{S}^{d-1}} \int_{0}^{1}
		 \mathbb{E}[\xi_{n,k}(r e_d, \mathcal{P}_\lambda)] r^{d-1} \mathrm{d}r \mathrm{d}\sigma_{d-1}(u), 
	\end{equation}
	where $\sigma_{d-1}$ is the unnormalized area measure on $\mathbb{S}^{d-1}$.
	
Recall that for every $h > 0$,
	$\xi_{n,k}^{(\la)}((0,h), \cP^{(\la)}) = \xi_{n,k}\left([T^{(\la)}]^{-1}((0,h)), \cP_\la\right),$ see \eqref{eq:egalie scores}.
	Consequently,
	 an application of the change of variables $h = \lambda^{\frac{2}{d+1}}(1-r)$ in \eqref{eq change spherical} leads to
	\begin{align}\label{eq:espinterm}
		\mathbb{E}[N_{n, k,\lambda}] &= \lambda \int_{\mathbb{S}^{d-1}} \int_{0}^{\lambda^{\frac{2}{d+1}}}
		\mathbb{E}[\xi_{n,k}^{(\lambda)}((0,h), \mathcal{P}^{(\lambda)})] \lambda^{-\frac{2}{d+1}} (1 - \lambda^{-\frac{2}{d+1}}h)^{d-1}  \mathrm{d}h \mathrm{d}\sigma_{d-1}(u)\nonumber\\
		&= \lambda^{\frac{d-1}{d+1}} \int_{\mathbb{S}^{d-1}} \int_{0}^{\lambda^{\frac{2}{d+1}}}
		\mathbb{E}[\xi_{n,k}^{(\lambda)}((0,h), \mathcal{P}^{(\lambda)})] (1 - \lambda^{-\frac{2}{d+1}}h)^{d-1}  \mathrm{d}h \mathrm{d}\sigma_{d-1}(u).
	\end{align}
	We now wish to use the dominated convergence theorem.
	Lemma \ref{expectation convergence xi k-faces} implies that
	for any $h > 0$
	$$\mathbb{E}[\xi_{n,k}^{(\lambda)}((0,h), \mathcal{P}^{(\lambda)})] (1 - \lambda^{-\frac{2}{d+1}}h)^{d-1}\mathds{1}_{0 \leq h \leq \lambda^{\frac{2}{d+1}}}
		\xrightarrow[\lambda \rightarrow +\infty]{}  \mathbb{E}[\xi_{n,k}^{(\infty)}((0,h), \mathcal{P})].$$
	Thanks to Lemma \ref{borne Lp}, the integrand in the rhs of \eqref{eq:espinterm} is bounded from above by an integrable function of $h$, i.e.
	$$\mathbb{E}[\xi_{n,k}^{(\lambda)}((0,h), \mathcal{P}^{(\lambda)})] (1 - \lambda^{-\frac{2}{d+1}}h)^{d-1}\mathds{1}_{0 \leq h \leq \lambda^{\frac{2}{d+1}}} \leq c_1 \exp(-c_2 h^{\frac{d+1}{2}}).$$
	Applying the dominated convergence theorem in \eqref{eq:espinterm} shows that 
	\begin{equation}\label{eq:limitedelesp}
	\lim\limits_{\lambda \rightarrow \infty} {\lambda^{-\frac{d-1}{d+1}}} {\mathbb{E}[N_{n,k,\lambda}]}
	= \text{\normalfont Vol}_{d-1}(\mathbb{S}^{d-1}) \int_{0}^{\infty} \mathbb{E}[\xi_{n,k}^{(\infty)}((0,h), \mathcal{P})] \mathrm{d}h < \infty .
	\end{equation}

\noindent\textit{Proof of the positivity of the limiting expectation.} Noticing that $$\mathbb{E}[\xi_{n,k}^{(\infty)}((0,h), \mathcal{P})]\ge \frac{1}{k+1}\mathbb{P}\left( (0,h_0) \in \partial \Phi_n^{(\infty)}(\mathcal{P})  \right),$$
	we deduce that the positivity of the rhs of \eqref{eq:limitedelesp} is a consequence of
	Lemma \ref{lem proba etre sur une couche} below.
\end{proof}

\begin{lem}\label{lem proba etre sur une couche}
	For any $(0, h_0)$ with $h_0 > 0$ and any $n \geq 1$ : $$\mathbb{P}\left( (0,h_0) \in \partial \Phi_n^{(\infty)}(\mathcal{P})  \right) \neq 0.$$
\end{lem}
\begin{proof} 
	Let us write $w_0:= (0, h_0)$ and put $\cT_0:=\{w_0\}$. Our first step is purely deterministic and consists in constructing for $n=2$ an idealized point configuration which  puts $w_0$ on the second layer. We then extend the procedure for every $n>2$ to get a point set which puts $w_0$ on the $n$-th layer. In a second step, we introduce randomness and show that this property is stable with respect to a small random perturbation of the configuration.
	
	For each $i \in \{1, \ldots, d - 1 \}$, we write $w_{(i, +)} = ( \frac{\sqrt{2h_0}}{2} e_i, \frac{h_0}{8})$ and 
 	$w_{(i,-)} = ( -\frac{\sqrt{2h_0}}{2} e_i, \frac{h_0}{8})$.
	A direct calculation shows that for any $i$, $\Pi^{\downarrow}(w_{(i, +)}) \subseteq \Pi^{\downarrow}(0,h_0)$,
	$\Pi^{\downarrow}(w_{(i, -)}) \subseteq \Pi^{\downarrow}(w_0)$ and $\Pi^{\downarrow}(w_{(i, s_1)}) \cap \Pi^{\downarrow}(w_{(j, s_2)}) = \varnothing$ for any $i,j \in \{1, \ldots, d-1\}$, $s_1, s_2 \in \{+, -\}$ such that
	$(i, s_1) \neq (j, s_2)$. 
	
	Let us consider the deterministic point set $\cT_1:=\{w_0\}\cup\{w_{(i,s)} : i \in \{1,\ldots,d-1\}, s \in \{+,-\}\}$ and describe the peeling of $\cT_1$.
	The points $w_{(i, s)}$ are on the first layer
    because each $\Pi^\downarrow(w_{(i, s)})$ is empty. 
	Any downward paraboloid whose boundary goes through $w_0$ contains at least half of $\Pi^\downarrow(w_0)$ by Lemma \ref{lem:demiparabole} so it contains at least 
	one of the $w_{(i,s)}$. This implies in turn that $w_0$ is not on the first layer. Furthermore $\Pi^\downarrow(w_0)$ only contains points on the first layer so
	$w_0$ is on the second layer.
	
	We are going to iterate this construction by induction to obtain for every $n>2$ an extended deterministic point configuration $\cT_{n-1}=\{w_a : a\in \cup_{l=0}^{n-1} (\{1,\ldots,d-1\}\times\{+,-\})^{l}\}$, with the convention $(\cdot)^0=\{0\}$, which guarantees that $w_0$ is on the $n$-th layer of the peeling of $\cT_{n-1}$ and that for any $1\le l\le (n-1)$, $\cT_l\setminus \cT_{l-1}$ is included on its $(n-l)$-th layer. To do so, let us assume that we have constructed $\cT_{n-2}$ with that property.
	Let us fix $a\in (\{1,\ldots,d-1\}\times\{+,-\})^{n-2}$ and write $w_a = (v,h)$. We define  $w_{a, (i,+)} := (v + \frac{1}{2} \sqrt{2\frac{h_0}{8^{n-2}}} e_i, \frac{h}{8^{n-1}})$ and
	$w_{a, (i,-)} := (v - \frac{1}{2} \sqrt{2\frac{h_0}{8^{n-2}}} e_i, \frac{h}{8^{n-1}})$.
	\begin{figure}
		\centering
		\includegraphics[trim=0.5cm 0cm 0.5cm 0cm, clip,width=0.45\textwidth]{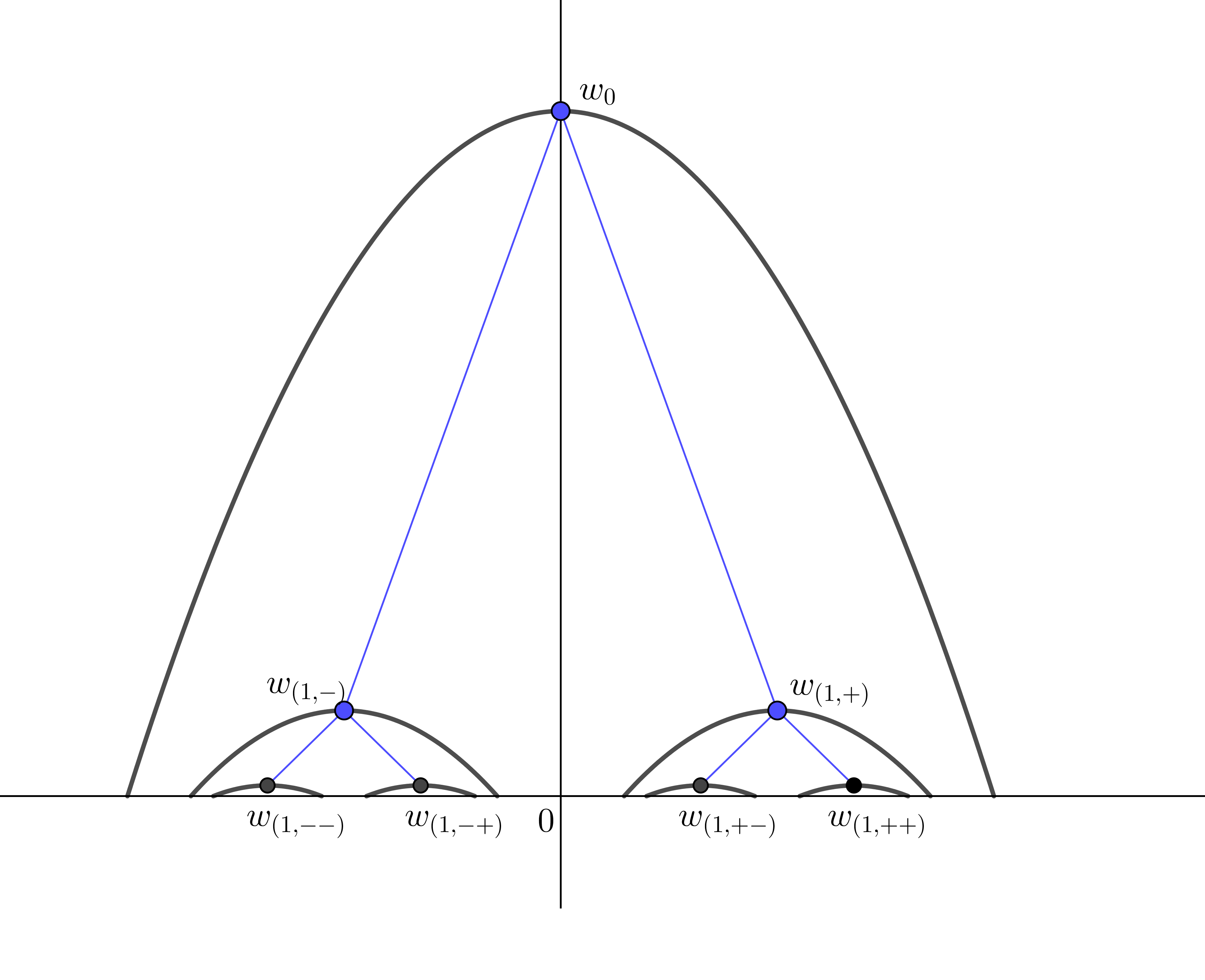}
		\includegraphics[trim=6cm 1.7cm 6cm 3.8cm, clip, width=0.54\textwidth]{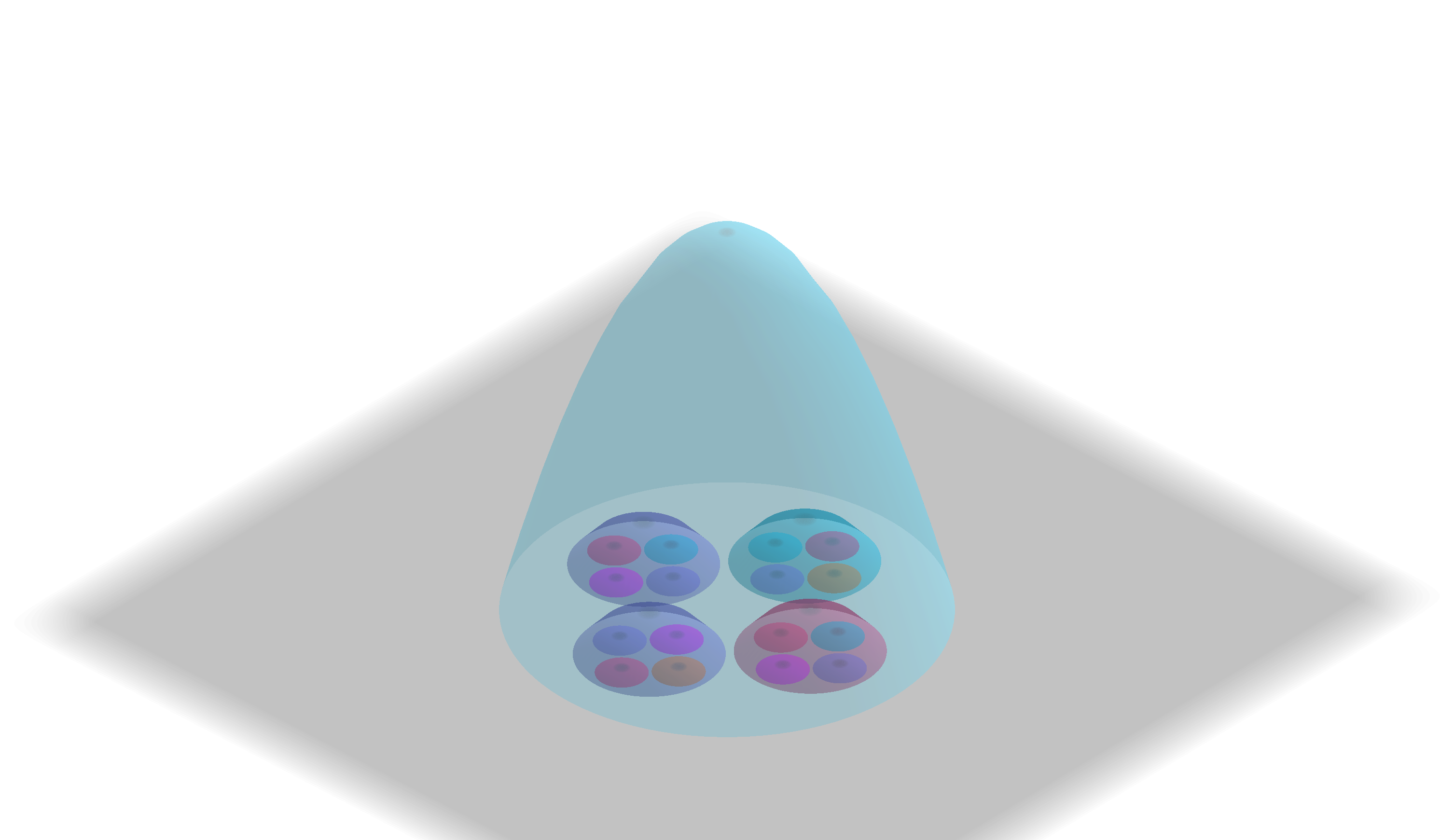}
		\caption{The tree of Lemma \ref{lem proba etre sur une couche} for $n = 3$ and $d=2,3$. 
		}
		\label{arbre}
	\end{figure}

	The construction is just a rescaling of the situation described for the case $n=2$ with $w_a$ playing the role of $w_0$ and
	$w_{a, (i,s)}$ that of $w_{(i, s)}$. Consequently, we get the same properties, i.e.
	for any $(i, s)$ 
	\begin{equation}\label{prop1 non nul}\Pi^{\downarrow}(w_{a, (i,s)}) \subseteq \Pi^{\downarrow}(w_a)\end{equation}
	and for $(i,s_1) \neq (j, s_2)$
	\begin{equation}\label{prop3 non nul}\Pi^{\downarrow}(w_{a, (i, s_1)}) \cap \Pi^{\downarrow}(w_{a,(j, s_2)}) = \varnothing . \end{equation}
	If we put an edge from $w_a$ to each $w_{a, (i,s)}$ then we endow $\cT_{n-1}$ with a structure of tree with root $w_0$, see Figure \ref{arbre} in the case $n=3$ in dimension $2$ and $3$. We also take the convention $\cT_{-1}=\varnothing$. 
	
	The construction has been done such that for every $l\in\{0,\ldots,(n-1)\}$, each point $w_a\in (\cT_l\setminus \cT_{l-1})$ is on the $(n-l)$-th layer of the peeling of $ {\mathcal T}_{n-1}$. 
	
	We prove now by downward induction the stability of the property above when the idealized point set is subject to a small random perturbation. Let us fix $\varepsilon>0$ and consider the event
		$$A_n=\cap_{w_a\in {\mathcal T}_{n-1}\setminus\{w_0\}}\{{\#(\mathcal P}\cap B(w_a,\varepsilon))=1\}\cap \{{\mathcal P}\cap (\Pi^{\downarrow}(w_0)\setminus \cup_{w_a\in {\mathcal T}_{n-1}} B(w_a,\varepsilon))=\varnothing\}.$$
	On the event $A_n$, for every $w_a\in {\mathcal T}_{n-1}$ we denote by $\tilde{w}_a$ the unique point belonging to ${\mathcal P}\cap B(w_a,\varepsilon)$ and call it a perturbed point from the tree (with same depth as the original point $w_a$). We then assert and prove below that the stability occurs, i.e. that the perturbed point $\tilde{w}_a$ is on the $(n-l)$-th layer of the peeling of ${\mathcal P}\cup\{w_0\}$ as soon as the original point $w_a$ is at depth $l$ in ${\mathcal T}_{n-1}$.
	
	We make the following preliminary observation: the calibrations in the construction of the tree ${\mathcal T}_{n-1}$ allow us to strengthen \eqref{prop1 non nul} and \eqref{prop3 non nul} by claiming that for $\varepsilon$ small enough, the $\varepsilon$-neighborhood of $\Pi^{\downarrow}(w_a)$ is included in $\Pi^{\downarrow}(w_b)$ for any $w_a\in \cT_{n-1}$ with parent $w_b$ and that the $\varepsilon$-neighborhoods of the downward parabolas associated with two distinct perturbed points at same depth do not meet. 
	
	We start with our base case, which is depth $(n-1)$. For any $a\in(\{1,\ldots, d-1\} \times \{ +, - \})^{n-1}$, thanks to the observation above, the perturbed point $\tilde{w}_a$ satisfies that $\Pi^{\downarrow}(\tilde{w}_a)$ does not intersect ${\mathcal P}$ when $\varepsilon$ is small enough. This shows that all perturbed points at depth $(n-1)$ are extreme.
	
	Now we assume that for some $l \in \{1, \ldots, n-1\}$, each perturbed point at depth $k$ is on the $(n-k)$-th layer for every $l\le k\le (n-1)$.
	We consider a perturbed point $\tilde{w}_a$ at depth $(l-1)$, which means $a \in(\{1, \ldots, d-1\} \times \{+, - \})^{(l-1)}$. Because of the preliminary observation, for $\varepsilon$ small enough, the downward paraboloid $\Pi^{\downarrow}(\tilde{w}_a)$ only contains points that can be written $\tilde{w}_{a,b}$ for some $b \in (\{1, \ldots, d-1\} \times \{+, - \})^p$
	and $1\le p\le (n-l)$.  Thus $\tilde{w}_a$ is at most
	on the $(n - l +1)$-th layer. 
	Moreover, noticing that any downward paraboloid  whose boundary goes through $\tilde{w}_a$ contains at least 
	one of the $\tilde{w}_{a,(i,s)}$, which are on the $(n-l)$-th layer. This shows that $\tilde{w}_a$ is on the $(n-l+1)-st$ layer.
	A finite induction on $l$ thus gives us that on the event $A_n$, $w_0$ is on the $n$-th layer.
	
	It remains to show that ${\mathbb P}(A_n)>0$. Using the Poisson property and \eqref{prop1 non nul}-\eqref{prop3 non nul}, we get
	\begin{align*}
	&\mathbb{P}(A_n
	)
	= 	\mathbb{P}({\mathcal P}\cap (\Pi^{\downarrow}(w_0)\setminus \cup_{w_a\in {\mathcal T}_{n-1}} B(w_a,\varepsilon))=\varnothing)\prod_{\mbox{\scalebox{.4}{$w_a\in {\mathcal T}_{n-1}\setminus\{w_0\}$}}} \mathbb{P}( \# (\mathcal{P}\cap B(w_a, \varepsilon)) =1) 
	> 0.
	\end{align*}
\end{proof}

\begin{proof}[Proof of Theorem \ref{lim variance precis}.]
	
~\\~\\\noindent{\textit{Proof of the convergence of the normalized variance.}}	
	In the same way as for the expectation asymptotics, the first idea consists in using Mecke's formula. Combining it with Fubini's theorem and writing $\xi'_{n,k,\la}(x,\cP_\la)$ for the same functional as $\xi_{n,k,\la}(x, \cP_\la)$ save for the fact that $x$ is not added to the process, we get
	\begin{align*}
		&\Var(N_{n,k,\lambda})\\&= \mathbb{E}\bigg[ \bigg(\sum_{x \in \mathcal{P_\lambda}} \xi'_{n,k, \lambda}(x, \mathcal{P}_\lambda) \bigg)^2 \bigg] - \mathbb{E}\bigg[ \sum_{x \in \mathcal{P_\lambda}} \xi'_{n,k, \lambda}(x, \mathcal{P}_\lambda) \bigg]^2\\
		&
		= \mathbb{E}\bigg[ \sum_{x \in \mathcal{P_\lambda}} \xi'_{n,k, \lambda}(x, \mathcal{P}_\lambda)^2 \bigg]
		 + \mathbb{E}\bigg[ \sum_{\substack{x,y \in \mathcal{P_\lambda}\\ x \neq y}} \xi'_{n,k, \lambda}(x, \mathcal{P}_\lambda) \xi'_{n,k, \lambda}(y, \mathcal{P}_\lambda) \bigg] 
		 - \mathbb{E}\bigg[ \sum_{x \in \mathcal{P_\lambda}} \xi'_{n,k, \lambda}(x, \mathcal{P}_\lambda) \bigg]^2
		\\
		& = I_1(\lambda) + I_2(\lambda)
	\end{align*}
	where
	$$ I_1(\lambda) := \lambda \int_{\B^d} \mathbb{E}\left[ \xi_{n,k,\lambda}(x, \mathcal{P}_{\lambda})^2 \right]\mathrm{d}x\; 
	\mbox{ and }\;
	I_2(\lambda) := \lambda^2 
	\iint_{(\B^d)^2} c_{n,k,\la}(x,y)\mathrm{d}x \mathrm{d}y.
	$$
	with $$c_{n,k,\la}(x,y)=\mathbb{E}\left[\xi_{n,k, \lambda}(x, \mathcal{P}_\lambda \cup \lbrace y \rbrace) \xi_{n,k, \lambda}(y, \mathcal{P}_\lambda \cup \lbrace x \rbrace) \right]
-  \mathbb{E}\left[ \xi_{n,k, \lambda}(x, \mathcal{P}_\lambda) \right] 
	\mathbb{E}\left[ \xi_{n,k, \lambda}(y, \mathcal{P}_\lambda) \right].$$
	We treat $I_1(\lambda)$ as in the proof of Theorem \ref{lim esperance precis} to get 
	$$\lim\limits_{\lambda \rightarrow +\infty } {\lambda^{-\frac{d-1}{d+1}}} I_1(\lambda) = I_1$$
	where we recall the definition of $I_1$ at \eqref{eq:defI1}.
	We now prove the convergence of $\la^{-\frac{d-1}{d+1}}I_2(\la)$. 
		We follow line by line the method used in \cite[calculation of II on pages 92-93]{CSY}, i.e. rewriting in spherical coordinates, then use of a change of variables provided by the scaling transformation $T^{(\la)}$, to obtain that 
$$I_2(\lambda) =  \lambda^{\frac{d-1}{d+1}} \int_{\mathbb{S}^{d-1}}  \int_{0}^{\lambda^{\frac{2}{d+1}}} \int_{T_u^{(\lambda)}(\mathbb{S}^{d-1})} 
	f^{(\lambda)}(u_0, h_0, v_1, h_1) \mathrm{d}v_1 \mathrm{d}h_1 \mathrm{d}h_0 \mathrm{d}\sigma_{d-1}(u_0) $$
	where
	$$ f^{(\lambda)}(u_0, h_0, v_1, h_1)  := (1 - \lambda^{-\frac{2}{d+1}}h_0)^{d-1}  \frac{\sin^{d-2}\left( \lambda^{-\frac{1}{d+1}} |v_1| \right)}{ |\lambda^{-\frac{1}{d+1}} v_1|^{d-2}} ( 1 - \lambda^{-\frac{2}{d+1}} h_1)^{d-1}
	 c_{n,k}^{(\lambda)}((0, h_0), (v_1, h_1)),$$
	 with $c_{n,k}^{(\la)}$ defined at \eqref{eq:defcnkla}.
	 
	It remains to apply the dominated convergence theorem.
	Using Lemma \ref{Lp bound correlation}, we obtain
	$$ |f^{(\lambda)}(u_0, h_0, v_1, h_1) | \leq c_1 h_0^{c_2} h_1^{c_3} \exp\left(-c_4 \left( \Vert v_1 \Vert^{d+1} + h_0^{\frac{d+1}{2}} + h_1^{\frac{d+1}{2}} \right) \right)$$
	which is integrable on $ \mathbb{S}^{d-1} \times \R_+ \times \mathbb{R}^{d-1} \times \R_+$.
	We deduce from Proposition \ref{correlation convergence} that
	$$\lim\limits_{\lambda \rightarrow \infty} f^{(\lambda)}(u_0, h_0, v_1, h_1) = c^{(\infty)}_{n,k}((0, h_0), (v_1, h_1)) . $$
	This implies that 
	$$ \lim\limits_{\lambda \rightarrow \infty} {\lambda^{-\frac{d-1}{d+1}}}I_2(\lambda) = I_2,$$
	where $I_2$ has been defined at \eqref{eq:defI2}.
	\\~\\
\textit{Proof of the positivity of the limiting variance.}
	This proof is essentially adapted from \cite[Section 4.5]{CY4}. The main difference is that we make sure
	that our points are on the $n$-th layer instead of being extremal.
	
	Similarly to what is done in \cite[Lemma 7.6]{CSY}, we can use the same arguments as in the proof of Theorem \ref{lim variance precis} to show that 
	\begin{equation}\label{var pos eq1}
	\lim\limits_{\lambda \rightarrow \infty} {\lambda^{-\frac{d-1}{d+1}}}{\Var[N_{n,k,\lambda}]} 
		= \lim\limits_{\lambda \rightarrow \infty} {\lambda^{-\frac{d-1}{d+1}}}{\Var[\tilde{N}_{n,k,\lambda}]}
	\end{equation}
	where we write 
	$\tilde{N}_{n,k,\lambda} := \sum_{w \in \mathcal{P} \cap W_\lambda} \xi^{(\infty)}_{n,k}(w, \mathcal{P})$.
	
	Our strategy to prove that the limit in the rhs of \eqref{var pos eq1} is positive is the following: we start by discretizing $W_\la$ and construct in each parallelepiped of that discretization two different configurations, said to be \textit{good}, which have a positive probability to occur and which give birth to two different values for the local number of $k$-faces of the $n$-th layer. Then we check that this counting is not affected by the external configuration and finally, we find a lower bound for the total variance conditional on the intersection of the point process with the outside of the parallelepipeds which contain one of the two \textit{good} configurations.
	\\~\\
	\label{page:step1 var pos}
	\noindent\textit{Step 1. Construction of a good configuration in a thin parallelepiped.}
	First, we consider a cube $Q \subseteq \mathbb{R}^{d-1}$ and we take $\rho \in (0, \infty)$ smaller than the
	diameter of $Q$. 
	We take $\delta > 0$ sufficiently small such that the paraboloids $\Pi^{\downarrow}(w)$ are pairwise disjoint for 
	$w$ belonging to the grid $(\rho \mathbb{Z}^{d-1}\cap Q) \times \{ \delta \}$. 
	For each $w \in (\rho \mathbb{Z}^{d-1} \cap Q) \times \{ \delta \}$, we make the same tree construction as in the proof of Lemma \ref{lem proba etre sur une couche} inside $\Pi^{\downarrow}(w)$. We obtain a forest that we call $\cT_{n,\rho}$. In particular, the construction ensures that all points $w\in (\rho \mathbb{Z}^{d-1} \cap Q) \times \{ \delta \}$ belong to the $n$-th layer of $\cT_{n,\rho}$.
	Let us write
	$F_{n,k}(Q, \rho, \delta)$ for the number of k-faces of the $n$-th layer of $(\cP\setminus (Q\times [0,\infty)))\cup \cT_{n,\rho}$ going through any 
	$w \in (\rho \mathbb{Z}^{d-1} \cap Q) \times \{ \delta \}$. If we ignore the points of $\mathcal{P}\setminus (Q\times [0,\infty))$, as the diameter of $Q$
	gets large compared to $\rho$, boundary effects become negligible and we get
	$$ F_{n,k}(Q, \rho/2, \delta) \sim 2^{d-1} F_{n,k}(Q, \rho, \delta).$$
	Then we consider $\varepsilon > 0$ such that
	\begin{equation}\label{eq:conddeltaeps}
	\delta+\varepsilon\le \frac{\rho^2}{8}.
	\end{equation}
	This guarantees in particular that for any 
$w=(v,\delta+\varepsilon) \mbox{ and }w'=(v',\delta+\varepsilon)\mbox{ with } |v-v'|\ge \rho$, we have $\Pi^{\downarrow}(w)\cap \Pi^{\downarrow}(w')=\varnothing$.    
	As in the proof of the positivity of the expectation, we consider a random perturbation of the points of $\cT_{n,\rho}$ by at most $\varepsilon$ for the Euclidean distance, i.e. we assume that $\cP\cap B(w,\varepsilon)$ consists of one point for each $w\in \cT_{n,\rho}$. We notice that the obtained perturbed points have a height at most equal to $\delta+\varepsilon$ and are distant by at most $2\rho$ when $\varepsilon<\rho/2$. Consequently, for $\varepsilon$ small enough, the maximal height of $(Q \times [0, \infty) ) \setminus \bigcup_{w'\in \cP\cap \cup_{w\in \cT_{n,\rho}}B(w,\varepsilon)} \Pi^\uparrow(w')$ is smaller than $\alpha$ where
	$\alpha$ is the maximal height of a point in $(Q \times [0, \infty) ) \setminus \bigcup_{w \in (Q \cap 2 \rho \mathbb{Z}^{d-1}) \times \{ \delta +\varepsilon\} } \Pi^\uparrow(w)$. In particular, we claim that there is a constant $c_\alpha>1$ depending only on dimension $d$ such that 
	\begin{equation}\label{eq:calpha}\alpha \leq c_\alpha\rho^2.
	\end{equation}
	Indeed, up to boundary effects, the set of apices of down paraboloids which contain $2^{d-1}$ points of $Q\cap 2\rho\Z^{d-1}\times \{\delta+\varepsilon\}$ is located on a translated dual grid at height $((d-1)\frac{\rho^2}{2}+\delta+\varepsilon) $, see Figure \ref{fig:lafiguredelamort}.  
	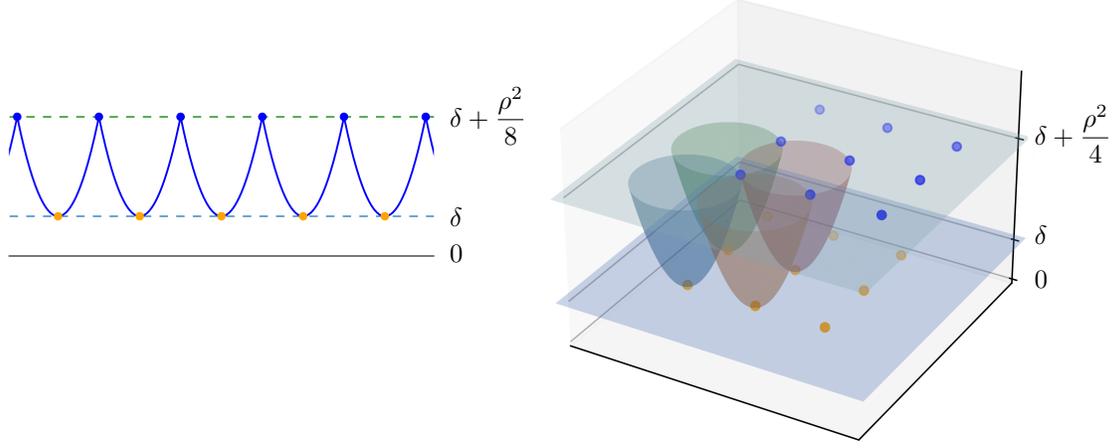
\begin{figure}
        \raisebox{1.2\height}{\begin{overpic}[width=0.42\textwidth]{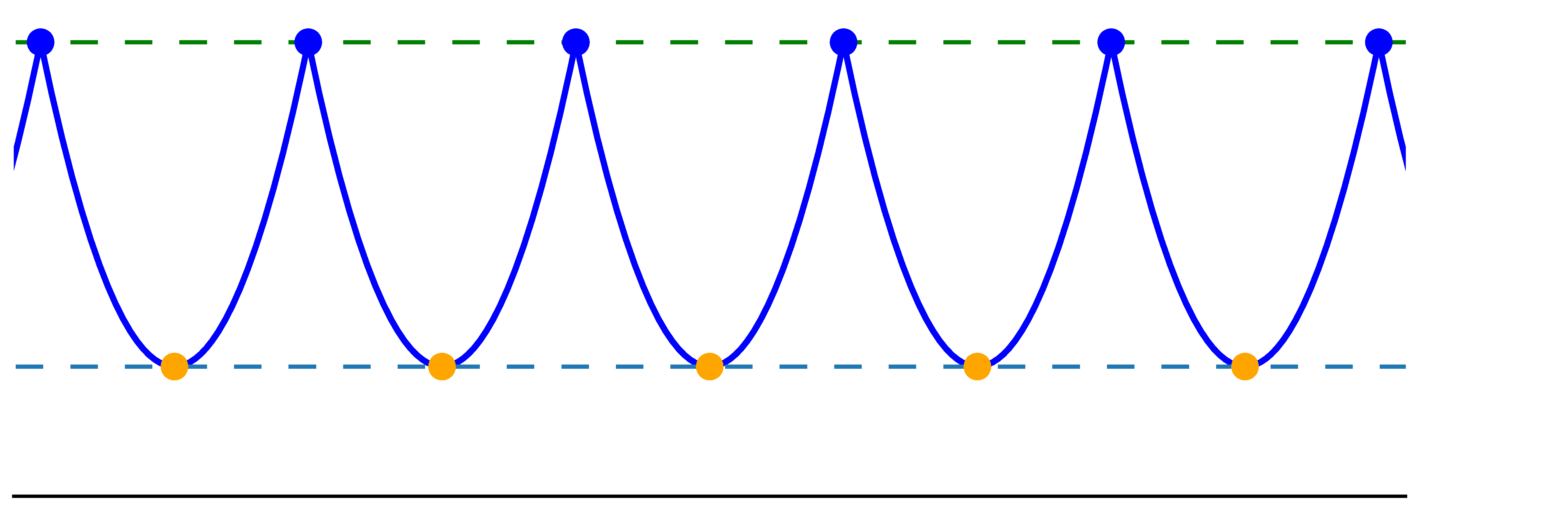}
         \put (93,0) {$\displaystyle 0$}
         \put (93,7.5) {$\displaystyle \delta$}
         \put (93,28.5) {$\displaystyle \delta + \frac{\rho^2}{8}$}
        \end{overpic}}
       \hspace{0.5cm} \begin{overpic}[ width=0.45\textwidth]{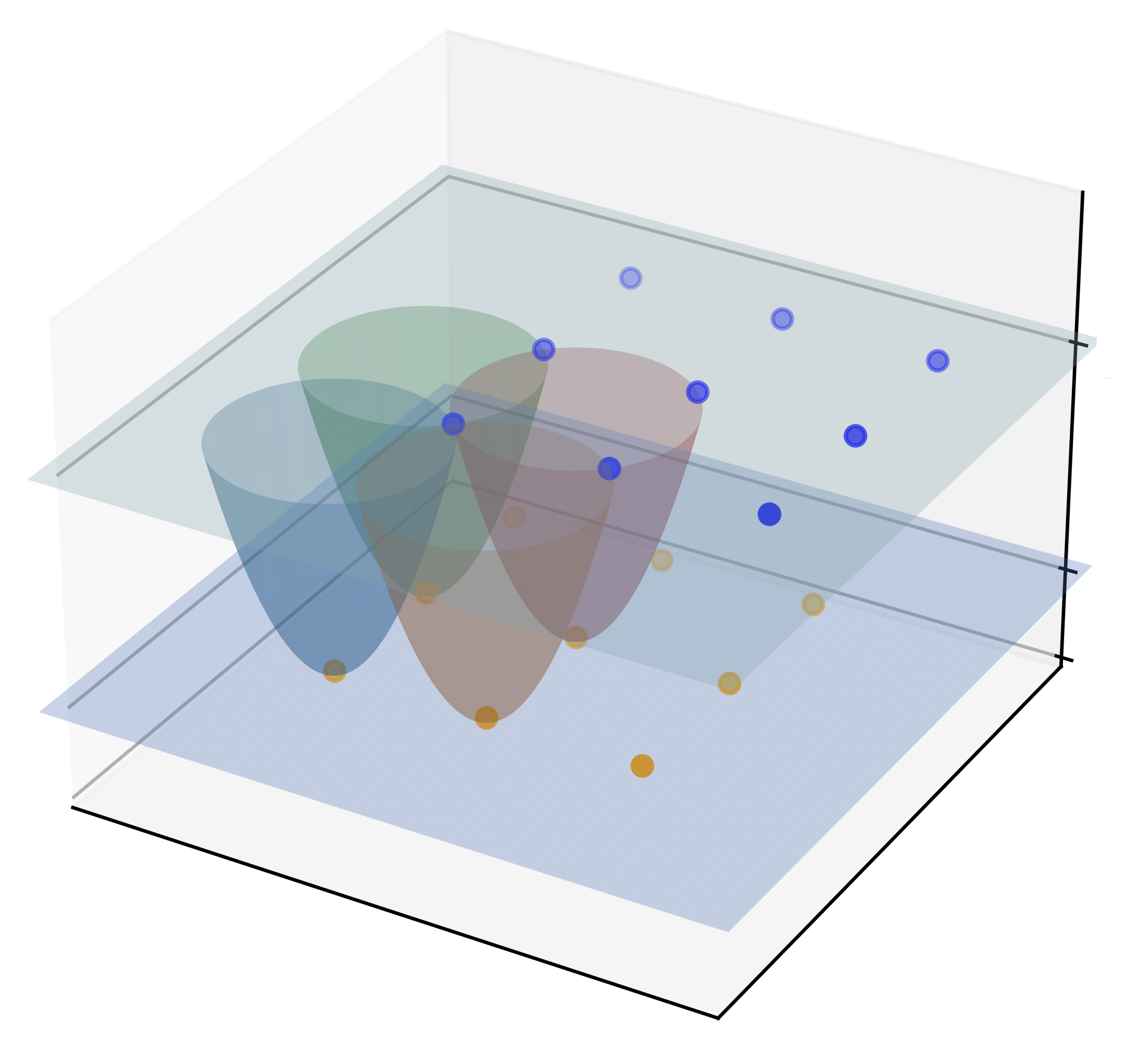}
            \put (97, 32) {$\displaystyle 0 $}
            \put (97, 40.5) {$\displaystyle \delta $}
            \put (97, 60) {$\displaystyle \delta + \frac{\rho^2}{4}$}
        \end{overpic}
        \caption{A part of the dual grid for non-perturbed points in dimensions $2$ and $3$.}
        \label{fig:lafiguredelamort}
    \end{figure}
	
	Consequently, let us consider the event
	$$A_{n,\rho}=\cap_{w\in {\cT_{n,\rho}}}\{{\#(\mathcal P}\cap B(w,\varepsilon))=1\}\cap \{({\mathcal P}\cap [(Q\times [0,\alpha])\setminus \cup_{w\in \cT_{n,\rho}}B(w,\varepsilon)]=\varnothing\}.$$
Let us write $F_{n,k}(Q, \rho, \delta, \varepsilon)$
	for the total number of $k$-faces of the $n$-th layer of $\cP$ going through at least one point in $\cP\cap \cup_{w\in (\rho \mathbb{Z}^{d-1} \cap Q) \times \{ \delta \}}B(w,\varepsilon)$. Conditional on $A_{n,\rho}$ and when the points of $\cP\setminus (Q\times [0,\alpha])$ are ignored, for $\varepsilon$ small enough, this quantity is in fact equal to $F_{n,k}(Q, \rho, \delta)$ and we keep the relation
	\begin{equation}\label{var pos scaling} F_{n,k}(Q, \rho/2, \delta, \varepsilon) \sim 2^{d-1} F_{n,k}(Q, \rho, \delta, \varepsilon).\end{equation}
	
	\noindent\textit{Step 2. Influence of the points outside of the thin parallelepiped $Q\times [0,\alpha]$.}
		For any closed set 
	$C \subseteq \mathbb{R}^{d-1}$, we define for any $\gamma > 0$ $C^{(\gamma)} := \{ x \in C : d(x, \partial C) > \gamma \}$. We claim that on $A_{n,\rho}$, for any $w \in \cP\cap (Q^{(\rho)}\times [0,\alpha])$, the status of $w$ does not depend on points outside $Q\times[0,\alpha]$, i.e.
	$$\ell^{(\infty)}(w,\cP\cap (Q\times[0,\alpha]))=\ell^{(\infty)}(w,\cP).$$
	This is due to the fact that the condition \eqref{eq:conddeltaeps} guarantees that the down paraboloids with apices at points from $\cP\cap (Q^{(\rho)}\times [0,\alpha])$ are included in $Q\times[0,\alpha]$. 
	Moreover, for $\textbf{c}=\sqrt{2c_\alpha}+2$, we assert that the facial structure around any point inside $\cP\cap (Q^{(\textbf{c}\rho)}\times [0,\alpha])$ which belongs to the $n$-th layer of the peeling of $\cP$ does not depend on points outside $Q\times[0,\alpha]$. Indeed, let us consider $w\in \cP\cap (Q^{(\textbf{c}\rho)}\times [0,\alpha])$. We choose $(d-1)$ points from $\cP\cap(Q^{(\sqrt{2c_\alpha}\rho)}\times [0,\alpha])$ which share a common facet of the $n$-th layer of the peeling of $\cP\cap (Q\times[0,\alpha])$. The down paraboloid which contains this facet has an apex in $Q^{(\sqrt{2 c_\alpha}\rho)}\times[0,\alpha]$. Consequently, thanks to \eqref{eq:calpha}, that down paraboloid is included in $Q\times[0,\alpha]$, which implies that the facet containing $w$ and the $(d-1)$ other points is a facet of the $n$-th layer of the peeling of $\cP$.

	~\\
	\textit{Step 3. Discretization of $W_\la$ and lower bound for the variance.}
		We are now ready to discretize $W_\la$ and isolate the \textit{good} parallelepipeds from the discretization, according to the two previous steps. 
	We choose $\delta$  and $\varepsilon$ which satisfy \eqref{eq:conddeltaeps} with the choice $\rho=1$. 
	We take a large positive number $M$ and we partition $\left[ -\frac{\lambda^{\frac{1}{d+1}}}{2}, \frac{\lambda^{\frac{1}{d+1}}}{2}\right]^{d-1}$ into $L:= \left[\frac{\lambda^{\frac{1}{d+1}}}{M}\right]^{d-1}$ cubes 	$Q_1, \ldots, Q_L$. We consider the cubes $Q_i$ satisfying the following properties:
	
	(a) For each $z \in (\mathbb{Z}^{d-1} \cap (Q_i \setminus Q_i^{(\textbf{c})})) \times  \{ \delta \}$,  $\cP \cap B(z, \varepsilon) $ is a singleton and is put on the $n$-th layer using the  tree construction associated with $\cT_{n,1}$ and the perturbation of each point by at most $\varepsilon$ as in Step $1$.
	
	(b) One of these two conditions holds:
	\begin{enumerate}
		\item For each $z \in (\mathbb{Z}^{d-1} \cap Q_i^{(\textbf{c})}) \times  \{ \delta \}$, $\cP \cap B(z, \varepsilon)$ is a singleton and this point is put on the $n$-th layer as in property (a).
		\item For each $z \in (\frac{1}{2}\mathbb{Z}^{d-1} \cap Q_i^{(\textbf{c})}) \times  \{ \delta \}$, $\cP \cap B(z, \varepsilon)$ is a singleton and this point is put on the $n$-th layer as before.
	\end{enumerate}

	(c) Aside from the points described above, $\mathcal{P}$ has no point in 
	$Q_i \times [0, \alpha]$.
	
	After a possible relabeling, let $I := \{1, \ldots K \} $ be the indices of cubes partitioning $\left[ -\frac{\lambda^{\frac{1}{d+1}}}{2}, \frac{\lambda^{\frac{1}{d+1}}}{2}\right]^{d-1}$ that verify properties (a) to (c).
	Since any cube has a positive probability to verify these properties, we have 
	\begin{equation} \label{var pos number Qi}
		\mathbb{E}[K] \geq c \lambda^{\frac{d-1}{d+1}} .
	\end{equation}
	Let $\mathcal{F}_\lambda$ be the $\sigma$-algebra
	generated by $I$, the positions of points in $W_\lambda \setminus (\bigcup_{i \in I} Q_i^{(\textbf{c})}  \times [0, \alpha])$ and the scores $\xi_{n,k}^{(\infty)}(x,\cP)$ at these points.
	For each $i \in I$ we claim that
	\begin{equation}\label{var pos Qi}
		\Var\bigg[ \sum_{x \in \mathcal{P} \cap (Q_i \times [0, \alpha])} \xi^{(\infty)}_{n,k}(x, \mathcal{P}) \Big| \mathcal{F}_\lambda \bigg]=\Var\bigg[ \sum_{x \in \mathcal{P} \cap (Q_i^{(3)} \times [0, \alpha])} \xi^{(\infty)}_{n,k}(x, \mathcal{P}) \Big| \mathcal{F}_\lambda \bigg] \geq c_0.
	\end{equation}
	Indeed,
	either condition (b1) or condition (b2) occurs in $Q_i^{(\textbf{c})} \times [0, \alpha]$, each with positive probability. Moreover, we observe that $\sum_{x \in \mathcal{P} \cap (Q_i \times [0, \alpha])} \xi^{(\infty)}_{n,k}(x, \mathcal{P})$ is larger when (b2) is satisfied. This is due to the scaling result \eqref{var pos scaling} which implies that the contribution of points inside $Q_i^{(\textbf{c})}\times [0, \alpha]$ provides a quantity almost $2^{d-1}$ times larger when (b2) is satisfied. 
	
	Since only scores in $\cup_{i \in I} Q_i \times [0, \alpha]$ have any variability conditional on $\mathcal{F}_\lambda$ we get 
	\begin{align*}
		\Var[\tilde{N}_{n,k,\lambda}] & =\Var[\E[\tilde{N}_{n,k,\la}|\cF_\la]+\E[\Var[\tilde{N}_{n,k,\la}|\cF_\la]]\\&\geq \mathbb{E}\left[ \Var\left[\tilde{N}_{n,k,\lambda} | \mathcal{F}_\lambda  \right] \right]\\
			&= \mathbb{E}\bigg[ \Var\bigg[\sum_{i \in I} \sum_{x \in \mathcal{P} \cap (Q_i^{(3)} \times [0, \alpha])} \xi^{(\infty)}_{n,k}(x, \mathcal{P}) \Big| \mathcal{F}_\lambda \bigg] \bigg].
	\end{align*}
	Then the sums of scores in $Q_i^{(\textbf{c})} \times [0, \alpha]$ and $Q_j^{(\textbf{c})} \times [0, \alpha]$ for $i \neq j$ are independent conditional on $\mathcal{F}_\lambda$ since the scores in $Q_i^{(\textbf{c})} \times [0, \alpha]$ only depend on points in $Q_i \times [0, \alpha]$.
	Thus we can write
	\begin{align*}
	\Var[\tilde{N}_{n,k,\lambda}] &\geq \mathbb{E}\bigg[ \sum_{i \in I} \Var\bigg[ \sum_{x \in \mathcal{P} \cap (Q_i \times [0, \alpha])} \xi^{(\infty)}_{n,k}(x, \mathcal{P}) \Big| \mathcal{F}_\lambda \bigg] \bigg]\\
	&\geq c_0 \mathbb{E}[K]\\
	&\geq c \lambda^{\frac{d-1}{d+1}}.
	\end{align*}
	where the second inequality comes from \eqref{var pos Qi} applied to each $Q_i$ and the last inequality comes from
	\eqref{var pos number Qi}. Thanks to \eqref{var pos eq1}, this proves the positivity of the limiting variance.
    \end{proof}
    
    \subsection{Results on intrinsic volumes}\label{sec:intrinsic}
    We define the score and two-point correlation function in the case of intrinsic volumes by following closely the method of \cite[pp. 54--55]{CSY}, which relies on Kubota's formula, see \cite[equation (6.11)]{SW08}.
    For convenience, let us write $\conv_{n,\la} := \conv_n(\cP_\la)$ and let us denote by $\kappa_m$ the volume of the $m$-dimensional unit ball. By Kubota's formula applied to $\conv_{n,\la}$, we get
    \begin{equation}\label{eq:VkKubota}
      V_k(\conv_{n,\la}) = \frac{d! \kappa_d}{k!\kappa_k(d-k)!\kappa_{d-k}} \int_{G(d,k)} \text{Vol}_k(\conv_{n,\la} | L) \mathrm{d}\nu_k(L)
    \end{equation}
    where $\nu_k$ is the normalized Haar measure on the $k$-th Grassmanian $G(d,k)$ of $\R^d$ and $\conv_{n,\la} | L$ is the orthogonal projection of $\conv_{n,\la}$ onto $L$.
    For every $x \in \R^d \setminus \{0\}$ we define 
    $\vartheta_L(x, \conv_{n,\la}) := \mathds{1}_{\{x \not\in \conv_{n,\la}|L \}}$ and the projection avoidance functionals
    $$\vartheta_k(x, \conv_{n,\la}) := \int_{G(\text{lin}[x],k) } \vartheta_L(x, \conv_{n,\la}) \mathrm{d}\nu_k^{\text{lin}[x]}(L)$$
    where $\text{lin}[x]$ is the linear space spanned by $x$, $G(\text{lin}[x],k)$ is the set of all $k$-dimensional linear subspaces of $\R^d$ containing $\text{lin}[x]$ and $\nu_k^{\text{lin}[x]}$ is the normalized Haar measure on $G(\text{lin}[x],k)$. Combining \eqref{eq:VkKubota} and Fubini's theorem we can rewrite the defect $k$-th intrinsic volume of $\conv_{n,\la}$ as
    \begin{equation*}
    V_{n,k,\la} = \frac{{d-1 \choose k-1}}{\kappa_{d-k}} \int_{\R^d} \vartheta_k(x, \conv_{n,\la}) \frac{\mathrm{d}x}{|x|^{d-k}}.
    \end{equation*}
    The functional $V_{n,k,\la}$ can then be written as a sum of scores, i.e.
    $$V_{n,k,\la} = \sum_{x \in \cP_\la} \xi_{V,n,k}(x, \cP_\la)$$
    with
    $$\xi_{V,n,k}(x, \cP_\la) := \left\{\begin{array}{ll}\frac{{d-1 \choose k-1}}{d\kappa_{d-k}} \int_{\text{cone}_n(x, \cP_\la)} \vartheta_k(y, \conv_{n}(\cP_\la \cup \{ x \})) \frac{\mathrm{d}y}{|y|^{d-k}}&\mbox{ if $x \in \partial \conv_{n,\la}$}\\0 &\mbox{ otherwise}  \end{array}\right.$$
    where $\text{cone}_n(x, \cP_\la) := \{ry : r > 0 \text{ and } y \in \cF_{n,d-1}(x, \cP_\la) \}$.
    We denote by 
    \begin{equation}\label{eq:defxiVnk}
    \xi_{V,n,k}^{(\la)}(x, \cP^{(\la)}) := \la\xi_{V,k,\la}([T^{(\la)}]^{-1}(x), \cP_\la)
    \end{equation}
    their rescaled counterparts, see \cite[page 84]{CSY} for an explanation on the factor $\la$. 
    Let us explain how to define the limit versions of the rescaled scores. For a point $w = (v,h) \in \R^{d-1} \times \R_+$ we write $w^{\updownarrow}$ for the set
    $\{v\}\times \R_+$. We denote by $A(w^\updownarrow,k)$ the set of all $k$-dimensional affine spaces in $\R^{d-1}\times\R_+$ containing $w^\updownarrow$. For any $k$ and any affine space $L \in A(w^\updownarrow,k)$  we define the orthogonal parabolic volume
    $$\Pi^\perp(w, L) :=  (w \oplus L^\perp) \cap \Pi^\downarrow(w) .$$
    We put $\vartheta_{n,L}^{(\infty)}(w) = 1$ if $\Pi^\perp(w,L) \cap \Phi_n(\cP) = \varnothing$ and $0$ otherwise. Then we define 
    $$\vartheta_{n,k}^{(\infty)}(w) := \int_{A(w^\updownarrow,k)} \vartheta_{n,L}^{(\infty)}(w) \mathrm{d}\mu_k^{w\updownarrow}(L) $$
    where $\mu_k^{w \updownarrow}$ is the normalized Haar measure on $A(w^\updownarrow,k)$. We are finally able to define the limit score
    \begin{equation}\label{eq:defxiinftyVnk}
    \xi^{(\infty)}_{V,n,k}(w,\cP) := \frac{{d-1 \choose k-1}}{d\kappa_{d-k}} \int_{\text{v-cone}(\cF_{n,d-1}^{(\infty)}(w, \cP) )} \vartheta_{n,k}^{(\infty)}(w') \mathrm{d}w'
    \end{equation}
    where $\text{v-cone}(\cF_{n,d-1}^{(\infty)}(w, \cP) ) := \{ (v', h'), \exists h''\ : (v',h'') \in \cF^{(\infty)}_{n,d-1}(w) \}$. 
    
    For any $\la \in (0, \infty]$, the corresponding two-point correlation function is then defined by the identity 
    \begin{align}\label{eq:defclaVnk}
&c_{V,n,k}^{(\lambda)}((0, h_0), (v_1, h_1))\notag\\&:=
\E[\xi_{V,n,k}^{(\la)}((0,h_0),\cP^{(\la)}\cup\{(v_1,h_1)\})\xi_{V,n,k}^{(\la)}((v_1,h_1),\cP^{(\la)}\cup\{(0,h_0)\})]
\notag\\&\hspace{3.6cm}-\E[\xi_{V,n,k}^{(\la)}((0,h_0),\cP^{(\la)})]\E[\xi_{Vn,k}^{(\la)}((v_1,h_1),\cP^{(\la)})].   
    \end{align}
    
    \begin{proof}[Proof of Theorems \ref{lim esperance precis vol} and \ref{lim variance precis vol}]

    ~\\~\\{\textit{Proof of the convergence of the normalized expectation and variance.}}
    The proofs of Theorems \ref{lim esperance precis vol} and \ref{lim variance precis vol} go along the same lines as the proofs of Theorems \ref{lim esperance precis} and \ref{lim variance precis} : after application of Mecke's formula and a suitable change of variables in the integral, we need to apply Lebesgue's dominated convergence theorem which requires the convergence of
    $\E[\xi_{V,n,k}^{(\la)}(w,\cP^{(\la)})], \E[(\xi_{V,n,k}^{(\la)}(w, \cP))^2] $ and $c^{(\la)}_{V,n,k}(w,w')$ in the same spirit as in Propositions \ref{expectation convergence xi k-faces} and \ref{correlation convergence} as well as moment bounds similar to those in Lemma \ref{borne Lp}. All these results rely on stabilization results identical to the tail estimates in Propositions \ref{stab k-faces} and \ref{localisation hauteur k-faces}. 
    
    As an example, we explain how to adapt Proposition \ref{stab k-faces}, Lemma \ref{borne Lp} and Proposition \ref{expectation convergence xi k-faces} to get the convergence of the expectation of $\xi^{(\la)}_{V,n,k}$. 
    
    We claim that as soon as $H_n^{(\la)}((0,h_0),\cP^{(\la)})\le \frac{r^2}{32}$, the calculation of $\xi^{(\la)}_{V,n,k}$ only depends on the set $\mathcal{U}$ defined at \eqref{eq:defEnsembleU}. Consequently, the radius of stabilization for $\xi^{(\la)}_{V,n,k}((0,h), \cP^{(\la)})$ is the same as the one considered in Proposition \ref{stab k-faces}.
    

    To prove the moment bounds for $\xi^{(\la)}_{V,n,k}$,
    we notice that $\xi_{V,n,k}^{(\la)}(0,h)$ is smaller than the volume of $C^{\leq H}(R)$ where $R$ is the stabilization radius of $\xi^{(\la)}_{V,n,k}((0,h), \cP^{(\la)})$ and $H=H_n^{(\la)}((0,h),\cP^{(\la)})$. Using Lemmas \ref{stab hauteur} and \ref{stab k-faces} as in the proof of Lemma \ref{borne Lp}, we deduce that for some $c_1,c_2>0$, 
\begin{equation*}
	\mathbb{E}\left[ (\xi_{V,n,k}^{(\la)}((0,h), \mathcal{P^{(\lambda)}})) ^p \right] \leq c_1 h^{p} \exp(- c_2h^{\frac{d+1}{2}}).
	\end{equation*}        
    
    These two ingredients allow us to prove the convergence of $\E[\xi_{V,n,k}^{(\la)}(w,\cP^{(\la)})]$ with the help of the continuous mapping theorem as in Lemmas \ref{lem:convesp1} and \ref{expectation convergence xi k-faces}. 
    
    ~\\
    \noindent\textit{Proof of the positivity of the limiting constants for the intrinsic volumes.} 
    The defect intrinsic volumes are increasing with respect to set inclusion. As the limiting constant is positive for $\conv_1(\cP^{(\la)})$ \cite[Theorem 3]{B89}, it remains true for $\conv_n(\cP_\la).$
    
    It remains to prove that the limiting constant for the variance is positive.
    The strategy is the same as the one that we used in the case of the $k$-faces.
    Because of the renormalization by $\la$ in the definition of $\xi_{V,n,k}^{(\la)}$, see \eqref{eq:defxiVnk} , the equality \eqref{var pos eq1} in the case of the volume scores becomes
    \begin{equation*}
	\lim\limits_{\lambda \rightarrow \infty} \la^{\frac{d+3}{d+1}}\Var[V_{n,k,\lambda}]
		= \lim\limits_{\lambda \rightarrow \infty} {\lambda^{-\frac{d-1}{d+1}}}{\Var[\tilde{V}_{n,k,\lambda}]},
	\end{equation*}
	where $\tilde{V}_{n,k,\la}:= \sum_{w \in \mathcal{P} \cap W_\lambda} \xi^{(\infty)}_{V,n,k}(w, \mathcal{P})$.
    The paragraph right after \eqref{var pos eq1} details the rest of our strategy.  We repeat word for word the first two steps of the proof of the positivity of the limiting variance on page \pageref{page:step1 var pos}, save for all the statements regarding $F_{n,k}$. 
    In the sequel we use the notations of the aforementioned proof and go directly to Step 3, i.e. the choice of the \textit{good} configurations.
    
    We discretize $W_\la$ by decomposing it into parallelepipeds and isolate the \textit{good} parallelepipeds from the discretization as we have done for the $k$-faces. 
	We choose $\delta$  and $\varepsilon$ which satisfy \eqref{eq:conddeltaeps} with the choice $\rho=1$. 
	We take a large positive number $M$ and we partition $\left[ -\frac{\lambda^{\frac{1}{d+1}}}{2}, \frac{\lambda^{\frac{1}{d+1}}}{2}\right]^{d-1}$ into $L:= \left[\frac{\lambda^{\frac{1}{d+1}}}{M}\right]^{d-1}$ cubes 	$Q_1, \ldots, Q_L$. We consider the cubes $Q_i$ satisfying the following properties, note that the main difference is the content of conditions b)1) and b)2):
	
	(a) For each $z \in (\mathbb{Z}^{d-1} \cap (Q_i \setminus Q_i^{(\textbf{c})})) \times  \{ \delta \}$,  $\cP \cap B(z, \varepsilon) $ is a singleton and is put on the $n$-th layer using the  tree construction associated with $\cT_{n,1}$ and the perturbation of each point by at most $\varepsilon$ as in Step $1$.
	
	(b) One of these two conditions holds:
	\begin{enumerate}
		\item For each $z \in (\mathbb{Z}^{d-1} \cap Q_i^{(\textbf{c})}) \times  \{ \delta \}$, $\cP \cap B(z, \varepsilon)$ is a singleton and this point is put on the $n$-th layer as in property (a).
		\item For each $z \in (\mathbb{Z}^{d-1} \cap Q_i^{(\textbf{c})}) \times  \{ \delta / 2 \}$, $\cP \cap B(z, \varepsilon)$ is a singleton and this point is put on the $n$-th layer as before.
	\end{enumerate}

	(c) Aside from the points described above, $\mathcal{P}$ has no point in 
	$Q_i \times [0, \alpha]$.
	
	After a possible relabeling, let $I := \{1, \ldots K \} $ be the indices of cubes partitioning $\left[ -\frac{\lambda^{\frac{1}{d+1}}}{2}, \frac{\lambda^{\frac{1}{d+1}}}{2}\right]^{d-1}$ that verify properties (a) to (c) and it remains true that
	\begin{equation} 
		\mathbb{E}[K] \geq c \lambda^{\frac{d-1}{d+1}} .
	\end{equation}
	Let $\mathcal{F}_\lambda$ be the $\sigma$-algebra
	generated by $I$, the positions of points in $W_\lambda \setminus (\bigcup_{i \in I} Q_i^{(\textbf{c})}  \times [0, \alpha])$ and the scores $\xi_{V,n,k}^{(\infty)}(x,\cP)$ at these points.
	For each $i \in I$ we claim that
	\begin{equation}\label{var pos Qi intr}
		\Var\bigg[ \sum_{x \in \mathcal{P} \cap (Q_i \times [0, \alpha])} \xi^{(\infty)}_{V,n,k}(x, \mathcal{P}) \Big| \mathcal{F}_\lambda \bigg]=\Var\bigg[ \sum_{x \in \mathcal{P} \cap (Q_i^{(3)} \times [0, \alpha])} \xi^{(\infty)}_{V,n,k}(x, \mathcal{P}) \Big| \mathcal{F}_\lambda \bigg] \geq c_0.
	\end{equation}
	
	Again,
	either condition (b1) or condition (b2) occurs in $Q_i^{(\textbf{c})} \times [0, \alpha]$, each with positive probability.
	To ensure that the remaining part of the proof for the $k$-faces works here as well, it remains to prove that $\sum_{x \in \mathcal{P} \cap (Q_i \times [0, \alpha])} \xi^{(\infty)}_{V,n,k}(x, \mathcal{P})$ is larger when (b1) is satisfied. We do it when $k=d$, i.e. in the case of the volume, as it contains all the ingredients needed to prove it for any intrinsic volume but with slightly less technicality that would only obfuscate the ideas.
    
    If each point $z$ were deterministic in $(\mathbb{Z}^{d-1} \cap Q_i^{(\textbf{c})}) \times  \{ \delta \}$ or $(\mathbb{Z}^{d-1} \cap Q_i^{(\textbf{c})}) \times  \{ \delta / 2 \}$ then the $n$-th layer in case (b2) would be a copy of the $n$-th layer in case (b1) at a lower height. The difference of volume between the two cases (b1) and (b2) would thus be a constant times the difference of height, i.e. $c \delta$. In the general case when the points are random, since we can choose $\varepsilon > 0$ as small as we want, we can make the difference of volume  as close as needed to this deterministic situation. Thus we can make sure that there exists a constant $c> 0$ such that $\sum_{x \in \mathcal{P} \cap (Q_i \times [0, \alpha])} \xi^{(\infty)}_{n,k}(x, \mathcal{P})$ in situation (b2) is smaller than $c$ times $\sum_{x \in \mathcal{P} \cap (Q_i \times [0, \alpha])} \xi^{(\infty)}_{n,k}(x, \mathcal{P})$ in situation (b1). Thus we have proved \eqref{var pos Qi intr}.
    
    The end of the proof follows along the same lines as in the case of the $k$-faces.
    \end{proof}
    
    \section{Concluding remarks} \label{sec:conclusion}
    
    
    
    
    
    Let us conclude with a list of possible extensions of our results and related open problems.
    \begin{itemize}
        \item  \textit{Other intensity measures}. In \cite{CY2}, the convex hull of a Poisson point process with a Gaussian intensity measure is studied. While the intensity of the process and therefore the rescaling are different, the techniques that are used are very similar so these results should extend to the $n$-th layer as we did in this paper for the uniform measure in the unit ball. We also expect that our results extend to a stationary Poisson point process in any smooth convex body as in \cite{CY1} where the case of the first layer is treated.  The $n$-th layer should have the same behaviour as the first one in this case as well. However, \cite{CY1} uses a sandwiching result stating that with high probability the first layer lies between two floating bodies, see \cite{R05b}. Such a result has not been proved for the $n$-th layer and would be interesting on its own.
        
        The same kind of problem arises in \cite{CY3} about the convex hull of the peeling of a uniform point set in a polytope. Again, an argument on the sandwiching of the first layer is needed, see \cite{BR10b}. An extension to the $n$-th layer inside a polytope will be the subject of a future work.
        
        Finally, in \cite{BFGHR20} asymptotic results on the expected number of vertices of the convex hull of an i.i.d. sample of an arbitrary non atomic probability distribution on $\R^d$ are derived from an approximation of the convex hull with floating bodies. We may hope that an approximation of the same kind should be possible for the $n$-th layer.
        \item {\it Invariance principles.} Let us consider the
        processes defined as the integrated versions of the defect support function and radius support function of each of the first layers. It is proved in \cite[Theorem 8.3]{CSY} that in the case of the first layer, such processes when properly rescaled and centered converge to a Brownian
        sheet process. We expect to get analogous functional central limit theorems for subsequent layers as we think that the arguments, for both the convergence of the finite-dimensional distributions and the tightness, should translate well to the case of the $n$-th layer. This would certainly be more challenging to do so when the number of the layer depends on $\lambda$, see below.
        \item \textit{Depoissonization.} We left aside the case where we have a fixed number of i.i.d. random points uniformly distributed in the unit ball, i.e. a binomial point process instead of a Poisson point process. We expect a result of depoissonization in the spirit of \cite[Theorem 1.2]{CY1} and \cite[Theorem 1.1]{CY4} to occur.
        
        \item \textit{Optimal Berry-Esseen bounds.} Recently, a method to derive central limit theorems for stabilizing functionals has been derived in \cite{LSY19}. The case of the number of $k$-faces of a convex hull of Poisson point processes in a smooth convex body is given as an application with an improved rate of convergence, i.e. $O\left(\la^{-\frac{d-1}{2(d+1)}} \right)$ instead of $O\left(\la^{-\frac{d-1}{2(d+1)}} (\log \la)^{3d + 1} \right)$ as in Theorem \ref{thm:principalintro}. We conjecture that the same rate of convergence should hold for the $n$-th layer but extending the method from \cite[Theorem 3.5]{LSY19} seems somehow challenging.
        \item \textit{Monotonicity.} A monotonicity problem arises naturally from Theorem \ref{thm:principalintro}. Denoting by $C_{n,k,d}$ the limit obtained for the expectation in Theorem \ref{thm:principalintro}, we might 
         wonder how the constants $C_{n,k,d}$ evolve with $n$. For $k=0$, we expect that the sequence $(C_{n,0,d})_n$ decreases with $n$. This is what our simulations suggest. In general, monotonicity problems in random polytopes are difficult, some insightful results on the monotonicity of the number of $k$-faces of the convex hull when the number of points increases can be found in \cite{DGGMR13}, \cite{BR17} and \cite{BG3TW17}.
        \item \textit{Other regimes.} Until now we have only considered a fixed layer number $n$ that does not depend $\la$ which means that we have only studied the first layers. It would be interesting to study different regimes, where $n$ would vary with $\la$. Thanks to \cite{D04} and \cite{CS20} we know that the expected number of layers is equivalent to $\la^{\frac{2}{d+1}}$ up to a constant. Thus a natural regime to study would be the case where $n = c \la^{\frac{2}{d+1}}$. Calder and Smart conjecture in \cite{CS20} that the mean number of points in this regime is still equivalent to $\la^{\frac{d-1}{d+1}}$ up to an explicit constant. They provide in particular a heuristic argument and simulations that we were able to reproduce.
    \end{itemize}
    ~\\
    \noindent{\em Acknowledgments}. This work was partially supported by the French ANR grant ASPAG (ANR-17-CE40-0017), the Institut Universitaire de France and the French research group GeoSto (CNRS-RT3477). The authors would also like to thank J.~E.~Yukich for fruitful discussions and helpful comments.
\bibliography{biblio.bib}
~\\
{Pierre Calka, Gauthier Quilan, Laboratoire de Math\'ematiques Rapha\"el Salem, Universit\'e de Rouen Normandie, Avenue de l'Universit\'e, BP 12, Technop\^ole du Madrillet, F76801 Saint-\'Etienne-du-Rouvray France;\quad\quad pierre.calka@univ-rouen.fr, gauthier.quilan@univ-rouen.fr} 
\end{document}